\documentclass[12pt,a4paper]{amsart}

\usepackage[includeheadfoot,
            marginratio={1:1,2:3}, 
            width=412pt, 
            height=688pt,]{geometry}
\usepackage{mathscinet}

\usepackage{amsmath}
\usepackage{amsfonts}
\usepackage{amssymb}
\usepackage{graphicx}
\usepackage{ulem}
\usepackage{amsthm}
\usepackage{tikz}
\usetikzlibrary{decorations.markings}
\usetikzlibrary{shapes}
\usetikzlibrary{decorations.pathmorphing}
\usepackage{mathtools}
\usetikzlibrary{positioning}
\usetikzlibrary{cd}
\usetikzlibrary{calc}
\usetikzlibrary{backgrounds}
\usepackage{fp}
\usepackage{ifthen}

\usepackage{array}
\usepackage{float}
\usepackage{xcolor}
\usepackage{enumitem}
\usepackage{BOONDOX-calo}
\usepackage{caption}
\usepackage{txfonts}
\usepackage{bbm}
\usepackage{caption}
\usepackage{framed}

\makeatletter
\renewcommand\section{\@startsection{section}{1}%
  \z@{.7\linespacing\@plus\linespacing}{.5\linespacing}%
 {\normalfont\bfseries\centering}}
\makeatother
\tikzset{%
    symbol/.style={%
        draw=none,
        every to/.append style={%
            edge node={node [sloped, allow upside down, auto=false]{$#1$}}}
    }
}

\tikzset{->-/.style={decoration={
  markings,
  mark=at position .5 with {\arrow{>}}},postaction={decorate}}}
  
\tikzset{mid/.style 2 args={
        decoration={markings,
            mark= at position #2 with {\arrow{{#1}[scale=1.5]}} ,
        },
        postaction={decorate}
    },
mid/.default={>}{0.5}
}
\definecolor{jade}{RGB}{0,168,107}
\def\centerarc[#1](#2)(#3:#4:#5){\draw[#1] ($(#2) + ({#5*cos(#3)},{#5*sin(#3)})$) arc (#3:#4:#5);}

\def\SurfScale{0.5}
\def\StringScale{0.95}

\newcommand{\eq}[1]{\begin{equation}
                     \begin{aligned} #1 \end{aligned}
                     \end{equation}}

\newcommand{\ov}[1]{\overline{#1}}

\DeclareMathAlphabet{\mathpzc}{OT1}{pzc}{m}{it}

\newcommand{\Kbb}{\mathbb{K}}

\newcommand{\Zbb}{\mathbb{Z}}

\newcommand{\onebb}{\mathbbm{1}}

\newcommand{\Autrm}{\mathrm{Aut}}

\newcommand{\Endrm}{\mathrm{End}}

\newcommand{\VGraphrm}{\mathrm{VGraph}}
\newcommand{\Idrm}{\mathrm{Id}}
\newcommand{\Imrm}{\mathrm{Im}}

\newcommand{\KSNrm}{\mathrm{KSN}}
\newcommand{\NGraphrm}{\mathrm{NGraph}}

\newcommand{\SNrm}{\mathrm{SN}}

\newcommand{\coevrm}{\mathrm{coev}}

\newcommand{\evrm}{\mathrm{ev}}

\newcommand{\orrm}{\mathrm{or}}

\newcommand{\trrm}{\mathrm{tr}}

\newcommand{\Ccal}{\mathcal{C}}
\newcommand{\Dcal}{\mathcal{D}}
\newcommand{\Ecal}{\mathcal{E}}
\newcommand{\Gcal}{\mathcal{G}}

\newcommand{\Pcal}{\mathcal{P}}

\newcommand{\Scal}{\mathcal{S}}
\newcommand{\Tcal}{\mathcal{T}}

\newcommand{\cbf}{\mathbf{c}}
\newcommand{\dbf}{\mathbf{d}}
\newcommand{\ebf}{\mathbf{e}}
\newcommand{\fbf}{\mathbf{f}}

\newcommand{\mbf}{\mathbf{m}}
\newcommand{\nbf}{\mathbf{n}}

\newcommand{\Bbf}{\mathbf{B}}

\newcommand{\Zbf}{\mathbf{Z}}

\newcommand{\Tsf}{\mathsf{T}}
\newcommand{\Zsf}{\mathsf{Z}}

\newcommand{\BiModpzc}{\mathpzc{BiMod}}

\newcommand{\Cylpzc}{\mathpzc{Cyl}}

\newcommand{\GBordpzc}{G\mathpzc{Bord}}
\newcommand{\GCobpzc}{G\mathpzc{Cob}}

\newcommand{\PBunpzc}{\mathpzc{PBun}}

\newcommand{\Vectpzc}{\mathpzc{Vect}}

\newcommand{\la}{\left\langle}
\newcommand{\ra}{\right\rangle}
\newcommand{\lbr}{\left\lbrace}
\newcommand{\rbr}{\right\rbrace}
\newcommand{\p}{\partial}
\newcommand{\lb}{\left(}
\newcommand{\rb}{\right)}

\renewcommand{\hom}{\mathrm{Hom}}

\allowdisplaybreaks[2]
\numberwithin{equation}{section}

\theoremstyle{definition}
\newtheorem{defn}{Definition}[section]
\theoremstyle{plain}
\newtheorem{theo}[defn]{Theorem}
\theoremstyle{plain}
\newtheorem{lem}[defn]{Lemma}
\theoremstyle{remark}
\newtheorem{rem}[defn]{Remark}
\theoremstyle{remark}

\theoremstyle{plain}

\theoremstyle{plain}
\newtheorem{prop}[defn]{Proposition}

\theoremstyle{plain}

\theoremstyle{plain}

\theoremstyle{definition}

\begin{document}

\begin{flushright}
  {\tiny
  ZMP-HH/22-17\\
  }
  \end{flushright}

  \title{A $G$-equivariant String-Net Construction}

\vspace{2cm}
\date{\today}

\vspace{0.4cm}

\author{Adrien DeLazzer Meunier, Christoph Schweigert, Matthias Traube}

\address{Adrien DeLazzer Meunier: Dalhousie University, 6299 South St, Halifax, NS, B3H 4R2 } 
\email{ adelazzer@dal.ca}
\address{Christoph Schweigert: Universit\"at Hamburg \\ 
   Bundesstra\ss e 55, 20146 Hamburg } 
\email{christoph.schweigert@uni-hamburg.de}
\address{Matthias Traube: Universit\"at Hamburg \\ 
   Bundesstra\ss e 55, 20146 Hamburg } 
\email{matthias.traube@uni-hamburg.de}
\begin{abstract}
We develop a string-net construction for the (2,1)-dimensional part of
a $G$-equivariant three-dimensional topological field theory
based on a $G$-graded spherical fusion category. In this construction,
a $G$-equivariant generalization of the Ptolemy groupoid enters.
We compute the associated cylinder categories and show that, as expected, the model
is closely related to the $G$-equivariant Turaev-Viro theory.
\end{abstract}

\maketitle

\tableofcontents

\vspace*{0.5cm}
We dedicate this article to the memory of Krzysztof Gaw\polhk{e}dzki, in admiration for
the mathematical and physical depth and breadth of his work. He pioneered
higher structures in quantum field theories and
used, in particular, equivariant structures and orbifold constructions to get beautiful insights  \cite{felder1988spectra}\cite{gawkedzki2004basic}\cite{gawkedzki2011bundle}.

\section{Introduction}

String-nets originated in physics from a description of topological phases of matter \cite{Levin:2004mi}. A mathematical formulation for string-nets was later given in \cite{kirillov2011string}. The idea is to consider a vector space generated by graphs embedded into a surface and labeled with data from a spherical fusion category. Relations are given by the graphical calculus in the category, which is considered locally on the surface. This can be understood as an example of a topological field theory in terms of generators and relations in the sense of \cite{Walker}.

Using the description of the bicategory of $(3,2,1)$-cobordisms in terms of generators and relations \cite{bartlett2015modular}, in \cite{Goosen}\cite{bartlett2022three} it was shown that the string-net construction of \cite{kirillov2011string} can be extended to a once extended three dimensional TQFT and that the string-net TQFT
is equivalent to the once extended Turaev-Viro TQFT of \cite{KirillovBalsam} \cite{balsam} \cite{balsam2010turaevviro}. 

A natural question is, whether this equivalence can be extended to a $G$-equivariant setting, for $G$ any finite group. There are several different points that need to be addressed when trying to answer this question. In \cite{heinrich2016symmetry}, a version of the Levin-Wen model on surfaces with $G$-bundles was defined. However, a rigorous mathematical formulation for $G$-equivariant strings-nets, in the spirit of \cite{kirillov2011string}, was not given. If there was a suitable $G$-equivariant string-net construction on surfaces, possibly with boundary, one should then compare it to a $G$-equivariant version of the Turaev-Viro construction. TQFTs on manifolds with $G$-bundles were defined in \cite{TuraevHQFTbook}, where they are called homotopy quantum field theories (HTQFTs). A construction of a (once extended) Turaev-Viro HTQFT with input a $G$-graded spherical fusion category is given in \cite{turaev20123}\cite{turaev20203}. Thus there is a natural candidate to compare an equivariant string-net construction to. 

In this paper we give a mathematical definition for $G$-equivariant string-nets on compact surfaces, which are allowed to have non-empty boundary. As an algebraic input we need a $G$-graded spherical fusion category $\Ccal$. Furthermore, we show that our construction indeed reproduces the $(2,1)$-part of a once extended HTQFT as defined in \cite{schweigert2020extended}. We show that its value on objects of a suitable $G$-bordism bicategory is equivalent to the $G$-center of $\Ccal$. In addition, we are able to compute string-net spaces on surfaces in terms of purely algebraic data, i.e. we will show in section \ref{higher genus} that the string-net space on a surface is isomorphic to a certain hom-space in the $G$-center. Comparing our string-net construction to the Turaev-Viro theory of \cite{turaev20203} is subtle, as we formulate, for reasons explained in \cite[Remark 2.4]{schweigert2020extended},
 our construction in the language of bicategories and use related, but slightly different, geometric inputs. Our assumptions allow us to compute cylinder categories rather than to postulate them, cf. Remark \ref{remIntro}.

An extension to three dimensional bordisms, though, is beyond the scope of this paper. Since the results of \cite{bartlett2015modular} are not available for $G$-equivariant HTQFTs such an extension would require different methods from the ones on \cite{Goosen}.

From the point of view of applications, our construction might be interesting for the construction of correlators in orbifold rational conformal theory (RCFT). Constructions of correlators in (orbifold) RCFTs in terms of three dimensional TQFT were given in a series of papers \cite{Fjelstad:2005ua, Fuchs:2004xi, Fuchs:2004dz, Fuchs:2003id, Fuchs:2002cm}. Due to the close connection of string-nets with three dimensional TQFTs, in \cite{Schweigert:2019zwt} a string-net constructions for closed RCFT correlators based on Cardy-bulk algebra field content was given. The construction was extended to arbitrary open-closed RCFTs with fixed open-closed field content in \cite{traube2022cardy} and to arbitrary open-closed RCFTs with defects in \cite{fuchs2021string}. Using the string-net construction given in this paper, it seems reasonable to obtain a construction of orbifold correlators very similar to the previous ones. 

This paper is organized as follows. In section 2 we recall some facts about spherical fusion categories as well as about $G$-graded categories. In particular we give in Proposition \ref{new crossing} an explicit expression for the $G$-crossing of the $G$-center of a $G$-graded spherical fusion category. In section 3 we recall once extended $G$-equivariant HTQFTs and explain our definition for an equivariant bordism bicategory. Sections 4, 5, and 6 are the main part of the paper. In section 4 a $G$-labeled version of the Ptolemy-groupoid is introduced. This is the central technical tool for our string-net construction. The main result in this section is Theorem \ref{G scc}, where we show that the $G$-enhanced Ptolemy-complex inherits from the ordinary Ptolemy complex the property of being connected and simply connected. Using the $G$-Ptolemy groupoid, in section 5 we finally lay out our construction of the $G$-equivariant string-net space and show that it indeed is the $(2,1)$-part of a once extended HTQFT. In section 6, we compute string-net spaces for a cylinder, pair of pants and a genus two surface with three boundary components. By doing so, we will see how the HTQFT structure of our string-net construction induces the $G$-crossing as well as the monoidal product in the $G$-center. A higher genus computation will then connect equivariant string-nets to the equivariant Turaev-Viro HTQFT.
 
\subsection{Miscellaneous Notation}

We fix some notation, which will be used throughout the whole paper. First, $\Kbb$ will be an algebraically closed field of characteristic zero, $G$ a finite group and $BG$ its classifying space, which is an Eilenberg-MacLance space of type $K(G,1)$. 

For a small category $\Ccal$, the set of objects is denoted $\Ccal_0$ and the set of morphisms $\Ccal_1$.

Given a graph $\Gamma$, its sets of vertices and edges will be $V(\Gamma)$ and $E(\Gamma)$. The set of half-edges incident to a vertex $v$ will be denoted $H(v)$. Given a surface $\Sigma$ and an embedded graph $\Gamma\hookrightarrow \Sigma$, the connected components of $\Gamma^{[2]}\coloneqq \Sigma\backslash \Gamma$ will be called \textit{$2$-faces of $\Gamma$}. If a graph $\Gamma$ is oriented, its set of oriented edges will be $E^{or}(\Gamma)$. In addition, an edge $e$ with an orientation will be written in bold symbols, i.e $\ebf\coloneqq (e,or)$. The same edge with the opposite orientation will get an additional overline $\ov{\ebf}\coloneqq (e,-or)$. A graph is called finite, if it has finitely many
vertices and edges.

\vspace*{0.5cm}
\textbf{Acknowledgement:} The authors thank Yang Yang and Theo Johnson-Freyd for useful discussions. CS and MT are supported by the
Deutsche Forschungsgemeinschaft (DFG, German Research Foundation) under SCHW1162/6-
1; CS is also supported by DFG under Germany’s Excellence Strategy - EXC 2121 "Quantum
Universe" - 390833306.

\section{Categorial Preliminaries}

\subsection{Spherical Fusion Categories}

We will work exclusively with spherical fusion categories. For the reader's convenience, we recall some definitions and facts. Proofs for the statements can be found in many sources, and an exhaustive textbook treatment is given in \cite{tensorCats}. 

Categories $\Ccal$ in this paper are always enriched in the symmetric monoidal category of finite dimensional $\Kbb$-vector spaces and abelian. In this case we speak of \textit{$\Kbb$-linear categories}. Since we fix the ground field from the start, we will just speak of \textit{linear categories}. A linear category is \textit{monoidal}, if there is a bilinear functor $\otimes:\Ccal\times \Ccal\rightarrow \Ccal$ and an unit object $\onebb$ in $\Ccal$ together with the usual associativity and unitality constraints. Without loss of generality we can assume that monoidal categories are strict, meaning that the following objects are identical $\onebb\otimes c=c=c\otimes \onebb$ and $(a\otimes b)\otimes c=a\otimes (b\otimes c)$. Assuming these strictness conditions, associativity and unitality constraints become trivial. 

In addition we take categories to be \textit{finitely semi-simple}. That is, there is a finite set of isomorphism classes of simple objects, for which we choose a set I(C) of representatives, which includes the monoidal unit.
Every object decomposes as a finite direct sum of simple objects and $\Endrm_\Ccal(\onebb)\simeq \Kbb$. A category satisfying these properties is called a \textit{fusion category}.

Furthermore, a monoidal category has \textit{right (resp.\ left) duals}, if for any object $c\in \Ccal$ there exists an object $c^\ast$  (resp.\ $ \prescript{\ast}{}{c}$) and morphisms
\eq{
\evrm_c:c^\ast\otimes c\rightarrow \onebb,\qquad \coevrm_c:\onebb\rightarrow c\otimes c^\ast
}
for right duals and 
\eq{
\widetilde{\evrm}_c:c\otimes \prescript{\ast}{}{c}\rightarrow \onebb,\qquad \widetilde{\coevrm}_c:\onebb\rightarrow \prescript{\ast}{}{c}\otimes c
}
for left duals. The evaluation and coevaluation morphisms have to satisfy the snake identities 
\eq{
(\Idrm_c\otimes \evrm_c)\circ (\coevrm_c\otimes \Idrm_c)=\Idrm_c,\qquad (\widetilde{\evrm}_c\otimes \Idrm_c)\circ (\Idrm_c\otimes \widetilde{\coevrm}_c)=\Idrm_c
}  

The category is called \textit{rigid} if every object has a left and right dual. A \textit{pivotal structure} on a rigid category is a monoidal natural isomorphism $\pi:\Idrm_\bullet\Rightarrow (\bullet)^{\ast\ast}$. Similar to the monoidal structure, we can assume  \cite[Theorem 2.2]{ng2007higher}
pivotal structures to be strict if they exist, meaning $\pi_c=\Idrm_c$. In a category with a (strict) pivotal structure, i.e. a pivotal category, we can form left and right traces of morphisms $f\in \Endrm_\Ccal(c)$
\eq{
\trrm_\ell(f)\coloneqq \evrm_{c^\ast}\circ(f\otimes \Idrm_{c^\ast})\circ \coevrm_c,\qquad \trrm_r(f)\coloneqq \widetilde{\evrm}_{\prescript{\ast}{}{c}}\circ (\Idrm_{\prescript{\ast}{}{c}}\otimes f)\circ \widetilde{\coevrm}_c\, .
}
A pivotal category is called \textit{spherical} if left and right traces coincide. Note that in a spherical category one can identify left and right duals, which we will implicitly do.
As left and right traces agree in spherical categories, we simply speak of \textit{the} trace and drop the distinction between left and right traces from notation. 

In a spherical category, we can associate to $c\in \Ccal_0$ its \textit{dimension} $d_c\in \Endrm_\Ccal(\onebb)\simeq \Kbb$ defined as 
\eq{
d_c\coloneqq \trrm(\Idrm_c)\quad .
}
The whole category has a \textit{global dimension}
\eq{
D\coloneqq \sum_{i\in I(\Ccal)}d_i^2
}
which is non-vanishing \cite[Theorem~7.21.12]{tensorCats}.

\subsection{Graphical Calculus}

The graphical calculus for spherical fusion categories plays a prominent role in the string-net construction. We discuss it quickly to fix some conventions. For $c\in \Ccal_0$ we represent $\Idrm_c$ by a straight line in the plane, oriented from bottom to top and labeled with $c$. Similar $\Idrm_{c^\ast}$ is represented by a $c$-labeled straight line oriented from top to bottom 

\begin{center}
\begin{tikzpicture}
\node at (-1,1) {$\Idrm_c=$};
\node at (0.2,0) {$c$};
\draw[->-] (0,0) -- (0,2);

\node at (2,1) {$\Idrm_{c^\ast}=$};
\node at (3.2,0) {$c$};
\draw[->-] (3,2) -- (3,0);
\end{tikzpicture}
\end{center}

A morphism $c\xrightarrow{f} d$ is drawn as

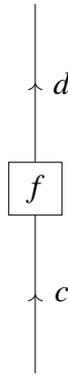
\begin{figure}[H]
    \centering
    \resizebox{0.7\width}{!}{
    \begin{tikzpicture}
        
        \draw[mid] (0,0) -- ++(0,3) node[midway, xshift=0.5cm] {\Large $c$};
        \draw (-0.5,3) rectangle ++(1,1) node[midway] {\Large $f$};
        \draw[mid](0,4) -- ++(0,3) node[midway, xshift=0.5cm] {\Large $d$};
        
    \end{tikzpicture}}
\end{figure}

and composition of morphism is simply concatenation of string-diagrams. The monoidal product is represented by drawing strands form left to right. Evaluation and coevaluation morphisms are given by oriented caps and cups

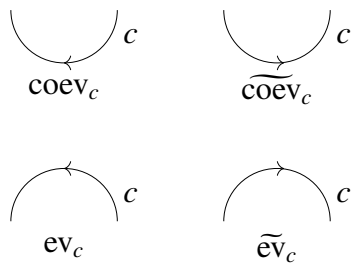
\begin{figure}[H]
    \centering
    \resizebox{0.7\width}{!}{
    \begin{tikzpicture}
        
        \draw[mid] (0,0) arc (0:-180:1cm) node[xshift=2.25cm, yshift=-0.5cm] {\Large $c$};
        \draw[mid={<}{0.5}] (4,0) arc (0:-180:1cm) node[xshift=2.25cm, yshift=-0.5cm] {\Large $c$};
        \draw[mid] (0,-4) arc (0:180:1cm) node[xshift=2.25cm, yshift=0.5cm] {\Large $c$};
        \draw[mid={<}{0.5}] (4,-4) arc (0:180:1cm) node[xshift=2.25cm, yshift=0.5cm] {\Large $c$};
        
        \node at ($(0,0) + (-1, -1.5)$) {\Large $\operatorname{coev}_c$};
        \node at ($(4,0) + (-1, -1.5)$) {\Large $\widetilde{\operatorname{coev}}_c$};
        \node at ($(0,-4) + (-1, -0.5)$) {\Large $\operatorname{ev}_c$};
        \node at ($(4,-4) + (-1, -0.5)$) {\Large $\widetilde{\operatorname{ev}}_c$};
    \end{tikzpicture}}
\end{figure}

In the string-net construction for oriented surfaces, it turns out to be convenient to label vertices not directly by morphisms in the category, but rather by elements in a vector space that treats inputs and outputs on the same footing and depends
only on the cyclic order of inputs and outputs. This vector space is constructed as a direct limit. For this limit, we need to identify different hom-spaces; for $\Ccal$ strictly pivotal, such identification maps are constructed using canonical isomorphisms $\tau_{a,b}:\hom_\Ccal(\onebb, a\otimes b )\xrightarrow{\simeq} \hom_\Ccal(\onebb, b\otimes a)$ for all $a,b\in \Ccal_0$

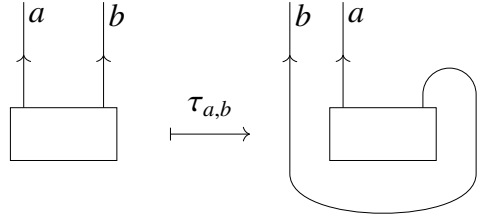
\begin{figure}[H]
    \centering
    \resizebox{0.7\width}{!}{
    \begin{tikzpicture}
        
        \draw (0,0) rectangle ++(2,1);
        \draw[mid] (0.25,1) -- ++(0,2) node[xshift=0.25cm, yshift=-0.25cm] {\Large $a$};
        \draw[mid] (1.75,1) -- ++(0,2) node[xshift=0.25cm, yshift=-0.25cm] {\Large $b$};
        \draw[{|[scale=1.5]}-{>[scale=1.5]}] (3,0.5) -- ++(1.5,0) node[midway, yshift=0.5cm] {\Large $\tau_{a,b}$};
        
        \draw (6,0) rectangle ++(2,1);
        \draw[mid] ($(6,0) + (0.25,1)$) -- ++(0,2) node[xshift=0.25cm, yshift=-0.25cm] {\Large $a$};
        \draw ($(6,0) + (1.75,1)$) -- ++(0,0.25) arc (180:0:0.5cm) -- ++(0,-1.5) .. controls ++(0,-1) and ++(0,-1) .. ++(-3.5,0) -- ++(0,3.25) [mid={>}{0.906}] node[xshift=0.25cm, yshift=-0.25cm] {\Large $b$};
        
    \end{tikzpicture}}
\end{figure}
Pivotality implies $\tau_{a,b}\circ \tau_{b,a}=\Idrm$. Therefore, instead of looking at each morphism set $\hom_\Ccal(\onebb, c_1\otimes\cdots \otimes c_n)$ on its own, we take the limit over the diagram 

\begin{tikzcd}
\ar[r]&[-1.2em] \hom_\Ccal(\onebb,c_{i+1}\otimes \cdots \otimes c_{i-1}\otimes c_i) \ar[r, "\tau_{c_{i+1},c_{i+2}\otimes \cdots \otimes c_{i}}"] &[2em] \hom_\Ccal(\onebb,c_{i+2}\otimes \cdots \otimes c_{i}\otimes c_{i+1})\ar[r]&[-1.2em]\phantom{0}
\end{tikzcd} 

For later use, we denote the limit by $\Ccal(c_1,\cdots, c_n)$. An element $ f\in \Ccal(c_1,\cdots, c_n)$ will be represented by a circular coupon rather than a box, since it only depends on the cyclic order of $c_1,\cdots, c_n$. 

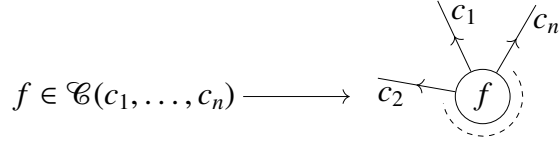
\begin{figure}[H]
    \centering
    \resizebox{0.7\width}{!}{
    \begin{tikzpicture}
        
        \draw[-{>[scale=1.5]}] (0,0) node[left] {\Large $f \in \Ccal(c_1,\ldots,c_n$)} -- ++(2,0);
        \node[circle, draw] at (4.5,0) (a) {\Large $f$};
        \centerarc[dashed](4.5,0)(40:-170:0.75cm);
        \draw[mid] (a) -- +(60:2) node[xshift=0.2cm, yshift=-0.35cm] {\Large $c_n$};
        \draw[mid] (a) -- +(115:2) node[xshift=0.5cm, yshift=-0.25cm] {\Large $c_1$};
        \draw[mid] (a) -- +(170:2) node[xshift=0.25cm, yshift=-0.35cm] {\Large $c_2$};
        
    \end{tikzpicture}}
\end{figure}

There is a partial composition map for elements in the limit, induced by the maps
\eq{
\hom_\Ccal(\onebb, a\otimes b)\otimes \hom_\Ccal(\onebb, b^\ast\otimes c)&\rightarrow \hom_\Ccal(\onebb, a\otimes c)\\
(f,g)&\mapsto \left(\Idrm_a\otimes \widetilde{\mathrm{ev}}_b\otimes \Idrm_c\right)\circ (f\otimes g)
} 
The composition will still be represented by concatenating strands in string-diagrams.

Using semi-simplicity for any $c\in \Ccal_0$ we can pick a basis $\lbr \alpha_{c,i}^k\rbr_k $ in $\Ccal(c, i^\ast)$ and $\lbr \alpha_m^{c,i}\rbr_m $ in $\Ccal(c^\ast, i)$ for any $i\in I(\Ccal)$ with 

\begin{figure}[H]
    \centering
    \resizebox{0.7\width}{!}{
    \begin{tikzpicture}
        
        \node[circle, draw, inner sep=0] at (0,0) (a) {\Large $\alpha_{c,i}^k$};
        \node[circle, draw, inner sep=0] at (0,2) (b) {\Large $\alpha_m^{c,i}$};
        \draw[mid] (0,-2) node[above right] {\Large $i$} -- (a);
        \draw[mid] (a) -- (b) node[anchor=west, midway, xshift=0.1cm] (c) {\Large $c$};
        \draw[mid] (b) -- (0,4) node[below right] {\Large $i$};

        \node at ($(c) + (2,0)$) (d) {\Large $=\ \ \delta_{km}$};
        \draw[mid] ($(d) + (2, -3)$) -- ++(0,6) node[midway, xshift=0.35cm] {\Large $i$};
        \end{tikzpicture}
        }
        \end{figure}
        and thus
        
        \begin{figure}[H]
    \centering
    \resizebox{0.7\width}{!}{
    \begin{tikzpicture}
        \node[anchor=east] at (-1.5,-6) {\LARGE $\displaystyle{\sum_{i \in I(\Ccal)} d_i}$};
        \node[circle, draw] at (-1.0,-7) (a) {\Large {$\alpha$}};
        \node[circle, draw] at (-1.0,-5) (b) {\Large {$\alpha$}};
        \draw[mid] (-1.0,-9) node[anchor=east, above right] {\Large $c$} -- (a);
        \draw[mid] (a) -- (b) node[anchor=west, midway, xshift=0.1cm] (m) {\Large $i$};
        \draw[mid] (b) -- (-1.0,-3) node[anchor=east, below right] {\Large $c$};
        
        \node[anchor=west] at ($(m) + (0.5,0)$) (e) {\LARGE $\displaystyle{\coloneqq\ \sum_{i \in I(\Ccal)} \sum_k d_i}$};
        \node[circle, draw, inner sep=0] at ($(e.east) + (0.5,-1)$) (c) {\Large {$\alpha^{c,i}_k$}};
        \node[circle, draw, inner sep=0] at ($(e.east) + (0.5,1)$) (d) {\Large {$\alpha_{c,i}^k$}};
        \draw[mid] ($(e.east) + (0.5,-3)$) node[above right] {\Large $c$} -- (c);
        \draw[mid] (c) -- (d) node[anchor=west, midway, xshift=0.1cm] (n) {\Large $i$};
        \draw[mid] (d) -- ($(e.east) + (0.5,3)$) node[below right] {\Large $c$};
        
        \node[anchor=west] at ($(e.east) + (2,0)$) (f) {\LARGE $\displaystyle{=\phantom{\sum_i}}$};
        \draw[mid] ($(f.east) + (0.25,-3)$) -- ($(f.east) + (0.25,3)$) node[midway, xshift=0.35cm] {\Large $c$};

    \end{tikzpicture}}
    \caption{Completeness relation in the semi-simple category $\Ccal$.}
    \label{completeness relation}
\end{figure}

where $d_i=\trrm(\Idrm_i)$. Finally we discuss \textit{6j-symbols} in $\Ccal$. In a pivotal finite semi-simple category, by choosing bases in the spaces of three point couplings, the vector space $\Ccal(i, j, k,\ell)$ has two distinct bases. One stems from the decomposition $\Ccal(i, j, k, \ell)\simeq \bigoplus_r\Ccal(i, j, r)\otimes \Ccal(r^\ast, k, \ell)$, whereas the other corresponds to the splitting $\Ccal(i, j, k, \ell)\simeq \bigoplus_{s}\Ccal(j, k, s)\otimes \Ccal(s^\ast,\ell, i)$. The entries of the transformation matrix between the two bases are the 6j-symbols

\begin{figure}[H]
    \makebox[\textwidth][c]{
    \resizebox{0.7\width}{!}{
    \begin{tikzpicture}
        
        \node[circle, draw] at (1,1.5) (a) {\Large $\alpha$};
        \node[circle, draw] at ($(a) + (1,1.5)$) (b) {\Large $\beta$};
        \draw[mid] (a) to[bend left] node[pos=0.5, xshift=-0.5cm] {\Large $r$} (b);
        \draw[mid] (a) to[bend right] (0,0) node[above right] {\Large $i$};
        \draw[mid] (a) to[bend left] (2,0) node[above right] {\Large $j$};
        \draw[mid] (b) to[bend left] (4,0) node[above right] {\Large $k$};
        \draw[mid] (b) -- ++(0,2) node[below right] {\Large $l$};
        
        \node[anchor=west] at (4.5,2) {\LARGE $\displaystyle{=\sum_{s\in I(\Ccal)} \sum_{\gamma, \delta} F^{ijkl} \begin{bmatrix} \alpha & r & \beta \\ \gamma & s &\delta \end{bmatrix}}$};
        
        \node at (13.5, 3) (c) {};
        
        \node[circle, draw] at (c) (d) {\Large $\delta$};
        \node[circle, draw] at ($(d) + (1,-1.5)$) (g) {\Large $\gamma$};
        \draw[mid] (g) to[bend right] node[pos=0.5, xshift=0.5cm] {\Large $s$} (d);
        \draw[mid] (g) to[bend left] ($(g) + (1,-1.5)$) node[above right] {\Large $k$};
        \draw[mid] (g) to[bend right] ($(g) + (-1,-1.5)$) node[above right] {\Large $j$};
        \draw[mid] (d) to[bend right] ($(d) + (-2,-3)$) node[above right] {\Large $i$};
        \draw[mid] (d) -- ++(0,2) node[below right] {\Large $l$};
        
    \end{tikzpicture}}}
\end{figure}

The 6j-symbols satisfy the usual pentagon relation.

\subsection{$G$-categories}

This section is a recollection of the relevant definitions associated to categories which are $G$-graded and possibly carry a $G$-action. The reader can find detailed accounts of $G$-categories in the literature, e.g. \cite[Appendix~5]{TuraevHQFTbook}\cite{turaev2013graded} and references therein. Besides recalling definitions, we give in Proposition \ref{new crossing} an alternative definition for a $G$-crossing on the $G$-center of a $G$-graded category. Though our definition is equivalent to the one in \cite{turaev2013graded}, we still introduce it, since it makes the discussion of string-nets on cylinders more transparent later. 

Throughout the section $\Ccal$ is a $\Kbb$-linear, strict monoidal category. $\ov{G}$ denotes the category having as objects the elements of $G$ and only identity morphisms. Multiplication in $G$ endows $\ov{G}$ with a strict monoidal structure.

\begin{defn}
The monoidal, linear category $\Ccal$ is \textit{$G$-graded}, if it decomposes into a direct sum of pairwise disjoint, full $\Kbb$-linear subcategories $\lbr \Ccal_g\rbr_{g\in G}$, such that
\begin{enumerate}[label=\roman*)]
\item any object $c\in\Ccal$ decomposes as a finite direct sum $c=c_{g_1}\oplus \cdots \oplus c_{g_n}$, with $c_{g_i}\in \Ccal_{g_i}$.
\item for $c\in \Ccal_g$, $d\in \Ccal_h$, it holds $c\otimes d\in \Ccal_{gh}$.
\item $\hom(c,d)=0$ for $c\in \Ccal_g$ and $d\in \Ccal_h$ with $g\neq h$.
\item $\onebb\in \Ccal_e$.
\end{enumerate}
\end{defn}

The subcategories $\Ccal_g$ are called homogeneous components of $\Ccal$. The component $\Ccal_e$ for the neutral element $e\in G$ is called the \textit{neutral component}. There is an obvious forgetful functor $U:\Ccal\rightarrow \dot{\Ccal}$, which simply forgets the $G$-grading. A $G$-graded category $\Ccal$ is rigid/ pivotal/ spherical if its underlying monoidal linear category $\dot{\Ccal}$ is rigid/ pivotal/ spherical. It is fusion, if its underlying category is a fusion category, such that any homogeneous component contains at least one simple object. In a $G$-graded fusion category, any set of representing objects for isomorphism classes of simple objects splits as a disjoint union $I=\bigsqcup_{g\in G} I_g$ with $I_g$ the set of isomorphism classes of simples in $\Ccal_g$. Thus simple objects are always homogeneous. 


In a $G$-graded category the group $G$ serves as an index set, but doesn't act on the category so far. A $G$-action will come in the form of a $G$-crossing. For a monoidal category $\Dcal$, recall that $\Autrm_\otimes(\Dcal)$ is the category having monoidal equivalences $F:\Dcal\xrightarrow{\simeq}\Dcal$ as objects and monoidal natural isomorphisms as morphisms. Composition of functors equips it with a strict monoidal structure whose monoidal unit is the identity functor. 

\begin{defn}
A $G$-crossed category is a $G$-graded category $\Ccal$ together with a strong monoidal functor $\rho:\ov{G}\rightarrow \Autrm_\otimes(\Ccal)$ such that $\rho_h(\Ccal_g)\subset \Ccal_{h^{-1}gh}$. The functor $\rho$ is called a \textit{$G$-crossing}.
\end{defn}

For any $g\in G$, $\rho_g$ is a strong monoidal functor. In addition, as $\rho$ is strong monoidal, a $G$-crossed category comes with natural isomorphisms $\lbr \eta_{h,g}(\bullet):\rho_g\rho_h(\bullet)\xrightarrow{\simeq}\rho_{hg}(\bullet)\rbr $ and $\eta_0(\bullet):\Idrm_\Ccal\xrightarrow{\simeq} \rho_e(\bullet)$. These maps satisfy the usual coherence diagrams, which can be found in \cite[section~3]{turaev2013graded}.


Just as it is sometimes necessary to equip a given monoidal category with the additional structure of a braiding, there is a meaningful notion of a braiding for G-crossed categories. This cannot be an ordinary braiding for the underlying monoidal category since for non-abelian $G$  it holds  in general that $\hom(c\otimes d,d\otimes c)=0$ for $c$, $d$ in different homogeneous components of $\Ccal$. However, $\hom(c\otimes d, d\otimes \rho_g(c))$ for $d\in \Ccal_g$ and $c\in \Ccal_h$ has a chance to be non-zero, as both $c\otimes d$ and $d\otimes \rho_g(c)\in \Ccal_{hg}$. Thus, a $G$-braiding is a natural isomorphism 
\eq{
\lbr \beta_{c,d}:c\otimes d\rightarrow d\otimes \rho_g(c)\rbr 
}
defined for homogeneous elements $d\in \Ccal_g$, $c\in \Ccal_h$ for all $g$, $h\in G$ and linearly extended to all objects of $\Ccal$. The natural isomorphism has to satisfy three coherence diagrams for which we refer to \cite[section~3]{turaev2013graded}.

To a $G$-graded category we associate its $G$-center $\Zsf_G(\Ccal)$, which is defined to be the relative center with respect to $\Ccal_{e}$. To be a bit more explicit, the $G$-center has objects pairs $(c,\gamma_{c,\bullet})$, where $c\in \Ccal$ and $\gamma_c$ is a relative half-braiding for objects in the neutral component $\Ccal_e$, i.e. 
\eq{
\gamma_{c,\bullet}=\lbr \gamma_{c,X}:c\otimes X\rightarrow X\otimes c\rbr_{X\in C_{e}} 
}
is a natural isomorphism satisfying the usual hexagon relation. A morphism $(c,\gamma_c)\rightarrow (d,\gamma_d)$ is a morphism $f:c\rightarrow d$ in $\Ccal$ satisfying $(\Idrm_X\otimes f)\circ\gamma_{c,X}=\gamma_{d,X}\circ (f\otimes \Idrm_X)$. The $G$-center obviously is a $G$-graded category. Similar to the non-graded case, for $\Ccal$ a $G$-graded spherical fusion category, the center has the structure of a $G$-modular category. Of course, for a $G$-graded category there also exists its Drinfeld center $\Zsf(\Ccal)$. The $G$-center can be very different from the Drinfeld center. However the two are related via orbifolding \cite[Theorem~3.5]{gelaki2009centers}. Since we don't need the full $G$-modular structure on $\Zsf_G(\Ccal)$ we simply refer to \cite{turaev2013graded} for the definition of a $G$-modular category. We only need that $\Zsf_G(\Ccal)$ is $G$-crossed and we explicitly give the $G$-crossing. 

Similar to the non-equivariant case, there is an adjunction 
\eq{
I:\Ccal\rightleftharpoons \Zsf_G(\Ccal):F,\qquad I\dashv F
}
with $F$ the forgetful functor and its adjoint functor $I$, the \textit{induction functor}, defined on objects 
\eq{\label{induction functor}
I(c)\coloneqq \bigoplus_{i\in I_e}i^\ast\otimes c\otimes i\, , 
}
where $I_e$ is a set of simple objects in $\Ccal_e$. Its action on morphisms is the obvious one. To have the structure of an element in the graded center, the object $I(c)$ is equipped with the standard non-crossing half-braiding

\begin{figure}[H]
    \def \h{3}
    \def \w{1}
    \centering
    \resizebox{0.7\width}{!}{
    \begin{tikzpicture}
    \node[anchor=east] at (0,0) (an) {\LARGE $\displaystyle{\sum_{i,j \in I_e} d_j}$};
        \node[circle, draw] at ($(an) + (2*\w,0)$) (bn) {$\alpha$};
        \draw[mid] (bn) -- ($(bn) + (0,-\h)$) node[above right] {\Large $i$};
        \draw[mid] ($(bn) + (0,\h)$) node[below right] {\Large $j$} -- (bn);
        \draw[mid] ($(bn) + (\w,-\h)$) -- ++(0,2*\h) node[midway, xshift=0.3cm] {\Large $c$};
        \node[circle, draw] at ($(bn) + (2*\w,0)$) (cn) {$\alpha$};
        \draw[mid] ($(cn) + (0,-\h)$) node[above right] {\Large $i$} -- (cn);
        \draw[mid] (cn) -- ($(cn) + (0,\h)$) node[below right] {\Large $j$};
        \draw[mid, violet] (bn) .. controls ($(bn) + (-\w,0.25*\h)$) and ($(bn) + (-\w,0.25*\h)$) .. ($(bn) + (-\w,\h)$) node[below left, text=violet] {\Large $X$};
        \draw[mid, violet] ($(cn) + (\w, -\h)$) node[above right, text=violet] {\Large $X$} .. controls ($(cn) + (\w, -0.25*\h)$) and ($(cn) + (\w, -0.25*\h)$) .. (cn);
        \node at (-3,0) {\Large $\gamma_{I(c),X}\coloneqq$};
        
    \end{tikzpicture}}
\end{figure}

where $X\in \Ccal_e$.  (Recall from Figure \ref{completeness relation} that the
pairwise appearance of $\alpha$ implies a summation over dual bases.)
We will continue with this notation, as figures tend to become overloaded with notation otherwise. Due to the compatibility condition with the half-braiding $\hom_{\Zsf_G(\Ccal)}(a,b)\subset \hom_\Ccal(a,b)$ is a proper subspace. 
(Here, by abuse of notation, we identify $a\in \Zsf_G(\Ccal)$ with the
underlying object in $\Ccal$.)
An idempotent $P:\hom_\Ccal(a,b)\rightarrow \hom_\Ccal(a,b)$ with $\Imrm(P)=\hom_{\Zsf_G(\Ccal)}(a,b)$ is given by

\begin{figure}[H]
    \def \h{3}
    \def \w{2}
    \centering
    \resizebox{0.7\width}{!}{
    \begin{tikzpicture}
        \node[anchor=east] at (0,0) (a) {\LARGE $\displaystyle{P(f) \coloneqq}$};
        \node[circle, draw] at (\w, 0) (f) {\Large $f$};
        \draw[violet] (f) circle (0.75*\w);
        \node[draw, fill=white, white] at ($(f) + (0,0.75*\w)$) {};
        \node[draw, fill=white, white] at ($(f) + (0,-0.75*\w)$) {};
        \draw[mid] ($(f) + (0, -\h)$) node[above right] {\Large $a$} -- (f);
        \draw[mid] (f) -- ($(f) + (0, \h)$) node[below right] {\Large $b$};
        
    \end{tikzpicture}}
\end{figure}

where we introduced the \textit{cloaking circle} in $\Ccal$

\begin{figure}[H]
    \def \h{3}
    \def \w{2}
    \def \r{0.75 * \w}
    \centering
    \resizebox{0.7 \width}{!}{
    \begin{tikzpicture}
        \draw[violet] (0,0) circle (\r);
        \node[anchor=west] at (\w,0) (a) {\Large $=\displaystyle{\sum_{i \in I_e}} \frac{d_i}{D}$};
        \draw[mid] ($(a.east) + (\w,0)$) circle (\r) node[xshift=\r cm + 0.25cm] {\Large $i$};
    \end{tikzpicture}}
\end{figure}

and the crossings of the cloaking circle with the $a$ and $b$-labeled lines correspond to the half-braidings of $a$ and $b$. The proof of that $P$ is an idempotent is exactly the same as in the non-graded case (c.f. \cite{BalsamKirillov}).
The existence of a $G$-crossing on $\Zsf_G(\Ccal)$ for $\Ccal$ a $G$-fusion category was proven in \cite{gelaki2009centers}. An explicit expression for a $G$-crossing on $\Zsf_G(\Ccal)$ is given in \cite{turaev2013graded}, where $\Ccal$ only needs to be a non-singular\footnote{See \cite[section~4.1]{turaev2013graded} for the definition of a non-singular category.}, pivotal $G$-graded category. We only work with spherical $G$-fusion categories, which are automatically non-singular. To define a $G$-crossing we generalize \eqref{induction functor} and consider a functor $I^h:\Zsf_G(\Ccal)\rightarrow \Ccal$, which acts on objects as follows $I^h(c)=\bigoplus_{i\in I_h}i^\ast\otimes c\otimes i$. The action on morphisms is the obvious one. We want to construct a $G$-crossing on $\Zsf_G(\Ccal)$ from $I^h$. In order to do so, we consider the idempotent

\begin{figure}[H]
    \def \h{3}
    \def \w{2}
    \def \r{0.75 * \w}
    \centering
    \resizebox{0.7\width}{!}{
    \begin{tikzpicture}
        \node at (0,0) (a) {\LARGE $\displaystyle{\pi^h_c \colon I^h(c)}$};
        \node[anchor=west] at ($(a.east) + (\w,0)$) (b) {\LARGE $\displaystyle{I^h(c)}$};
        \draw[-{To[scale=1.5]}] (a.east) -- (b.west);
        
        \node[anchor=west] at ($(a.west) + (0,-1.25*\h)$) (c) {\LARGE $\displaystyle{\pi^h_c \coloneqq \sum_{i,j \in I_h} \frac{d_i}{D}}$};
        \draw[mid={>}{0.9}] ($(c.east) + (0,+0.75*\h)$) .. controls ++(0,-1.5) and ++(0,-1.5) .. ($(c.east) + (1*\w,+0.75*\h)$) node[below right] {\Large $j$} node[midway, fill=white] {};
        \draw[mid={<}{0.9}] ($(c.east) + (0*\w,-0.75*\h)$) .. controls ++(0,1.5) and ++(0,1.5) .. ($(c.east) + (1.*\w,-0.75*\h)$) node[above right] {\Large $i$} node[midway, fill=white] {};
        \draw[mid] ($(c.east) + (0.5*\w,-0.75*\h)$) -- ($(c.east) + (0.5*\w,0.75*\h)$) node[midway, xshift=0.5cm] {\Large $c$};
    
    \end{tikzpicture}}
\end{figure}

where we again use the half-braiding $\gamma_{c,i^\ast\otimes j}$ to braid the strands. It is shown in \cite{gelaki2009centers} that $\sum_{i\in I_{e}}d_i^2=\sum_{U\in I_h}d_U^2$ for all $h\in G$. Using this, it is straightforward to show that $\pi_c$ is indeed an idempotent. Its image is denoted $P(c)\coloneqq \Imrm(\pi_c)$, restriction and inclusion maps are depicted as follows

\begin{figure}[H]
    \def \h{3}
    \def \w{2}
    \makebox[\textwidth][c]{
    \resizebox{0.7\width}{!}{
    \begin{tikzpicture}
        \node at (0,0) (a) {\LARGE $\displaystyle{e_c \colon P(c)}$};
        \node[anchor=west] at ($(a.east) + (\w,0)$) (b) {\LARGE $\displaystyle{I^h(c)}$};
        \draw[{|[scale=1.5]}-{To[scale=1.5]}] (a.east) -- (b.west);
        
        \node[anchor=west] at ($(a.west) + (0,-1.25*\h)$) (c) {\LARGE $\displaystyle{e_c \coloneqq \sum_{i \in I_h}}$};
        \draw[mid] ($(c.east) + (0.5*\w,-2)$) node[above right] {\Large $P(c)$} -- ++(0,2) coordinate (x);
        \draw (x) -- ++(1, 1) -- ++(-2,0) -- cycle;
        
        \draw[mid] ($(x) + (0,1)$) -- ++(0,2) node[below right] {\Large $c$};
        \draw[mid={<}{0.5}] ($(x) + (-0.75,1)$) -- ++(0,2);
        \draw[mid] ($(x) + (0.75,1)$) -- ++(0,2) node[below right] {\Large $i$};
        
        \node at (8,0) (a2) {\LARGE $\displaystyle{r_c \colon I^h(c)}$};
        \node[anchor=west] at ($(a2.east) + (\w,0)$) (b2) {\LARGE $\displaystyle{P(c)}$};
        \draw[{|[scale=1.5]}-{To[scale=1.5]}] (a2.east) -- (b2.west);
        
        \node[anchor=west] at ($(a2.west) + (0,-1.25*\h)$) (c2) {\LARGE $\displaystyle{r_c \coloneqq \sum_{i \in I_h}}$};
        \draw[mid={<}{0.5}] ($(c2.east) + (0.5*\w,3)$) node[below right] {\Large $P(c)$} -- ++(0,-2) coordinate (x2);
        \draw (x2) -- ++(-1, -1) -- ++(2,0) -- cycle;
        
        \draw[mid={<}{0.5}] ($(x2) + (0,-1)$) -- ++(0,-2) node[above right] {\Large $c$};
        \draw[mid] ($(x2) + (-0.75,-1)$) -- ++(0,-2);
        \draw[mid={<}{0.5}] ($(x2) + (0.75,-1)$) -- ++(0,-2) node[above right] {\Large $i$};
        
    \end{tikzpicture}}}
\end{figure}
The image $P(c)$ has half-braiding 

\begin{figure}[H]
    \def\s{1}
    \centering
    \resizebox{0.7\width}{!}{
    \begin{tikzpicture}
        \node[anchor=west] at (0,0) (a) {\LARGE $\displaystyle{\gamma_{P(c),X} \coloneqq \sum_{i,j \in I_h} \frac{d_i}{D}}$};
        
        \draw[mid] ($(a.east) + (2,-3.5)$) node[above right] {\Large $P(c)$} -- ++(0,1) coordinate (x);
        \draw (x) -- ++(1,1) -- ++(-2,0) -- cycle;
        
        \node at ($(x) + (0,4)$) (y) {};
        \draw[mid] ($(y) + (0,1)$) -- ++(0,1) node[below right] {\Large $P(c)$} coordinate (z);
        \draw ($(y) + (0,1)$) -- ++(-1,-1) -- ++(2,0) -- cycle;
        
        \draw[mid={<}{0.2}] ($(y) + (0.75, 0)$) node[below right] {\Large $j$} .. controls ++(0,-1) and ++(0,-1) .. ($(y) + (-0.75, 0)$) node[midway, fill=white] {};
        \draw[mid={>}{0.2}] ($(x) + (0.75, 1)$) node[above right] {\Large $i$} .. controls ++(0,1) and ++(0,1) .. ($(x) + (-0.75, 1)$) node[midway, fill=white] {};
        \draw[mid={>}{0.1}, violet] ($(a.east) + (4,-3.5)$) node[above right, violet] {\Large $X$} .. controls ++(0,6) and ++(0,-6) ..($(z) + (-2,0)$) node[midway, fill=white] {};
        
        \draw[mid={>}{0.5}] ($(x) + (0,1)$) -- (y.center) node[midway, xshift=-0.25cm, yshift=-0.4cm] {\Large $c$};
        
    \end{tikzpicture}}
\end{figure}

where $X\in \Ccal_e$ and we use again the half-braiding $\gamma_{c,i\otimes X\otimes j^\ast}$. Thus $(P(c),\gamma_{P(c),\bullet})\in \Zsf_G(\Ccal)_{h^{-1}gh}$. From that, we can give an explicit $G$-crossing for $\Zsf_G(\Ccal)$.
\begin{prop}\label{new crossing} The maps
\eq{
\phi_h:\Zsf_G(\Ccal)_g&\rightarrow \Zsf_G(\Ccal)_{h^{-1}gh}\\
c&\mapsto (P(c),\gamma_{P(c),\bullet})
}
constitute a $G$-crossing on $\Zsf_G(\Ccal)$. 
\end{prop}

Up to a reordering of tensor factors the proof is  the same as the one given in \cite[section~4]{turaev2013graded}, hence we skip it here. 
\begin{rem} The $G$-crossing we defined is not the same as the $G$-crossing given in \cite{turaev2013graded}. However, in the semi-simple case, it is not hard to show that the center equipped with the two different $G$-crossings are equivalent as $G$-crossed categories. Our definition is motivated by string-net constructions on cylinders (cf. section \ref{cylinderSection}).
\end{rem}

\section{Once extended $G$-equivariant HTQFTs}

The modern formulation of once extended HTQFTs uses the language of bicategories and 2-functors. A suitable definition of a symmetric monoidal bicategory $\GBordpzc(n,n-1,n-2)$ of once extended $G$-equivariant bordisms was given in \cite{schweigert2020extended}. For technical reasons we need a slight modification of this bicategory and consider pointed maps to $BG$ for objects. For convenience, we will also choose base points for one-dimensional manifolds in the following definition. As we will explain in Remark 3.2, these choices are not really essential 

In the following a manifold $M$ is \textit{pointed}, if for each of its connected components a distinguished basepoint has been chosen. A \textit{map between pointed manifolds $M\xrightarrow{f} N$} is a continuous, basepoint-preserving map. In addition, once and for all we choose a basepoint $\star\in BG$.

\begin{defn} The symmetric monoidal bicategory $\GBordpzc_\star(3,2,1)$ is given by:
\begin{itemize}
\item[(0)] Objects of $\GBordpzc_\star(3,2,1)$ are pairs $(M,f)$ of a pointed, closed, oriented $1$-dimensional manifold $M$ and a pointed map $f:M\rightarrow (BG,\star)$. 
\item[(1)] A $1$-morphism is a pair $(\Sigma,\zeta):(M_0,f_0)\rightarrow (M_1,f_1)$ consisting of an oriented, compact, collared 2-dimensional manifold with boundary $\Sigma$, orientation preserving diffeomorphisms $\iota_0:M_0\times (-1,0]\rightarrow \Sigma_0$, $\iota_1:M_1\times [0,1)\rightarrow \Sigma_1$ and a map $\zeta:\Sigma\rightarrow BG$, such that the diagram 
\begin{figure}[H]
\centering
\begin{tikzcd}
& \Sigma\ar[dd,"\zeta"] &\\
M_0\times\lbr 0\rbr \ar[ur,"\iota_0"]\ar[dr,"f_0"'] & & M_1\times \lbr 0\rbr \ar[ul,"\iota_1"']\ar[dl,"f_1"]\\
& BG &
\end{tikzcd}
\label{1-morph}
\end{figure}
\noindent
commutes. Here $\Sigma_0\cup \Sigma_1$ is a collar for $\Sigma$. Note that no basepoint is chosen for $\Sigma$.

\item[(2)] A $2$-morphism $(W,\phi):(\Sigma_0,\zeta_0)\Rightarrow (\Sigma_1,\zeta_1)$ in $\GBordpzc_\star(3,2,1)$ is a diffeomorphism class of a $3$-dimensional, collared, compact oriented bordism $W:\Sigma_0\rightarrow \Sigma_1$ together with a map $\phi:W\rightarrow BG$ satisfying a compatibility diagram for the maps into $BG$. 
\end{itemize}

For the precise definition of $2$-morphisms, composition of morphisms and the symmetric monoidal structure in $\GBordpzc_\star(3,2,1)$ we refer to \cite[Definition~2.3]{schweigert2020extended}.  $1$-morphisms in $\GBordpzc_\star(3,2,1)$ are also referred to as $G$-surfaces.
\end{defn}

\begin{rem}
Since $BG$ is path connected there is an equivalence of symmetric monoidal bicategories $\GBordpzc_\star(3,2,1)$ and $\GBordpzc(3,2,1)$. The equivalence is given by the forgetful functor $\GBordpzc_\star(3,2,1)\rightarrow \GBordpzc(3,2,1)$, forgetting the base point on objects. Since any map $S^1\rightarrow BG$ is homotopy equivalent to a basepoint preserving map, the forgetful functor is essentially surjective on objects. Note that $1$- and $2$-morphisms are literally the same for $\GBordpzc_\star(3,2,1)$ and $\GBordpzc(3,2,1)$, thus the two bicategories are equivalent as symmetric monoidal bicategories.
\end{rem}

\begin{defn} A $3$-dimensional, once extended $G$-equivariant HTQFT with values in a symmetric monoidal bicategory $\Scal$ is a symmetric monoidal $2$-functor
\eq{
Z:\, \GBordpzc_\star(3,2,1)\rightarrow \Scal
}
with the additional requirement that $Z$ depends only on the homotopy class 
realtive to the boundary of the map $\phi:W\rightarrow BG$ for a $2$-morphism $(W,\phi)$. 
\end{defn}

In our definition of the $G$-equivariant bordism bicategory and once extended $G$-equivariant HTQFTs, one has pointed maps to the classifying space as data and all compatibility diagrams strictly commute. Homotopy invariance is built in as a property of the functor $Z$. In \cite{turaev20203} a different formulation is given. The authors set up a category $\GCobpzc$ with objects pointed surfaces equipped with homotopy classes of pointed maps to the classifying space. Morphisms in $\GCobpzc$ are $3$-dimensional cobordisms, which again come with a homotopy class of pointed maps to $BG$. Diagrams of maps to $BG$ then only have to commute up to homotopy. A comparison between the two formulations is subtle. We chose the bicategorical framework, since we want to construct the $(2,1)$-part of a once extended HTQFT using string-nets. Bicategories are the natural algebraic framework for such a construction.

Though the definition of a once extended HTQFT is formulated for arbitrary symmetric monoidal bicategories as targets, we will only work with the category $\BiModpzc_\Kbb$ of bimodules as targets. This symmetric monoidal bicategory is standard \cite[Proposition~7.8.2]{borceux_1994}, see e.g. \cite[Definition~2.8]{bartlett2015modular} for the $\Vectpzc_\Kbb$-enriched version. Its objects are linear categories. A $1$-morphism $\Ccal\xrightarrow{F}\Dcal$ is a linear functor 
\eq{
\Dcal^{op}\boxtimes \Ccal\rightarrow \Vectpzc_\Kbb\, ,
}
where $\Dcal^{op}\boxtimes \Ccal$ has objects pairs $(d,c)\in \Dcal^{op}\times \Ccal$ and $\hom_{\Dcal^{op}\boxtimes \Ccal}\left((d,c),(d^\prime,c^\prime)\right)\coloneqq \hom_{\Dcal^{op}}(d,d^\prime)\otimes_\Kbb\hom_{\Ccal}(c,c^\prime)$. Another name for $1$-morphisms are bimodules, or profunctors. A $2$-morphism simply is a natural transformation between linear functors. The enriched tensor product gives $\BiModpzc_\Kbb$ a symmetric monoidal structure. Given profunctors $F:\Ecal^{op}\boxtimes \Dcal\rightarrow \Vectpzc_\Kbb$ and $G:\Dcal^{op}\boxtimes \Ccal\rightarrow \Vectpzc_\Kbb$ the composition $F\circ G:\Ecal^{op}\boxtimes \Ccal\rightarrow \Vectpzc_\Kbb$ is given by the coend
\eq{
(F\circ G)(e,c)\coloneqq \int^{d\in \Dcal}F(e,d)\otimes G(d,c)\quad . 
}
\section{$G$-equivariant Ptolemy Groupoid}\label{Ptolemy section}

In this section we introduce $G$-triangulations on surfaces, which are an enhancement of ideal triangulations of a surface. The main result is Theorem \ref{G scc}. It allows us to treat any $G$-surface as combinatorial object. 

We consider an oriented, compact smooth surface $\Sigma$ with $r>0$ boundary components. We choose a distinguished point on each connected component of the boundary and denote by $\delta=\lbr \delta_1,\cdots, \delta_r\rbr $ the chosen set of points. In addition, we fix an arbitrary finite set $M\subset \Sigma\backslash \p \Sigma$ of marked points. In case $\Sigma$ is a disk, $M$ has to contain at least one element, and in all other cases $M$ may be empty.

\begin{defn}\cite[Definition~1.19]{PennerBook} An ideal triangulation $\Tsf$ of $(\Sigma,\delta, M)$ consists of a collection $\lbr \alpha_i\rbr_{i\in I}$ of isotopy classes of embedded arcs, having endpoints in $\delta\cup M$, such that for every boundary component $b$, its isotopy class relative to $\delta\cup M$ is contained in $\lbr \alpha_i\rbr_{i\in I}$ and $\Sigma-\cup_{i\in I}\alpha_i$ is a disjoint union of triangles. 
\end{defn}

The dual graph to an ideal triangulation is a uni-trivalent fat graph, whose cyclic order at vertices is induced by the orientation of the boundaries of dual triangles. The orientation of the boundary of a triangle is, of course, induced by the orientation of $\Sigma$.

The $1$-skeleton $\Tsf^{[1]}$ of $\Tsf$ is an isotopy class of an embedded graph and when speaking of vertices, edges and faces of an ideal triangulation, we always mean vertices, edges and faces of $\Tsf^{[1]}$. Ideal triangulations aren't triangulations in general, e.g. we may encounter bubble graphs like in figure \ref{figure1}.

\begin{figure}
\centering
\begin{tikzpicture}
\coordinate (a) at (0,0);
\coordinate (b) at (0,-1.5);

\draw (a) circle (1.5);
\draw (a) -- (b);

\fill (a) circle (2pt);
\fill (b) circle (2pt);
\end{tikzpicture}
\caption{One configuration of possible arcs in an ideal triangulation. By cutting along the edges one obtains an honest triangle.}
\label{figure1}
\end{figure}
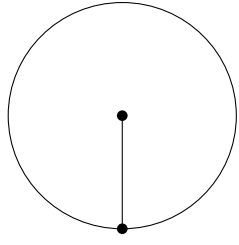

It is well known that any two triangulations of a surface can be transformed into each other by a finite sequence of $2-2$ and $3-1$ Pachner moves. Though ideal triangulations aren't triangulations, there is an analog of the $2-2$-Pachner move for ideal triangulations. 

\begin{defn} Let $f\in E(\Tsf)$ be such that $f$ is adjacent to two different $2$-faces of $\Tsf$ and $f$ is not a boundary edge. A \textit{flip along $f$} is the move shown in figure \ref{flip move}.
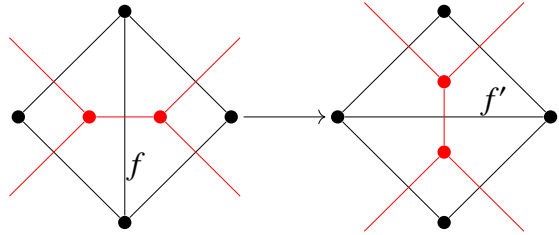
\begin{figure}[H]
    \def\t{6}
    \centering
    \resizebox{0.7\width}{!}{
    \begin{tikzpicture}
        
        \node[circle, fill=black, inner sep=0, minimum size=2.5mm] (p1) at (0,0) {};
        \node[circle, fill=black, inner sep=0, minimum size=2.5mm] (p2) at (2,2) {};
        \node[circle, fill=black, inner sep=0, minimum size=2.5mm] (p3) at (4,0) {};
        \node[circle, fill=black, inner sep=0, minimum size=2.5mm] (p4) at (2,-2) {};
        
        \node[circle, fill=black, inner sep=0, minimum size=2.5mm] (q1) at ($(p1) + (\t,0)$) {};
        \node[circle, fill=black, inner sep=0, minimum size=2.5mm] (q2) at ($(p2) + (\t,0)$) {};
        \node[circle, fill=black, inner sep=0, minimum size=2.5mm] (q3) at ($(p3) + (\t,0)$) {};
        \node[circle, fill=black, inner sep=0, minimum size=2.5mm] (q4) at ($(p4) + (\t,0)$) {};
        
        \draw (p2) -- (p4) node[near end, xshift=0.2cm] {\Large $f$}
              (p1) -- (p2) 
                   -- (p3) 
                   -- (p4) 
                   -- (p1);
                   
        \draw[-{To[scale=1.5]}] ($(p3.east) + (1mm, 0)$) -- ($(q1.west) + (-1mm, 0)$);
        
        \draw (q1) -- (q3) node[near end, yshift=0.3cm] {\Large $f'$}
              (q1) -- (q2) 
                   -- (q3) 
                   -- (q4) 
                   -- (q1); 
        
        \node[circle, fill=red, inner sep=0, minimum size=2.5mm] (d1) at (barycentric cs:p1=0.5,p2=0.5,p4=0.5) {};
        \node[circle, fill=red, inner sep=0, minimum size=2.5mm] (d2) at (barycentric cs:p3=0.5,p2=0.5,p4=0.5) {};
        \node[circle, fill=red, inner sep=0, minimum size=2.5mm] (d3) at (barycentric cs:q1=0.5,q2=0.5,q3=0.5) {};
        \node[circle, fill=red, inner sep=0, minimum size=2.5mm] (d4) at (barycentric cs:q1=0.5,q3=0.5,q4=0.5) {};
        
        \draw[red] (d1) -- (d2)
                   (d1) -- ++(-1.5,1.5)
                   (d1) -- ++(-1.5,-1.5)
                   (d2) -- ++(1.5,1.5)
                   (d2) -- ++(1.5,-1.5);
                   
        \draw[red] (d3) -- (d4)
                   (d3) -- ++(-1.5,1.5)
                   (d3) -- ++(1.5,1.5)
                   (d4) -- ++(-1.5,-1.5)
                   (d4) -- ++(1.5,-1.5);

    \end{tikzpicture}}
    \caption{Flip move along the edge $f$. The red lines show the dual fat graph. }
    \label{flip move}
\end{figure}
\end{defn}

In contrast to triangulations, there is no $3-1$ move for ideal triangulation, since we work with a fixed set of vertices. To be more precise, there exists a $2$-dimensional CW-complex $\Pcal(\Sigma,\delta,M)$, the Ptolemy-complex, describing ideal triangulations on $(\Sigma,\delta,M)$. Like the Lego-Teichm\"uller complex, the Ptolemy-complex is additional structure om $\Sigma$, which enables us to treat the surface $\Sigma$ in a combinatorial way.

\begin{defn}\label{Ptolemy definition} The \textit{Ptolemy complex} $\Pcal(\Sigma,\delta, M)$ is the $2$-dimensional CW-complex with 
\begin{itemize}
\item[\textbf{$0$-cells:}] Vertices of $\Pcal(\Sigma,\delta, M)$ are ideal triangulations.
\item[\textbf{$1$-cells:}] There is an edge between two vertices for any flip $F_e:\Tsf\rightarrow \Tsf^\prime$ as in figure \ref{flip move}.
\item[\textbf{$2$-cells:}] There are three types of $2$-cells:
\begin{itemize}
\item[]\textbf{P1:} For $f\in E(\Tsf)$ such that the flip $F_f$ exists, there is the $2$-cell
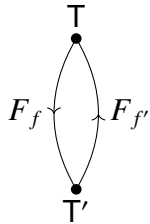
\begin{center}
\begin{tikzpicture}
\coordinate[label=above:{$\Tsf$}] (a) at (0,0);
\coordinate[label=below:{$\Tsf^\prime$}] (b) at (0,-2);

\fill (a) circle (2pt);
\fill (b) circle (2pt);

\draw[->-] (a) to [bend right] node [midway, left]{$F_f$} (b);
\draw[->-] (b) to [bend right] node [midway, right]{$F_{f^\prime}$} (a);
\end{tikzpicture}

The edge $f^\prime$ is the new edge appearing in the flip (see figure \ref{flip move}).
\end{center}
\item[]\textbf{P2:} For any two edges $e$, $f\in E(\Tsf)$ with disjoint endpoints 
such that $F_e$ and $F_f$ are defined, the flips commute, i.e. there is a quadrilateral $2$-cell
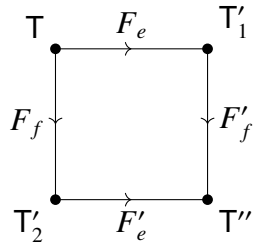
\begin{center}
\begin{tikzpicture}
\coordinate[label=above left:$\Tsf$] (a) at (0,0);
\coordinate[label=above right:$\Tsf_1^\prime$] (b) at (2,0);
\coordinate[label=below left:$\Tsf_2^{\prime}$] (c) at (0,-2);
\coordinate[label=below right:$\Tsf^{\prime\prime}$] (d) at (2,-2);

\fill (a) circle (2pt);
\fill (b) circle (2pt);
\fill (c) circle (2pt);
\fill (d) circle (2pt);

\draw[->-] (a) to node [midway, above]{$F_e$} (b);
\draw[->-] (a) to node [midway, left]{$F_f$} (c);
\draw[->-] (b) to node [midway, right]{$F^\prime_f$} (d);
\draw[->-] (c) to node [midway, below]{$F^\prime_e$} (d);
\end{tikzpicture}
\end{center}
The flips $F^\prime_e$, $F^\prime_f$ are the flips performed at the edges $e$, respectively $f$, which are now part of the new ideal triangulations $T^\prime_1$ and $T^\prime_2$.
\item[]\textbf{P3:} Given two edges $e$, $f\in E(\Tsf)$ sharing exactly one endpoint and $F_e$, $F_f$ exist, there is a pentagonal $2$-cell

\begin{figure}[H]
    \def\s{2}
    \centering
    \resizebox{\width}{!}{
    \begin{tikzpicture}
        
        \node[circle, fill=black, inner sep=0, minimum size=2.5mm] (a0) at (0,0) {};
        \node[circle, fill=black, inner sep=0, minimum size=2.5mm] (b0) at ($(a0) + (36:\s)$) {};
        \node[circle, fill=black, inner sep=0, minimum size=2.5mm] (c0) at ($(a0) + (-72:\s)$) {};
        \node[circle, fill=black, inner sep=0, minimum size=2.5mm] (d0) at ($(c0) + (\s,0)$) {};
        \node[circle, fill=black, inner sep=0, minimum size=2.5mm] (e0) at ($(d0) + (72:\s)$) {};
        
        \node[circle, fill=black, inner sep=0, minimum size=1.5mm] (d1) at ($(a0) + (144:\s)$) {};
        \node[circle, fill=black, inner sep=0, minimum size=1.5mm] (e1) at ($(d1) + (72:0.5*\s)$) {};
        \node[circle, fill=black, inner sep=0, minimum size=1.5mm] (c1) at ($(d1) + (-0.5*\s,0)$) {};
        \node[circle, fill=black, inner sep=0, minimum size=1.5mm] (a1) at ($(c1) + (108:0.5*\s)$) {};
        \node[circle, fill=black, inner sep=0, minimum size=1.5mm] (b1) at ($(a1) + (36:0.5*\s)$) {};
        
        \node[circle, fill=black, inner sep=0, minimum size=1.5mm] (c2) at ($(b0) + (-0.25*\s,\s)$) {};
        \node[circle, fill=black, inner sep=0, minimum size=1.5mm] (a2) at ($(c2) + (108:0.5*\s)$) {};
        \node[circle, fill=black, inner sep=0, minimum size=1.5mm] (b2) at ($(a2) + (36:0.5*\s)$) {};
        \node[circle, fill=black, inner sep=0, minimum size=1.5mm] (d2) at ($(c2) + (0.5*\s,0)$) {};
        \node[circle, fill=black, inner sep=0, minimum size=1.5mm] (e2) at ($(d2) + (72:0.5*\s)$) {};
        
        \node[circle, fill=black, inner sep=0, minimum size=1.5mm] (e3) at ($(c0) + (-144:\s)$) {};
        \node[circle, fill=black, inner sep=0, minimum size=1.5mm] (d3) at ($(e3) + (-108:0.5*\s)$) {};
        \node[circle, fill=black, inner sep=0, minimum size=1.5mm] (b3) at ($(e3) + (144:0.5*\s)$) {};
        \node[circle, fill=black, inner sep=0, minimum size=1.5mm] (c3) at ($(d3) + (-0.5*\s,0)$) {};
        \node[circle, fill=black, inner sep=0, minimum size=1.5mm] (a3) at ($(c3) + (108:0.5*\s)$) {};
        
        \node[circle, fill=black, inner sep=0, minimum size=1.5mm] (a4) at ($(d0) + (-36:\s)$) {};
        \node[circle, fill=black, inner sep=0, minimum size=1.5mm] (b4) at ($(a4) + (36:0.5*\s)$) {};
        \node[circle, fill=black, inner sep=0, minimum size=1.5mm] (c4) at ($(a4) + (-72:0.5*\s)$) {};
        \node[circle, fill=black, inner sep=0, minimum size=1.5mm] (d4) at ($(c4) + (0.5*\s,0)$) {};
        \node[circle, fill=black, inner sep=0, minimum size=1.5mm] (e4) at ($(d4) + (72:0.5*\s)$) {};
        
        \node[circle, fill=black, inner sep=0, minimum size=1.5mm] (c5) at ($(e0) + (36:\s)$) {};
        \node[circle, fill=black, inner sep=0, minimum size=1.5mm] (a5) at ($(c5) + (108:0.5*\s)$) {};
        \node[circle, fill=black, inner sep=0, minimum size=1.5mm] (b5) at ($(a5) + (36:0.5*\s)$) {};
        \node[circle, fill=black, inner sep=0, minimum size=1.5mm] (d5) at ($(c5) + (0.5*\s,0)$) {};
        \node[circle, fill=black, inner sep=0, minimum size=1.5mm] (e5) at ($(d5) + (72:0.5*\s)$) {};
        
        \draw[mid] (a0) -- (b0) node[midway, xshift=-0.3cm, yshift=0.3cm] {\Large $F_e$};
        \draw[mid] (b0) -- (e0) node[midway, xshift=0.3cm, yshift=0.3cm] {\Large $F_f^\prime$};
        \draw[mid] (e0) -- (d0) node[midway, xshift=0.3cm, yshift=-0.3cm] {\Large $F_{e'}^{\prime\prime}$};
        \draw[mid] (a0) -- (c0) node[midway, xshift=-0.3cm, yshift=-0.3cm] {\Large $F_f$};
        \draw[mid] (c0) -- (d0) node[midway, yshift=-0.4cm] {\Large $F_e^\prime$};
        
        \draw (a1) -- (b1)
                   -- (e1) 
                   -- (d1)
                   -- (c1)
                   -- (a1);
         \draw (a1) -- (d1) node[midway,xshift=0.25cm,yshift=0.15cm] {\small $f$};
         \draw    (a1) -- (e1) node[midway, yshift=0.15cm] {\small $e$};
        
        \draw (a2) -- (b2)
                   -- (e2)
                   -- (d2)
                   -- (c2)
                   -- (a2);
        \draw (d2) -- (a2) node[midway,xshift=0.1cm,yshift=0.2cm] {\small $f$};
         \draw  (d2) -- (b2) node[midway,xshift=0.25cm] {\small $e^\prime$};
        
        \draw (a3) -- (b3)
                   -- (e3)
                   -- (d3)
                   -- (c3)
                   -- (a3);
        \draw  (e3) -- (c3) node[midway, xshift=-0.2cm,yshift=0.1cm] {\small $f^\prime$};
        \draw  (e3) -- (a3) node[midway, yshift=0.2cm] {\small $e$};
        
        \draw (a4) -- (b4)
                   -- (e4)
                   -- (d4)
                   -- (c4)
                   -- (a4);
        \draw (c4) -- (e4) node[midway, xshift=-0.2cm, yshift=0.1cm] {\small $f^\prime$};
         \draw     (c4) -- (b4) node[midway, xshift=-0.2cm, yshift=0.1cm] {\small $h$};             
        
        \draw (a5) -- (b5)
                   -- (e5)
                   -- (d5)
                   -- (c5)
                   -- (a5);
         \draw  (b5) -- (c5) node[midway, xshift=-0.2cm] {\small $h$};
          \draw (b5) -- (d5) node[midway, xshift=0.2cm] {\small $e^\prime$};
        
    \end{tikzpicture}}
\end{figure}

\end{itemize}
\end{itemize} 
\end{defn}

\begin{theo}\label{Ptolemy c sc}\cite[Corollary~V.1.1.2]{PennerBook} The Ptolemy-complex $\Pcal(\Sigma,\delta, M)$ is connected and simply connected. 
\end{theo}

Connectedness is a classical result and can be proven by purely combinatorial methods. However, the proof that $\Pcal(\Sigma,\delta, M)$ is simply connected uses a fair amount of Teichm\"uller theory. The crucial point is that ideal triangulations, or fat-graphs, give a cell decomposition of the decorated Teichm\"uller space $\Tcal(\Sigma,\delta,M)$, which is a contractible space. In \cite[Chapter~5, Definition~1.1]{PennerBook} the \textit{Ptolemy-groupoid} is defined as the path groupoid of $\Tcal(\Sigma,\delta,M)$, thus it is the fundamental groupoid of the Ptolemy-complex.

So far we discussed the classical situation of just a surface with ideal triangulations. However, in this paper, we need a Ptolemy complex which also accounts for $G$-bundles over the surface. Thus, we introduce the notion of a $G$-Ptolemy-complex and show, that it is still connected and simply connected. In spirit this is a discrete realization of the holonomy functor for the $G$-bundle determined by $\zeta:\Sigma\rightarrow BG$.
  
\begin{defn} A \textit{$G$-labeled ideal triangulation} on $(\Sigma,\delta,M)$ is an isotopy class of an embedded oriented graph $\Tsf\hookrightarrow \Sigma$ together with a map $g:E^{or}(\Tsf)\rightarrow G$, such that
\begin{enumerate}[label=\roman*)]
\item the underlying graph of $\Tsf$ obtained by forgetting the orientation
is an ideal triangulation of $(\Sigma,\delta,M)$.
\item the map $g$ satisfies $g(\ebf)=g(\ov{\ebf})^{-1}$. 
\item if oriented edges $\ebf_1$, $\ebf_2$, $\ebf_3$ form the counterclockwise oriented boundary of a triangle of $\Tsf$ the following relation holds in $G$
\eq{\label{cyclic face}
g\lb \ebf_1\rb g\lb \ebf_2\rb g\lb \ebf_3\rb =e\quad .
}
\end{enumerate}
\end{defn}

The map $g$ is defined using oriented edges. Due to the condition in ii) giving the map on one orientation of the edge uniquely fixes the map on the edge with opposite orientation. 

\begin{defn} Given an edge $\fbf$ of a $G$-triangulation $\Tsf$, which is adjacent to two different $2$-cells the \textit{$G$-flip along $\fbf$} is given by the transformation shown in figure \ref{G move}.
\begin{figure}
\setlength{\tabcolsep}{1cm}

\centering
\begin{tabular}{ m{2cm} >{\centering\arraybackslash} m{1cm} m{2cm}}

\begin{tikzpicture}
\coordinate (a) at (0,0);
\coordinate (b) at (2,-2);
\coordinate (c) at (2,2);
\coordinate (d) at (4,0);
\coordinate (e) at (0,2);
\coordinate (f) at (0,-2);
\coordinate (g) at (1.4,0);
\coordinate (h) at (2.6,0);
\coordinate (i) at (4,2);
\coordinate (j) at (4,-2);

\draw[->-] (a) to node [midway, below]{$g_1$} (b);
\draw[->-] (a) to node [midway, above left]{$g_4$} (c);
\draw[->-] (b) to node [midway, below right]{$g_2$}(d);
\draw[->-] (c) to node [midway, above right]{$g_3$}(d);
\draw[->-](c) to node [midway, right]{$g$} (b);
\draw[->-](c) to node [midway, left]{$\fbf$} (b);

\fill (a) circle (2pt);
\fill (b) circle (2pt);
\fill (c) circle (2pt);
\fill (d) circle (2pt);
\end{tikzpicture}
 & $\phantom{---}\longrightarrow$ &

\begin{tikzpicture}
\coordinate (a) at (0,0);
\coordinate (b) at (2,-2);
\coordinate (c) at (2,2);
\coordinate (d) at (4,0);
\coordinate (e) at (0,2);
\coordinate (f) at (0,-2);
\coordinate (g) at (2,1);
\coordinate (h) at (2,-1);
\coordinate (i) at (4,2);
\coordinate (j) at (4,-2);

\draw[->-] (a) to node [midway, below]{$g_1$} (b);
\draw[->-] (a) to node [midway, above left]{$g_4$} (c);
\draw[->-] (b)to node [midway, below right]{$g_2$} (d);
\draw[->-] (c) to node [midway, above right]{$g_3$} (d);
\draw[->-] (a)to node [midway, above]{$g^\prime$}(d);
\draw[->-] (a)to node [midway, below]{$\fbf^\prime$}(d);

\fill (a) circle (2pt);
\fill (b) circle (2pt);
\fill (c) circle (2pt);
\fill (d) circle (2pt);
\end{tikzpicture}
\end{tabular}
\caption{Due to the cyclicity condition \eqref{cyclic face} the $G$-labels $g$ and $g^\prime$ are uniquely fixed by the $G$-labels $g_1$, $g_2$, $g_3$, $g_4$.}
\label{G move}
\end{figure}
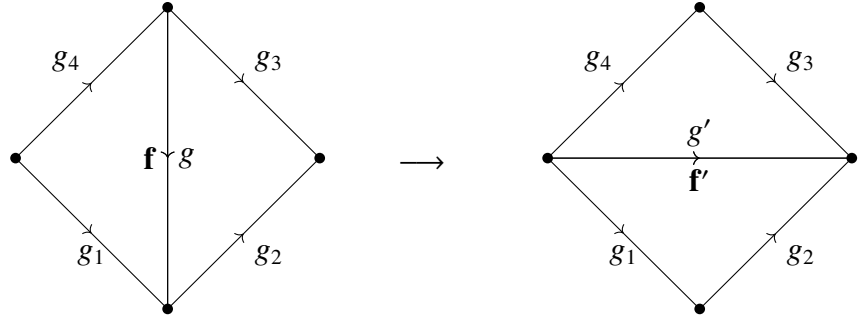
\end{defn}

\begin{rem} The edge $\fbf$ at which a $G$-flip is performed cannot be a boundary edge. 
\end{rem}

$G$-triangulations are a combinatorial tool to describe $1$-morphisms in the bicategory $\GBordpzc_\star(3,2,1)$. In $\GBordpzc_\star(3,2,1)$ all three layers of data are equipped with a map to $BG$. Equivalently one can regard $\GBordpzc_\star(3,2,1)$ as the bordism $2$-category of bordisms together with the datum of a $G$-principal bundle. To make this precise, recall that the \textit{mapping groupoid $\Pi(M,N)$}, for topological spaces $M$,$N$, has as objects continuous maps $M\xrightarrow{f}N$ and morphisms $f\rightarrow g$ are homotopy classes relative to $M\times \lbr 0,1\rbr $ of homotopies $f\sim_h g$. For $N=BG$ the classifying space of $G$, there is a canonical equivalence of groupoids 
\eq{
\Pi(M,BG)\simeq \PBunpzc_G(M)\, ,
} 
where $\PBunpzc_G(M)$ is the groupoid of $G$-principal bundles on $M$. By choosing a basepoint $\bullet\in M$ and restricting objects to pointed maps from $M$ to $BG$ but keeping as morphisms equivalence classes of unpointed homotopies, one gets an equivalent groupoid $\Pi_\star(M,BG)$. So from the pointed map $f:(M,\bullet)\rightarrow (BG,\star)$ we get that objects of $\GBordpzc_\star(3,2,1)$ can be described as unions of circles with fixed $G$-principal bundles. The underlying manifold of a connected object $((M,\bullet),f)$ is diffeomorphic to a circle $S^1$ and its classifying map determines a homomorphism $f_\ast:\pi_1(S^1,\bullet)\simeq \Zbb \rightarrow G\simeq \pi_1(BG,\star)$, i.e. an element $f_\ast(1)\eqqcolon g$ in $G$. In this sense, we can identify objects of $\GBordpzc_\star(M,BG)$ with finitely
many circles which are labeled by a group element.

Let $(\Sigma,\zeta):((M_0,\bullet),f_0)\rightarrow ((M_1,\bullet),f_1)$. The surface $\Sigma$ itself is not pointed, however, if $\ast\in M_0$ is a basepoint of a connected component, it holds $\zeta(\iota_0(\ast))=f_0(\ast)=\star$. Thus, the images of the basepoints in the boundary of $\Sigma$ all get mapped to the basepoint in $BG$ by $\zeta$. 

To the data of a $1$-morphism $(\Sigma,\zeta):((M_0,m_0),f_0)\rightarrow ((M_1,m_1),f_1)$ we want to associate a $G$-triangulation of $\Sigma$. We begin by choosing a finite set of points $K\subset \Sigma\backslash \p \Sigma$, which for $\Sigma$ homeomorphic to a disk needs to be non-empty. Let $\delta\coloneqq \iota_0(m_0)\cup \iota_1(m_1)\in \p \Sigma$ be the image of the basepoints and $\Tsf$ an ideal triangulation based at $\delta\cup K$. Let $b\subset \p \Sigma$ be a connected component of the boundary. Assume that the circle which is mapped to $b$ is labeled by $g\in G$. By construction $\zeta_\ast(b)\in \pi_1(BG,\star)$ is the loop corresponding to $g$. Thus, $G$-labels of boundary edges in $\Tsf$ are uniquely fixed by source and target objects of the $1$-morphism. Since all images of basepoints in $\Sigma$ get mapped to $\star$, the image of a non-boundary edge $\ebf\in E^{or}(\Tsf)$ with both endpoints in $\delta$ is an oriented loop in $\pi_1(BG,\star)$, hence determines a group element $g_\ebf$. The $G$-color of all other edges is determined up to certain \textit{gauge transformations}. Let $p=(\ebf_1,\cdots, \ebf_r)$ be an oriented edge path in $\Tsf$ with endpoints in $\delta$. The path $p$ gets mapped to a loop in $BG$, based at $\star$. Hence there exists $g_p\in G$, such that $\mathrm{Im}(\zeta|_p)\simeq g_p$ holds. We can assume that only the first and last edge of the path have endpoints in $\delta$. To each edge $\ebf_i\in p$ we assign a group element $g_i$, such that
\eq{\label{edge path relation}
g_1g_2\cdots g_r=g_p\quad. 
} 
This coloring is not unique. We can change $g_i\mapsto g_ih$ and $g_{i+1}\mapsto h^{-1}g_{i+1}$ and the new color still satisfies \eqref{edge path relation}. These are so called \textit{gauge transformations}. A specific labeling of the oriented edges of $\Tsf$ with group elements is \textit{compatible} with $\zeta$, if \eqref{edge path relation} holds for all possible edge paths with endpoints in $\delta$. The set of all $G$-colorings of $\Tsf$ which are compatible with $\zeta$ is denoted by $\mathrm{Col}_G(\Tsf)$. A gauge transformation then is a map
\eq{
\lambda:M\rightarrow \mathrm{Maps}\lb G,\Endrm\lb\mathrm{Col}_G(\Tsf)\rb \rb,
}
where for $p\in M$ and $g\in G$, $\lambda_g(p)$ acts as 

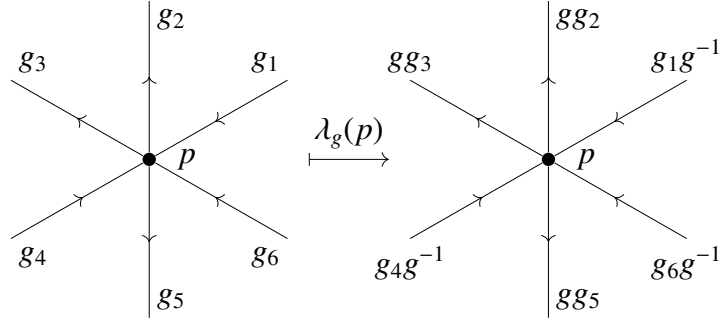
\begin{figure}[H]
    \def\l{3}
    \makebox[\textwidth][c]{
    \resizebox{0.7\width}{!}{
    \begin{tikzpicture}
        \node[circle, 
              fill=black, 
              inner sep=0, 
              minimum size=2.5mm, 
              label={[label distance=0.25cm]0:\Large $p$}] 
              (p) at (0,0) {};
        \draw[mid={<}{0.5}] (p) -- ++(30:\l) node[above left] {\Large $g_1$};
        \draw[mid={>}{0.5}] (p) -- ++(90:\l) node[below right] {\Large $g_2$};
        \draw[mid={>}{0.5}] (p) -- ++(150:\l) node[above right] {\Large $g_3$};
        \draw[mid={<}{0.5}] (p) -- ++(210:\l) node[below right] {\Large $g_4$};
        \draw[mid={>}{0.5}] (p) -- ++(270:\l) node[above right] {\Large $g_5$};
        \draw[mid={<}{0.5}] (p) -- ++(330:\l) node[below left] {\Large $g_6$};
        
        \node[circle, 
              fill=black, 
              inner sep=0, 
              minimum size=2.5mm, 
              label={[label distance=0.25cm]0:\Large $p$}]
              (q) at (2.5*\l,0) {};
        \draw[mid={<}{0.5}] (q) -- ++(30:\l) node[above] {\Large $g_1 g^{-1}$};
        \draw[mid={>}{0.5}] (q) -- ++(90:\l) node[below right] {\Large $g g_2$};
        \draw[mid={>}{0.5}] (q) -- ++(150:\l) node[above] {\Large $g g_3$};
        \draw[mid={<}{0.5}] (q) -- ++(210:\l) node[below] {\Large $g_4 g^{-1}$};
        \draw[mid={>}{0.5}] (q) -- ++(270:\l) node[above right] {\Large $g g_5$};
        \draw[mid={<}{0.5}] (q) -- ++(330:\l) node[below] {\Large $g_6 g^{-1}$};
        
        \draw[{|[scale=1.5]}-{>[scale=1.5]}] ($(p) + (\l,0)$) -- ($(q) + (-\l,0)$) node[midway, yshift=0.5cm] {\Large $\lambda_g(p)$};
        
    \end{tikzpicture}}}
\end{figure}

Multiplication in $G$ endows the set of gauge transformations with a group structure such that the gauge group acts transitively on $\mathrm{Col}_G(\Tsf)$. To summarize, for any $1$-morphism $(\Sigma,\zeta)$ in $\GBordpzc_\star(3,2,1)$ and any ideal triangulation $\Tsf$ of $\Sigma$, based at the images of the basepoints and a possibly empty set $M$, we get a $G$-triangulation, unique up to gauge transformation.

\begin{defn}\label{Ptolemy} Let $(\Sigma,\zeta):(M_0,f_0)\rightarrow (M_1,f_0)$ be $1$-morphism in $\GBordpzc_\star(3,2,1)$ and $M\subset \Sigma\backslash \p \Sigma$ a finite set of points. The \textit{$G$-equivariant Ptolemy groupoid $\Pcal^G(\Sigma,\zeta,M)$} is a $2$ dimensional CW-complex defined as follows.
\begin{itemize}
\item[\textbf{0-cells:}] Vertices of $\Pcal^G(\Sigma,\zeta, M)$ are (ideal) $G$-triangulations induced by $\zeta$.
\item[\textbf{1-cells:}] There is an edge $\Tsf\rightarrow \Tsf^\prime$ in $\Pcal^G(\Sigma,\zeta,M)$ for any $G$-flip and any gauge transformation.
\item[\textbf{2-cells:}] Since the $G$-label for $G$-flips is uniquely determined, we lift the relations \textbf{P1-P3} of the Ptolemy-groupoid to $\Pcal^G(\Sigma,\zeta,M)$. This yields three relations \textbf{GP1}, \textbf{GP2} and \textbf{GP3}. In addition there is a mixed relation as well as two pure gauge relations. 
\begin{itemize}
\item[]\textbf{GP4:} For $\ebf$ an edge of a G-triangulation and $v$ a vertex incident to $\ebf$, the flip along  $\ebf$ and a gauge transformation at $v$ commute. That is, there is a quadrilateral $2$-cell
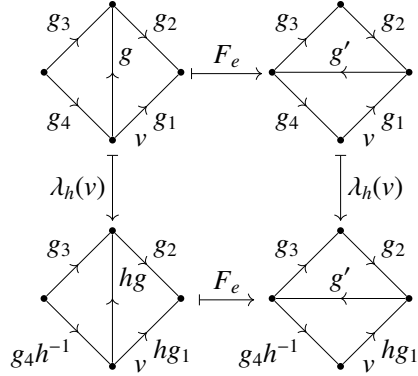
\begin{figure}[H]
    \def\l{5}
    \def\s{1}
    \makebox[\textwidth][c]{
    \resizebox{0.6\width}{!}{
    \begin{tikzpicture}
        
        \node at (0,0) (a) {
        \resizebox{\s\width}{!}{
        \begin{tikzpicture}
            \node[circle, fill=black, inner sep=0, minimum size=1.5mm] (a1) at (0,0) {};
            \node[circle, fill=black, inner sep=0, minimum size=1.5mm] (a2) at ($(a1) + (1.5,1.5)$) {};
            \node[circle, fill=black, inner sep=0, minimum size=1.5mm] (a3) at ($(a2) + (1.5,-1.5)$) {};
            \node[circle, fill=black, inner sep=0, minimum size=1.5mm, label={[label distance=0.25cm]0:\Large $v$}] (a4) at ($(a3) + (-1.5,-1.5)$) {};
            \draw[mid] (a1) -- (a2) node[anchor=south east, midway] {\Large $g_3$};
            \draw[mid] (a1) -- (a4) node[anchor=north east, midway] {\Large $g_4$};
            \draw[mid] (a2) -- (a3) node[anchor=south west, midway] {\Large $g_2$};
            \draw[mid] (a4) -- (a3) node[anchor=north west, midway] {\Large $g_1$};
            \draw[mid] (a4) -- (a2) node[anchor=south west, midway] {\Large $g$};
        \end{tikzpicture}}
        };
        
        \node at (\l,0) (b) {
        \resizebox{\s\width}{!}{
        \begin{tikzpicture}
            \node[circle, fill=black, inner sep=0, minimum size=1.5mm] (b1) at (0,0) {};
            \node[circle, fill=black, inner sep=0, minimum size=1.5mm] (b2) at ($(b1) + (1.5,1.5)$) {};
            \node[circle, fill=black, inner sep=0, minimum size=1.5mm] (b3) at ($(b2) + (1.5,-1.5)$) {};
            \node[circle, fill=black, inner sep=0, minimum size=1.5mm, label={[label distance=0.25cm]0:\Large $v$}] (b4) at ($(b3) + (-1.5,-1.5)$) {};
            \draw[mid] (b1) -- (b2) node[anchor=south east, midway] {\Large $g_3$};
            \draw[mid] (b1) -- (b4) node[anchor=north east, midway] {\Large $g_4$};
            \draw[mid] (b2) -- (b3) node[anchor=south west, midway] {\Large $g_2$};
            \draw[mid] (b4) -- (b3) node[anchor=north west, midway] {\Large $g_1$};
            \draw[mid] (b3) -- (b1) node[anchor=south, midway] {\Large $g'$};
        \end{tikzpicture}}
        };

        \node at ($(\l,-\l) + (-0.25,0)$) (c) {
        \resizebox{\s\width}{!}{
        \begin{tikzpicture}
            \node[circle, fill=black, inner sep=0, minimum size=1.5mm] (c1) at (0,0) {};
            \node[circle, fill=black, inner sep=0, minimum size=1.5mm] (c2) at ($(c1) + (1.5,1.5)$) {};
            \node[circle, fill=black, inner sep=0, minimum size=1.5mm] (c3) at ($(c2) + (1.5,-1.5)$) {};
            \node[circle, fill=black, inner sep=0, minimum size=1.5mm, label={[label distance=0.25cm]0:\Large $v$}] (c4) at ($(c3) + (-1.5,-1.5)$) {};
            \draw[mid] (c1) -- (c2) node[anchor=south east, midway] {\Large $g_3$};
            \draw[mid] (c1) -- (c4) node[anchor=north east, midway] {\Large $g_4 h^{-1}$};
            \draw[mid] (c2) -- (c3) node[anchor=south west, midway] {\Large $g_2$};
            \draw[mid] (c4) -- (c3) node[anchor=north west, midway] {\Large $h g_1$};
            \draw[mid] (c3) -- (c1) node[anchor=south, midway] {\Large $g'$};
        \end{tikzpicture}}

        };
        \node at ($(0,-\l) + (-0.25,0)$) (d) {
        \resizebox{\s\width}{!}{
        \begin{tikzpicture}
            \node[circle, fill=black, inner sep=0, minimum size=1.5mm] (d1) at (0,0) {};
            \node[circle, fill=black, inner sep=0, minimum size=1.5mm] (d2) at ($(d1) + (1.5,1.5)$) {};
            \node[circle, fill=black, inner sep=0, minimum size=1.5mm] (d3) at ($(d2) + (1.5,-1.5)$) {};
            \node[circle, fill=black, inner sep=0, minimum size=1.5mm, label={[label distance=0.25cm]0:\Large $v$}] (d4) at ($(d3) + (-1.5,-1.5)$) {};
            \draw[mid] (d1) -- (d2) node[anchor=south east, midway] {\Large $g_3$};
            \draw[mid] (d1) -- (d4) node[anchor=north east, midway] {\Large $g_4 h^{-1}$};
            \draw[mid] (d2) -- (d3) node[anchor=south west, midway] {\Large $g_2$};
            \draw[mid] (d4) -- (d3) node[anchor=north west, midway] {\Large $h g_1$};
            \draw[mid] (d4) -- (d2) node[anchor=south west, midway] {\Large $h g$};
        \end{tikzpicture}}
        };
        
        \draw[{|[scale=1.5]}-{>[scale=1.5]}] ($(a.east) + (0,0.075)$) -- ($(b.west) + (0.1,0.075)$) node[anchor=south, midway] {\Large $F_e$};
        \draw[{|[scale=1.5]}-{>[scale=1.5]}] ($(a.south) + (0.05,0.1)$) -- ($(d.north) + (0.3,0)$) node[anchor=east, midway] {\Large $\lambda_h(v)$};
        \draw[{|[scale=1.5]}-{>[scale=1.5]}] ($(d.east) + (-0.125,0.075)$) -- ($(c.west) + (0.725,0.075)$) node[anchor=south, midway] {\Large $F_e$};
        \draw[{|[scale=1.5]}-{>[scale=1.5]}] ($(b.south) + (0.05,0.1)$) -- ($(c.north) + (0.3,0)$) node[anchor=west, midway] {\Large $\lambda_h(v)$};
        
    \end{tikzpicture}}}
    \caption{Flip commutes with gauge transformation}
\end{figure}

\item[]\textbf{GP5:} Gauge transformations at different vertices commute, i.e. for $v\neq u$ vertices of a $G$-triangulation, there are quadrilateral $2$-cells corresponding to $\lambda(v)\lambda(u)=\lambda(u)\lambda(v)$. 
\item[]\textbf{GP6:} Finally we add triangular $2$-cells $\lambda_g(v)\cdot\lambda_h(v)=\lambda_{gh}(v)$, implementing the gauge-action.
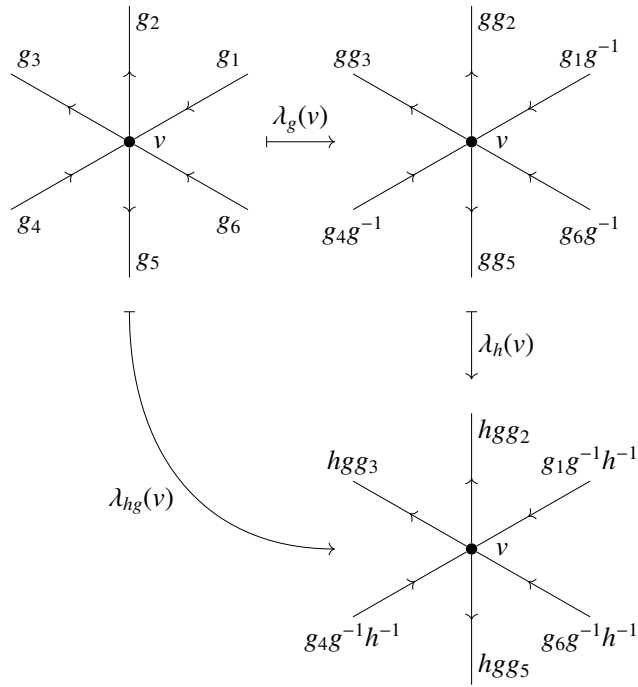
\begin{figure}[H]
    \def\l{3}
    \makebox[\textwidth][c]{
    \resizebox{0.6 \width}{!}{
    \begin{tikzpicture}
        \node[circle, 
              fill=black, 
              inner sep=0, 
              minimum size=2.5mm, 
              label={[label distance=0.25cm]0:\Large $v$}] 
              (p) at (0,0) {};
        \draw[mid={<}{0.5}] (p) -- ++(30:\l) node[above left] {\Large $g_1$};
        \draw[mid={>}{0.5}] (p) -- ++(90:\l) node[below right] {\Large $g_2$};
        \draw[mid={>}{0.5}] (p) -- ++(150:\l) node[above right] {\Large $g_3$};
        \draw[mid={<}{0.5}] (p) -- ++(210:\l) node[below right] {\Large $g_4$};
        \draw[mid={>}{0.5}] (p) -- ++(270:\l) node[above right] {\Large $g_5$};
        \draw[mid={<}{0.5}] (p) -- ++(330:\l) node[below left] {\Large $g_6$};
        
        \node[circle, 
              fill=black, 
              inner sep=0, 
              minimum size=2.5mm, 
              label={[label distance=0.25cm]0:\Large $v$}]
              (q) at (2.5*\l,0) {};
        \draw[mid={<}{0.5}] (q) -- ++(30:\l) node[above] {\Large $g_1 g^{-1}$};
        \draw[mid={>}{0.5}] (q) -- ++(90:\l) node[below right] {\Large $g g_2$};
        \draw[mid={>}{0.5}] (q) -- ++(150:\l) node[above] {\Large $g g_3$};
        \draw[mid={<}{0.5}] (q) -- ++(210:\l) node[below] {\Large $g_4 g^{-1}$};
        \draw[mid={>}{0.5}] (q) -- ++(270:\l) node[above right] {\Large $g g_5$};
        \draw[mid={<}{0.5}] (q) -- ++(330:\l) node[below] {\Large $g_6 g^{-1}$};
        
        \node[circle, 
              fill=black, 
              inner sep=0, 
              minimum size=2.5mm, 
              label={[label distance=0.25cm]0:\Large $v$}]
              (r) at (2.5*\l,-3*\l) {};
        \draw[mid={<}{0.5}] (r) -- ++(30:\l) node[above] {\Large $g_1 g^{-1} h^{-1}$};
        \draw[mid={>}{0.5}] (r) -- ++(90:\l) node[below right] {\Large $h g g_2$};
        \draw[mid={>}{0.5}] (r) -- ++(150:\l) node[above] {\Large $h g g_3$};
        \draw[mid={<}{0.5}] (r) -- ++(210:\l) node[below] {\Large $g_4 g^{-1} h^{-1}$};
        \draw[mid={>}{0.5}] (r) -- ++(270:\l) node[above right] {\Large $h g g_5$};
        \draw[mid={<}{0.5}] (r) -- ++(330:\l) node[below] {\Large $g_6 g^{-1} h^{-1}$};

        \draw[{|[scale=1.5]}-{>[scale=1.5]}] ($(p) + (\l,0)$) -- ($(q) + (-\l,0)$) node[midway, yshift=0.5cm] {\Large $\lambda_g(v)$};
        \draw[{|[scale=1.5]}-{>[scale=1.5]}] ($(q) + (0,-1.25*\l)$) -- ($(r) + (0,1.25*\l)$) node[anchor=west, midway] {\Large $\lambda_h(v)$};
        
        \draw[{|[scale=1.5]}-{To[scale=1.5]}] ($(p) + (0,-1.25*\l)$) .. controls ++(0,-\l) and ++(-\l,0) .. ($(r) + (-\l,0)$) node[anchor=north east, midway] {\Large $\lambda_{hg}(v)$};
                
    \end{tikzpicture}}}
\end{figure}

\end{itemize}
\end{itemize} 
\end{defn}

\begin{theo}\label{G scc} The $G$-equivariant Ptolemy groupoid $\Pcal^G(\Sigma,\zeta, M)$ from Definition \ref{Ptolemy} is connected and simply connected. 
\end{theo}

We defer the proof of Theorem \ref{G scc} to the appendix, since it doesn't offer any deeper insight for the main results of this paper.

\section{Bare String-Net Spaces and Cylinder Categories}\label{section 5}

\subsection{A vector space for surfaces}

We have set the stage for constructing a $G$-equivariant string-net space for $1$-morphisms. The main idea is to start with an ideal $G$-triangulation $\Tsf$ on $(\Sigma,\zeta)$ and define a string-net space relative to it by only allowing graphs transversal to $\Tsf$. We proceed by subsequently taking quotients implementing the rules of the graphical calculus for $\Ccal$ inside disks having different relative positions to $\Tsf$. This will yield string-net spaces $\SNrm^\Ccal_\Tsf(\Sigma,\zeta)$ for any $G$-triangulation $\Tsf$. In order to get a space which solely depends on $(\Sigma,\zeta)$, we define for any edge of $\Pcal^G(\Sigma,\zeta)$ an isomorphism between string-nets spaces. These maps will satisfy the relations \textbf{GP1-6} and we get a projective system of vector spaces. Taking the limit of the system will ultimately give the string-net space. 

At first sight, this might appear as an overly cumbersome way of defining a string-net space on a $G$-surface. In the usual string-net approach one just considers all embedded $\Dcal$-colored graphs, for $\Dcal$ any spherical fusion category. However, there seems to be no way of deciding whether an arbitrary $\Ccal$-colored, embedded graph on $\Sigma$ is compatible with the $G$-structure on $\Sigma$. This appears to be only possible if the graph is fine enough, meaning that it is at least isotopic to the $1$-skeleton of a CW-decomposition of $\Sigma$. In other words, in the equivariant case string-net graphs need to be sensitive to the global topology of $\Sigma$. Instead of arbitrary CW-decompositions we work with ideal triangulations, because there we have full control over the combinatorics involved. Without further ado we now spell out the details of the construction.

We choose an arbitrary finite, but fixed set of points $M\subset \Sigma\backslash \p \Sigma$. Note that $M$ can be empty except for the disk. 
\begin{defn}
Let $\Gamma$ be an arbitrary embedded finite oriented graph in $\Sigma$ and $\Ccal$ a $G$-graded fusion category. A \textit{$\Ccal$-coloring of $\Gamma$} comprises two functions. First
\eq{
c:E^{\orrm}(\Gamma)&\rightarrow \Ccal^{\mathrm{hom}}_0
}
assigns to each oriented edge $(e,\mathrm{or})$ an homogeneous object of $\Ccal$, such that $c(\ebf)=c(\ov{\ebf})^\ast$. The second function is a map
\eq{
\phi:V(\Gamma)&\rightarrow \Ccal_1 \\
v&\mapsto \phi_v\in  \Ccal\left(\bigotimes_{\ebf\in H(v)}c(\ebf)^{\epsilon_\ebf}\right),
}
where $\epsilon_\ebf=1$ if $\ebf$ is oriented away from $v$ and $\epsilon_\ebf=-1$ if it is oriented towards $v$. A negative exponent for an object in $c\in \Ccal$ indicates its dual.
\end{defn}
In particular via the grading map $p:\Ccal^{\mathrm{hom}}_0\rightarrow G$ we get a $G$-labeling $g_\Gamma$ of $\Gamma$ satisfying $g_\Gamma(\ov{\ebf})=g_\Gamma(\ebf)^{-1}$. The graph and its $G$-labeling should be compatible with a $G$-triangulation in a sense we are about to define (cf. Definition \ref{G compatible}).

\begin{defn}\label{G-transversal} Let $\Tsf\in \Pcal^G(\Sigma,\zeta,M)$. An embedded graph $\Gamma\hookrightarrow \Sigma$ is \textit{totally transversal} to $\Tsf$ if $V(\Tsf)\cap \Gamma=\Tsf\cap V(\Gamma)=\emptyset$ and any edge of $\Tsf$ that is not on the boundary intersects at least one edge of $\Gamma$. Furthermore, all intersections are transversal.
\end{defn}

The obvious example of a graph transversal to an ideal triangulation is its dual trivalent fat graph. Let $\Gamma$ be a $\Ccal$-colored graph, which is totally transversal to the underlying ideal triangulation of a $G$-triangulation $\Tsf$. Pick a representative in the isotopy class for $\Tsf$. For an edge $\ebf\in E^{\mathrm{or}}(\Tsf)$ consider the edges $\lbr \fbf_1,\cdots , \fbf_n\rbr $ of $\Gamma$ intersecting with $\ebf$, as in figure \ref{G compat}. The edge $\fbf_i$ has $G$-color $g_i$. The orientation of $\ebf$ gives a linear order on the set of intersecting edges. The transversal intersection can be positive or negative, depending on the orientation of the $\fbf_i$, e.g. the intersection $\ebf\cap \fbf_1$ is positive, whereas $\ebf\cap \fbf_{n-1}$ is negative in figure \ref{G compat}. The linear order allows us to multiply the $G$-labels of the the $\fbf_i$'s taking orientations into account. For the graph in figure \ref{G compat} this gives 
\eq{
G\ni m_{\ebf}=g_1\cdots g_{n-1}^{-1}g_n\quad. 
} 
It is easy to check that $m_\ebf$ is well defined, i.e. independent of the chosen representative in the isotopy class for $\Tsf$.
\begin{defn}\label{G compatible} Let $\Gamma$ be a $\Ccal$-colored graph which is totally transversal to the underlying ideal triangulation $G$-triangulation $\Tsf$. Then $\Gamma$ is \textit{$G$-transversal} to $\Tsf$ if 
\eq{
m_{\ebf}=g(\ebf)
}
holds for any edge of $\Tsf$.
\end{defn}
If the $G$-triangulation is clear from the context we sometimes just say that an embedded graph is $G$-transversal. 

We are working on surfaces with boundary, therefore string-net spaces will depend on a choice of boundary condition, or boundary value.

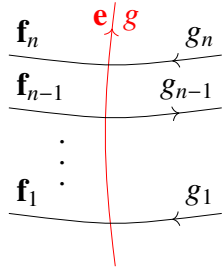
\begin{figure}
    \centering
    \makebox[\textwidth][c]{
    \resizebox{0.7 \width}{!}{
    \begin{tikzpicture}
        \draw[mid={>}{0.9}, red] (3,0) .. controls ++(-0.25,2) and ++(-0.25,-2) .. (3,5) node[anchor=west, below right] {\Large $g$} node[anchor=east, below left] {\Large $\mathbf{e}$};
        \draw[mid={<}{0.8}] (1,1) node[anchor=south, above right] {\Large $\mathbf{f}_1$} .. controls ++(2,-0.25) and ++(-2,-0.25) .. (5,1)  node[anchor=south, above left] {\Large $g_1$};
        \draw[mid={>}{0.8}] (1,3) node[anchor=south, above right] {\Large $\mathbf{f}_{n-1}$} .. controls ++(2,-0.25) and ++(-2,-0.25) .. (5,3) node[anchor=south, above left] {\Large $g_{n-1}$};
        \draw[mid={<}{0.8}] (1,4) node[anchor=south, above right] {\Large $\mathbf{f}_n$} .. controls ++(2,-0.25) and ++(-2,-0.25) .. (5,4) node[anchor=south, above left] {\Large $g_n$};
        \node at (2,1.9) {\Large $\cdot$};
        \node at (2,1.5) {\Large $\cdot$};
        \node at (2,2.3) {\Large $\cdot$};
    \end{tikzpicture}
    }
    }
    \caption{In the color version, the orange line represents an edge of $\Tsf$ with $G$-label $g$. The black lines are edges $(\Gamma,c)$}
    \label{G compat}
\end{figure}
 
\begin{defn} 
\begin{enumerate}[label=\arabic*)]
\item A \textit{boundary value} for a surface $(\Sigma,\zeta)$ is a disjoint union of finitely many points $B\in \p\Sigma$ together with a map $\Bbf:B\rightarrow \Ccal_0^{hom}$.
\item The \textit{boundary value} of a $G$-transversal graph $\Gamma$ on $\Sigma$ is defined as the disjoint union of intersection points $B_\Gamma$ of the graph with $\p \Sigma$, together with the map $\Bbf:B_\Gamma\rightarrow \Ccal_0^{hom}$, mapping an intersection point to the $\Ccal$-color of its corresponding edge.
\end{enumerate}
\end{defn}

\begin{rem} Note that one can state the definition for a boundary value in case of surfaces without the map $\zeta$ in the exact same fashion. $\zeta$ enters the definition only in so far, as restricting the possible boundary values, as $G$-labels of edges have to compatible with $\zeta$.
\end{rem}
  
\begin{defn} Let $\VGraphrm^\Ccal_{G\Tsf}(\Sigma,\zeta,M;\Bbf)$ be the $\Kbb$-vector space freely generated by all graphs which are $G$-transversal to $\Tsf$ with boundary value $\Bbf$.
\end{defn}

\begin{rem}The vector space $\VGraphrm^\Ccal_{G\Tsf}(\Sigma,\zeta,M;\Bbf)$ is huge, e.g. isotopic graphs with the same $\Ccal$-coloring are considered as different graphs so far. Isotopy invariance will follow only after requiring the graphical calculus for $\Ccal$ to hold locally on disks.
\end{rem}

In a first step, we consider embedded closed disks $D\hookrightarrow\Sigma\backslash \Tsf$. Given a $G$-transversal graph $\Gamma$ the boundary of the disk $D$ is required to intersect edges of $\Gamma$ transversally and mustn't intersect $\Gamma$ at vertices of $\Gamma$. As usual we get a cyclic set of objects $\lbr c_i\rbr$ from the $\Ccal$-color of the edges intersecting $\p D$ and a linear evaluation map \cite[Theorem~2.3]{kirillov2011string}
\eq{
\la \bullet\ra_D:\VGraphrm^\Ccal_{G\Tsf}\lb\Sigma,\zeta,M;\Bbf\rb\rightarrow \Ccal\lb\bigotimes_i c_i^{\epsilon_i}\rb\quad.
}
which is defined on vectors meeting the transversality requirements with
respect to $D$. The sign conventions are obvious from figure \ref{Eval disk}.

\begin{figure}
    \centering
    \makebox[\textwidth][c]{
    \resizebox{0.7 \width}{!}{
    \begin{tikzpicture}
        \node[circle, draw, inner sep=5mm] (a) at (0,0) {\Large $\gamma$};
        \draw[cyan, mid={>}{0}] (0,0) circle (2.5cm) node[anchor=east, xshift=-2.6cm] {\Large $\partial D$};
        \draw[mid] (a.south east) .. controls ++(0.5,-0.5) and ++(-0.5,0.5) .. ++(1,-2) node[anchor=south west, midway] {\Large $c_1$};
        \draw[mid] (a.north east) .. controls ++(0.5,0.5) and ++(0,-1) .. ++(1,2) node[anchor=north west, midway] {\Large $c_2$};
        \draw[mid={<}{0.5}] (a.north) .. controls ++(0.1,0.25) and ++(0.25,-0.25) .. ++(0,2) node[anchor=south west, midway] {\Large $c_3$};
        \draw[mid={>}{0.5}] (a.north west) .. controls ++(-0.1,0.25) and ++(0,-0.5) .. ++(-1,2) node[anchor=north east, midway] {\Large $c_4$};
        \draw[mid={<}{0.5}] (a.south west) .. controls ++(-0.1,0.25) and ++(0,0.5) .. ++(-1,-2) node[anchor=south east, midway] {\Large $c_5$};
        
        \draw[{|[scale=1.5]}-{>[scale=1.5]}] ($(a) + (0,-3)$) -- ++(0,-1) node[anchor=west, midway] {\Large $\langle \cdot \rangle_D$};
        
        \node[anchor=north] at ($(a) + (0,-4.5)$) {\Large $\displaystyle{\langle \gamma \rangle_D \in Hom_{\Ccal}\left(\onebb, c_1 \otimes c_2 \otimes c_3^* \otimes c_4 \otimes c_5^*\right)}$};
        
    \end{tikzpicture}
    }
    }
    \caption{$\gamma$ is the subgraph inside $D$ of a $G$-transversal graph $\Gamma$.}
    \label{Eval disk}
\end{figure}

\begin{defn} Let $D\hookrightarrow \Sigma\backslash \Tsf$ be an embedded disk. An element $\Gamma\coloneqq \sum_i x_i\Gamma_i\in \VGraphrm^\Ccal_{G\Tsf}(\Sigma,\zeta,M;\Tsf)$ is \textit{null with respect to $D$} if 
\begin{itemize}
\item all $\Gamma_i$'s are transversal to $D$.
\item $\Gamma_i|_{\Sigma\backslash D}=\Gamma_j|_{\Sigma\backslash D}$ for all $i,j$
\item $\la \Gamma\ra_D=0$.
\end{itemize}

The \textit{vector space of $\Tsf$-disk null graphs $\NGraphrm^\Ccal_{G\Tsf}(\Sigma,\zeta,M;\Bbf)$} is the subspace of $\VGraphrm^\Ccal_{G\Tsf}(\Sigma,\zeta,M;\Bbf)$ spanned by all null graphs for all possible disks $D\hookrightarrow \Sigma\backslash \Tsf$.
\end{defn}

The quotient of $\VGraphrm^\Ccal_{G\Tsf}(\Sigma,\zeta,M;\Bbf)$ by the vector space of $\Tsf$-disk null graphs $\NGraphrm^\Ccal_{G\Tsf}(\Sigma,\zeta,M;\Bbf)$ is denoted
\eq{
\SNrm_{\Tsf,\Ccal}^{\prime\prime}(\Sigma,M;\Bbf)\coloneqq \frac{\VGraphrm^\Ccal_{G\Tsf}(\Sigma,\zeta,M;\Bbf)}{\NGraphrm^\Ccal_{G\Tsf}(\Sigma,\zeta,M;\Bbf)}\quad .
}
\begin{rem} \label{remark cyclic} Due to the $G$-grading of $\Ccal$, non-zero vectors in $\SNrm_{\Tsf,\Ccal}^{\prime\prime}(\Sigma,M;\Bbf)$ satisfy a cyclicity condition at all of their vertices. That is, given $v\in V(\Gamma)$, the $\Ccal$-color of its incident edges $\lbr \ebf_i\rbr =E^{\mathrm{or}}(v)$ has to satisfy 
\eq{\label{cyclic}
p\lb c\lb \ebf_1\rb^{\epsilon_1}\rb \cdots p\lb c \lb \ebf_n^{ \epsilon_n}\rb \rb=e\quad.
}
There is a sign convention involving the orientation of the incident edges, which can be easily deduced from figure \ref{cyclic figure} and is similar to the one used in figure \ref{Eval disk}.

\begin{figure}
    \def \l{3}
    \centering
    \makebox[\textwidth][c]{
    \resizebox{0.7 \width}{!}{
    \begin{tikzpicture}
        \node[circle, 
              fill=black, 
              inner sep=0, 
              minimum size=2.5mm,
              label={[label distance=-1cm]0:\Large $v$}]
              (p) at (0,0) {};

        \draw[mid={<}{0.5}] (p) -- ++(-72:\l) node[above right] {\Large $c(\boldsymbol{e_1})$};
        \draw[mid={<}{0.5}] (p) -- ++(0:\l) node[above left] {\Large $c(\boldsymbol{e_2})$};
        \draw[mid={<}{0.5}] (p) -- ++(72:\l) node[below right] {\Large $c(\boldsymbol{e_3})$};
        \draw[mid={>}{0.5}] (p) -- ++(144:\l) node[above right] {\Large $c(\boldsymbol{e_4})$};
        \draw[mid={>}{0.5}] (p) -- ++(216:\l) node[below right] {\Large $c(\boldsymbol{e_5})$};
        
        \draw[decoration={snake, segment length=5mm, post length=2mm, pre length=2mm}, decorate, -{>[scale=1.5]}] ($(p) + (0,-\l) + (0,-0.5)$) -- ++(0,-1.5);
        
        \node[anchor=north] (b) at ($(p) + (0,-\l) + (0,-2.5)$) {\Large $\displaystyle{p(c(\boldsymbol{e_1})^*) p(c(\boldsymbol{e_2})^*) p(c(\boldsymbol{e_3})^*) p(c(\boldsymbol{e_4})) p(c(\boldsymbol{e_5})) =e}$};
        
    \end{tikzpicture}
    }
    }
    \caption{Cyclicity condition for the color at a vertex.}
\label{cyclic figure}
\end{figure}
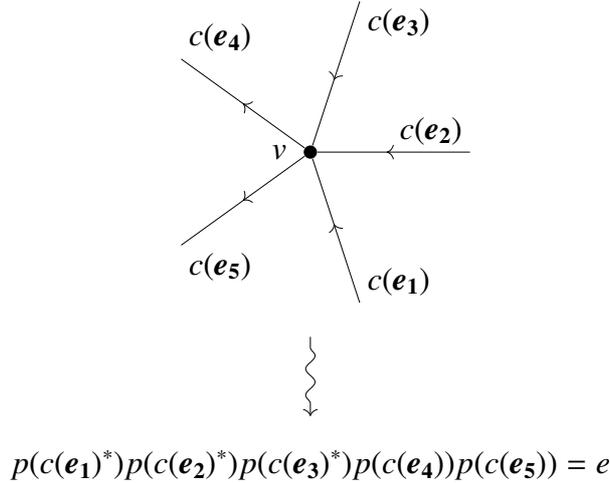

\end{rem}

The space $\SNrm_{\Tsf,\Ccal}^{\prime\prime}(\Sigma,M;\Bbf)$ only partly achieves the goal of globalizing the graphical calculus for $\Ccal$ to $\Sigma$. So far we have imposed local relations inside $2$-faces of $\Tsf$. In order to define local relations everywhere on $\Sigma$ we take a further quotient. This will allow for relations inside disks intersecting edges of the ideal triangulation $\Tsf$. We introduce an equivalence relation implementing isotoping a vertex of $\Gamma$ through an edge of $\Tsf$.
\begin{defn}\label{vertex quotient}
Let $\Gamma$, $\Gamma^\prime\in \SNrm^{\prime\prime}_{\Tsf,\Ccal}(\Sigma,M;\Bbf)$, then $\Gamma\sim_\fbf\Gamma^\prime$ if and only if the two graphs are related by the move shown in figure \ref{vertex move}.

\begin{figure}
    \centering
    \def \ls{6}
    \def \ll{1/3*\ls}
    \def \lc{1/6*\ls}
    \makebox[\textwidth][c]{
    \resizebox{0.7 \width}{!}{
    \begin{tikzpicture}
    
        \node[circle, 
              fill=black, 
              inner sep=0, 
              minimum size=2.5mm] 
              (a) at (0,0) {};
        
        \node[circle, 
              fill=black, 
              inner sep=0, 
              minimum size=2.5mm] 
              (b) at (\ls,0) {};
              
        \draw[mid={>}{0.75}] (a) .. controls ++(\lc,-\lc) and ++(-\lc,0) 
                                .. ++(\ll,-0.5*\ll) 
                                node[anchor=north] {\Large $g_1$};
        \draw[mid={>}{0.75}] (a) .. controls ++(\lc,\lc) and ++(-\lc,0) 
                                .. ++(\ll,0.5*\ll)
                                node[anchor=south] {\Large $g_m$};
        \draw[mid={>}{0.5}] (a) .. controls ++(-\lc,-\lc) and ++(\lc,0) 
                                .. ++(-\ll,-0.5*\ll)
                                node[anchor=north] {\Large $g_n$};
        \draw[mid={>}{0.5}] (a) .. controls ++(-\lc,\lc) and ++(\lc,0) 
                                .. ++(-\ll,0.5*\ll)
                                node[anchor=south] {\Large $g_{m+1}$};
        
        \node[anchor=west] at ($(a) + (0.25*\lc,0)$) {\Large $\cdot$};
        \node[anchor=west] at ($(a) + (0.25*\lc,2mm)$) {\Large $\cdot$};
        \node[anchor=west] at ($(a) + (0.25*\lc,-2mm)$) {\Large $\cdot$};
        \node[anchor=east] at ($(a) + (-0.25*\lc,0)$) {\Large $\cdot$};
        \node[anchor=east] at ($(a) + (-0.25*\lc,2mm)$) {\Large $\cdot$};
        \node[anchor=east] at ($(a) + (-0.25*\lc,-2mm)$) {\Large $\cdot$};
        \draw[mid={>}{0.5}, red] 
                ($(a) + (0.5*\ll, -\ll)$) node[anchor=west, black] {\Large $\fbf$}
                                          .. controls ++(-0.25*\lc, 0.25*\lc) and ++(-0.25*\lc, -0.25*\lc) 
                                          .. ($(a) + (0.5*\ll, \ll)$);
        
        \draw[mid={>}{0.75}] (b) .. controls ++(\lc,-\lc) and ++(-\lc,0) 
                                .. ++(\ll,-0.5*\ll) 
                                node[anchor=north] {\Large $g_1$};
        \draw[mid={>}{0.75}] (b) .. controls ++(\lc,\lc) and ++(-\lc,0) 
                                .. ++(\ll,0.5*\ll)
                                node[anchor=south] {\Large $g_m$};
        \draw[mid={>}{0.5}] (b) .. controls ++(-\lc,-\lc) and ++(\lc,0) 
                                .. ++(-\ll,-0.5*\ll)
                                node[anchor=north] {\Large $g_n$};
        \draw[mid={>}{0.5}] (b) .. controls ++(-\lc,\lc) and ++(\lc,0) 
                                .. ++(-\ll,0.5*\ll)
                                node[anchor=south] {\Large $g_{m+1}$};
        
        \node[anchor=west] at ($(b) + (0.25*\lc,0)$) {\Large $\cdot$};
        \node[anchor=west] at ($(b) + (0.25*\lc,2mm)$) {\Large $\cdot$};
        \node[anchor=west] at ($(b) + (0.25*\lc,-2mm)$) {\Large $\cdot$};
        \node[anchor=east] at ($(b) + (-0.25*\lc,0)$) {\Large $\cdot$};
        \node[anchor=east] at ($(b) + (-0.25*\lc,2mm)$) {\Large $\cdot$};
        \node[anchor=east] at ($(b) + (-0.25*\lc,-2mm)$) {\Large $\cdot$};
        \draw[mid={>}{0.5}, red] 
                ($(b) + (-0.5*\ll, -\ll)$) node[anchor=west, black] {\Large $\fbf$}
                                          .. controls ++(-0.25*\lc, 0.25*\lc) and ++(-0.25*\lc, -0.25*\lc) 
                                          .. ($(b) + (-0.5*\ll, \ll)$);
                                          
        \draw[decoration={snake, segment length=7.5mm, post length=0.5mm, pre length=0.5mm}, decorate, -{>[scale=1.5]}] ($(a) + (1.25*\ll,0)$) -- ($(b) + (-1.25*\ll,0)$);

    \end{tikzpicture}
    }
    }
\caption{The two string-nets $\Gamma$, $\Gamma^\prime$ agree outside of the neighborhood shown here. The move consists of isotoping the string-net vertex trough the edge $\fbf$ of $\Tsf$.}
\label{vertex move}
\end{figure}
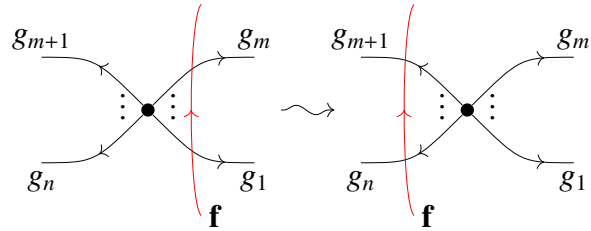

\end{defn}
Note that due to the cyclicity condition \eqref{cyclic}, the relation $\sim_\fbf$ is well defined. We take the quotient of $\SNrm_{\Tsf,\Ccal}^{\prime\prime}(\Sigma,M;\Bbf)$ by $\sim_\fbf$ for all edges $\fbf$ of $\Tsf$ and denote it 
\eq{
\SNrm^{\prime}_{\Tsf,\Ccal}(\Sigma,M;\Bbf)
}

In $\SNrm_{\Tsf,\Ccal}^{\prime}(\Sigma,M;\Bbf)$ local relations hold inside all disks not meeting the images of the marked points. As a consequence boundary points of string-nets cannot be moved.

Finally, the graphs still depend on the chosen $G$-triangulation $\Tsf$. We get rid of this dependence next. Let $\Delta$ be a $2$-face of $\Tsf$ and we take its counterclockwise oriented boundary $\lbr \ebf_1,\ebf_2,\ebf_3\rbr $ with $G$-color $\lbr g_1,g_2,g_3\rbr$. The $2$-face $\Delta$ may contain a boundary edge, which we assume to be $\ebf_1$. We denote $\Bbf_\Delta$ for the boundary value restricted to that boundary component. So for a $2$-face with boundary edge, we consider the vector space
\eq{
\hom_\Delta(\Ccal)\coloneqq  \bigoplus_{i\in I_{g_2},\, j\in I_{g_3}}\hom_\Ccal(\onebb,\Bbf_\Delta\otimes j\otimes k)\, .
}
For all other $2$-faces we set
\eq{
\hom_\Delta(\Ccal)\coloneqq\bigoplus_{i\in I_{g_1},j\in I_{g_2}, k\in I_{g_3}} \hom_\Ccal(\onebb,i\otimes j\otimes k)\, .
}
and define a vector space
\eq{
H_\Tsf^\Ccal(\Sigma;\Bbf)\coloneqq\bigotimes_{\Delta\in \Tsf^{[2]}}\hom_\Delta(\Ccal)\, ,
}
where $\Tsf^{[2]}$ denotes the set of 2-faces of $\Tsf$.
\begin{lem}\label{fat graph basis} There is an isomorphism
\eq{
\Psi:H_\Tsf^\Ccal(\Sigma;\Bbf)\rightarrow \SNrm^\prime_{\Tsf,\Ccal}(\Sigma,M;\Bbf)
}
which maps an element $\bigotimes_{f\in \pi_0(\Delta)}\phi_f\in H_\Tsf^\Ccal(\Sigma;\Bbf)$ to the equivalence class of the dual fat graph $\Gamma$ with boundary value $\Bbf_\Gamma=\Bbf$ and whose vertices are colored by the maps $\phi_f$.
\end{lem}

Among other things Lemma \ref{fat graph basis} tells, that $\SNrm^\prime_{\Tsf,\Ccal}(\Sigma,M;\Bbf)$ has a basis in terms of equivalence classes of fat graphs with edges meeting $\partial\Sigma$ which are dual to $\Tsf$. Internal edges of the basis elements are colored with simple objects and its vertices are colored with basis elements in the corresponding hom-spaces $H_\Tsf^\Ccal(\Sigma;\Bbf)$ via the map $\Psi$. We call this the \textit{fat graph basis for $\Tsf$}. The proof of Lemma \ref{fat graph basis} is very similar to the proof of \cite[Lemma~5.3]{kirillov2011string} and we leave it to the reader to adapt it to the present situation. 

We almost arrived at a sensible definition of a string-net space. A remaining issue is that edges of a string-net cannot be moved through the marked points $M$. This issue can be approached by cloaking marked points, which is defined by projecting to a subspace of $\SNrm^\prime_{\Tsf,\Ccal}(\Sigma;\Bbf)$. Let 
\eq{
\Pi_g:\SNrm^\prime_{\Tsf,\Ccal}(\Sigma,M;\Bbf)\rightarrow \SNrm^\prime_{\Tsf,\Ccal}(\Sigma,M;\Bbf)
}
be the map acting in a neighborhood of $M$ by figure \ref{projector definition}.

\begin{figure}
    \centering
    \def \w{1}
    \def \h{2}
    \makebox[\textwidth][c]{
    \resizebox{0.7 \width}{!}{
    \begin{tikzpicture}
        \node[circle, 
              fill=red, 
              inner sep=0, 
              minimum size=2.5mm,
              label={[red, label distance=0.25cm]90:\Large $s$}]
              (a) at (0,0) {};
        \node (ll) at ($(a)-(1,0)$) {\Large $\Pi_g$:};
        \draw[red] ($(a) + (-\w,-\h)$) 
                            .. controls ++(0,1) and ++(0,-1) 
                            .. ($(a) + (\w,\h)$);
        \draw[red] ($(a) + (\w,-\h)$) 
                            .. controls ++(0,1) and ++(0,-1) 
                            .. ($(a) + (-\w,\h)$);

        \node[circle, 
              fill=red, 
              inner sep=0, 
              minimum size=2.5mm,
              label={[red, label distance=0.25cm]90:\Large $s$}]
              (b) at (4*\w,0) {};
        \draw[red] ($(b) + (-\w,-\h)$) 
                            .. controls ++(0,1) and ++(0,-1) 
                            .. ($(b) + (\w,\h)$);
        \draw[red] ($(b) + (\w,-\h)$) 
                            .. controls ++(0,1) and ++(0,-1) 
                            .. ($(b) + (-\w,\h)$);
        \draw (b) circle (1cm) node[anchor=west, xshift=1cm] {\Large $g$};
        
        \draw[{|[scale=1.5]}-{>[scale=1.5]}] ($(a) + (1.5*\w,0)$) -- ($(b) + (-1.5*\w,0)$);
        
        \node[anchor=west] (eq) at ($(b) + (1.5*\w,0)$) {\Large $\coloneqq \displaystyle{\sum_{i\in I_g} \frac{d_i}{D}}$};
        \node[circle, 
              fill=red, 
              inner sep=0, 
              minimum size=2.5mm,
              label={[red, label distance=0.25cm]90:\Large $s$}]
              (c) at ($(eq.east) + (1.5*\w,0)$) {};
        \draw[red] ($(c) + (-\w,-\h)$) 
                            .. controls ++(0,1) and ++(0,-1) 
                            .. ($(c) + (\w,\h)$);
        \draw[red] ($(c) + (\w,-\h)$) 
                            .. controls ++(0,1) and ++(0,-1) 
                            .. ($(c) + (-\w,\h)$);
        \draw[mid] (c) circle (1cm) node[anchor=west, xshift=1cm] {\Large $i$};
    \end{tikzpicture}
    }
    }
    \caption{Action of the map $\Pi_g$ on a vertex $s$ of an ideal triangulation.}
    \label{projector definition}
\end{figure}
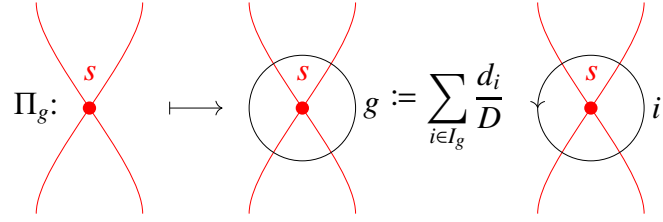
For the neutral element $e\in G$, we simply write $\Pi\coloneqq \Pi_e$ and depict it with a purple circle without group label

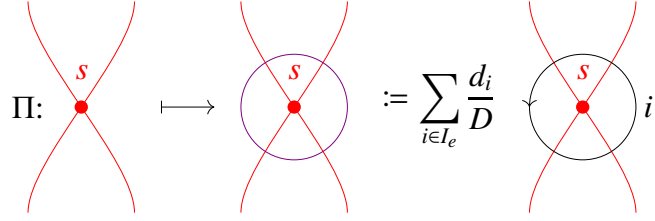
\begin{figure}[H]
    \centering
    \def \w{1}
    \def \h{2}
    \makebox[\textwidth][c]{
    \resizebox{0.7 \width}{!}{
    \begin{tikzpicture}
        \node[circle, 
              fill=red, 
              inner sep=0, 
              minimum size=2.5mm,
              label={[red, label distance=0.25cm]90:\Large $s$}]
              (a) at (0,0) {};
        \node (ll) at ($(a)-(1,0)$) {\Large $\Pi$:};
        \draw[red] ($(a) + (-\w,-\h)$) 
                            .. controls ++(0,1) and ++(0,-1) 
                            .. ($(a) + (\w,\h)$);
        \draw[red] ($(a) + (\w,-\h)$) 
                            .. controls ++(0,1) and ++(0,-1) 
                            .. ($(a) + (-\w,\h)$);

        \node[circle, 
              fill=red, 
              inner sep=0, 
              minimum size=2.5mm,
              label={[red, label distance=0.25cm]90:\Large $s$}]
              (b) at (4*\w,0) {};
        \draw[red] ($(b) + (-\w,-\h)$) 
                            .. controls ++(0,1) and ++(0,-1) 
                            .. ($(b) + (\w,\h)$);
        \draw[red] ($(b) + (\w,-\h)$) 
                            .. controls ++(0,1) and ++(0,-1) 
                            .. ($(b) + (-\w,\h)$);
        \draw[violet] (b) circle (1cm);
        
        \draw[{|[scale=1.5]}-{>[scale=1.5]}] ($(a) + (1.5*\w,0)$) -- ($(b) + (-1.5*\w,0)$);
        
        \node[anchor=west] (eq) at ($(b) + (1.5*\w,0)$) {\Large $\coloneqq \displaystyle{\sum_{i\in I_e} \frac{d_i}{D}}$};
        \node[circle, 
              fill=red, 
              inner sep=0, 
              minimum size=2.5mm,
              label={[red, label distance=0.25cm]90:\Large $s$}]
              (c) at ($(eq.east) + (1.5*\w,0)$) {};
        \draw[red] ($(c) + (-\w,-\h)$) 
                            .. controls ++(0,1) and ++(0,-1) 
                            .. ($(c) + (\w,\h)$);
        \draw[red] ($(c) + (\w,-\h)$) 
                            .. controls ++(0,1) and ++(0,-1) 
                            .. ($(c) + (-\w,\h)$);
        \draw[mid] (c) circle (1cm) node[anchor=west, xshift=1cm] {\Large $i$};
    \end{tikzpicture}
    }
    }
\end{figure}
It is easy to see, that $\Pi^2=\Pi$ and we set 
\eq{
\SNrm_\Tsf^\Ccal(\Sigma,M;\Bbf)\coloneqq \Imrm(\Pi)\, .
}

Using the completeness relation of figure \ref{completeness relation} the cloaking circle allows us to move edges of string-nets through marked points. If the edge has non-trivial $G$-color, though, e.g.\ $c\in \Ccal_g$ in figure \ref{edge through vertex}, the color of the cloaking circle might change 
\begin{figure}[H]
    \centering
    \def\w{1.5}
    \makebox[\textwidth][c]{
    \resizebox{0.7 \width}{!}{
    \begin{tikzpicture}
    \node[circle, 
              fill=red, 
              inner sep=0, 
              minimum size=2.5mm]
              (a) at (0,0) {};
    
    \draw[red] (a) .. controls ++(0.5,0.75) and ++(-0.75,-0.5)
                   .. ($(a) + (30:1.5*\w cm)$);
    \draw[red] (a) .. controls ++(-0.5,0.75) and ++(0.75,-0.5)
                   .. ($(a) + (150:1.5*\w cm)$);          
    \draw[red] (a) .. controls ++(0.5,-0.75) and ++(0.5,0.75)
                   .. ($(a) + (270:1.5*\w cm)$);
    \draw[violet] (a) circle (1cm);
    \draw[rounded corners, mid] ($(a) + (0,-1.25*\w)$) -- ($(a) + (0,-\w)$)
                                                  .. controls ++(1.75,1) and ++(1.75,-1)
                                                  .. ($(a) + (0,\w)$) 
                                                  node[anchor=south west, midway] {\Large $c$}
                                                  -- ($(a) + (0,1.25*\w)$);
    
    \node[anchor=west] (eq1) at ($(a) + (1.5*\w,0)$){\Large $\displaystyle{=\sum_{\substack{i \in I_e \\ j \in I_g}} \frac{d_i d_j}{D^2} \sum_{\alpha}}$};

    \node[circle, 
              fill=red, 
              inner sep=0, 
              minimum size=2.5mm]
              (b) at ($(eq1.east) + (1.5*\w,0)$) {};
    
    \draw[red] (b) .. controls ++(0.5,0.75) and ++(-0.75,-0.5)
                   .. ($(b) + (30:1.5*\w cm)$);
    \draw[red] (b) .. controls ++(-0.5,0.75) and ++(0.75,-0.5)
                   .. ($(b) + (150:1.5*\w cm)$);          
    \draw[red] (b) .. controls ++(0.5,-0.75) and ++(0.5,0.75)
                   .. ($(b) + (270:1.5*\w cm)$);
    
    \centerarc[mid](b)(-45:15:1cm);
    \centerarc[mid](b)(15:315:1cm);
    \node[anchor=west, xshift=1mm] at ($(b) + (-15:1)$) {\Large $j$};
    \node[anchor=east] at ($(b) + (180:1)$) {\Large $i$};
    \node[circle, draw,
              fill=white,
              inner sep=0.25mm, 
              minimum size=2.5mm] 
              (alph1) at ($(b) + (15:1)$) {\Large $\alpha$};
              
    \node[circle, draw,
              fill=white,
              inner sep=0.25mm, 
              minimum size=2.5mm] 
              (alph2) at ($(b) + (-45:1)$) {\Large $\alpha$};
    
    \draw
    [rounded corners, mid]
    ($(b) + (0,-1.25*\w)$) node[anchor=south east] {\Large $c$} 
                           -- ($(b) + (0,-\w)$) 
                           .. controls ++(0.1,0.1) and ++(0,-0.35)
                           .. (alph2);
    \draw
    [rounded corners, mid]
    (alph1) .. controls ++(0,0.75) and ++(0.1,-0.1)
            .. ($(b) + (0,\w)$)
            -- ($(b) + (0,1.25*\w)$)
            node[anchor=north east] {\Large $c$} ;
    
    \node[anchor=west] (eq2) at ($(eq1.west) + (0,-3*\w)$){\Large $\displaystyle{=\sum_{j \in I_g} \frac{d_j}{D^2}}$};

    \node[circle, 
              fill=red, 
              inner sep=0, 
              minimum size=2.5mm]
              (c) at ($(eq2.east) + (1.5*\w,0)$) {};
    
    \draw[red] (c) .. controls ++(0.5,0.75) and ++(-0.75,-0.5)
                   .. ($(c) + (30:1.5*\w cm)$);
    \draw[red] (c) .. controls ++(-0.5,0.75) and ++(0.75,-0.5)
                   .. ($(c) + (150:1.5*\w cm)$);          
    \draw[red] (c) .. controls ++(0.5,-0.75) and ++(0.5,0.75)
                   .. ($(c) + (270:1.5*\w cm)$);
    \draw[mid={>}{0}] (c) circle (1cm);
    \node[anchor=north west] at ($(c) + (1,0)$) {\Large $j$};
    
    \draw
    [rounded corners, mid] 
    ($(c) + (0,-1.25*\w)$) -- ($(c) + (0,-\w)$)
                           .. controls ++(-1.75,1) and ++(-1.75,-1)
                           .. ($(c) + (0,\w)$) 
                           node[anchor=south east, midway] {\Large $c$}
                           -- ($(c) + (0,1.25*\w)$);
                           
   \node[anchor=west] (eq3) at ($(eq2.west) + (0,-3*\w)$){\Large $=$};

    \node[circle, 
              fill=red, 
              inner sep=0, 
              minimum size=2.5mm]
              (d) at ($(eq3.east) + (1.5*\w,0)$) {};
    
    \draw[red] (d) .. controls ++(0.5,0.75) and ++(-0.75,-0.5)
                   .. ($(d) + (30:1.5*\w cm)$);
    \draw[red] (d) .. controls ++(-0.5,0.75) and ++(0.75,-0.5)
                   .. ($(d) + (150:1.5*\w cm)$);          
    \draw[red] (d) .. controls ++(0.5,-0.75) and ++(0.5,0.75)
                   .. ($(d) + (270:1.5*\w cm)$);
    \draw (d) circle (1cm);
    \node[anchor=north west] at ($(d) + (1,0)$) {\Large $g$};
    
    \draw
    [rounded corners, mid] 
    ($(d) + (0,-1.25*\w)$) -- ($(d) + (0,-\w)$)
                           .. controls ++(-1.75,1) and ++(-1.75,-1)
                           .. ($(d) + (0,\w)$) 
                           node[anchor=south east, midway] {\Large $c$}
                           -- ($(d) + (0,1.25*\w)$);    
    
    \end{tikzpicture}
    }
    }
    \caption{We can use completeness in $\Ccal$ to pull a string-net edge through a vertex of $\Tsf$. However, the color of the cloaking circle might change.}
    \label{edge through vertex}
\end{figure}
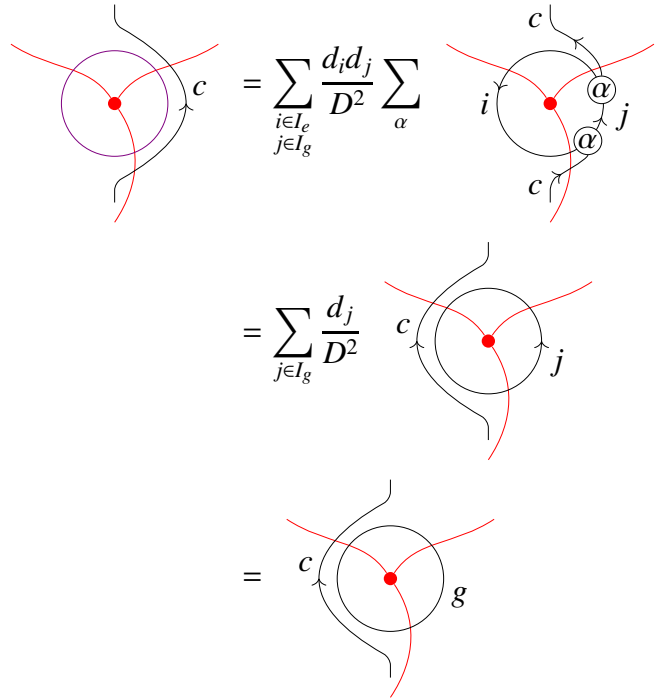

The resulting string-net is still in the image of $\Pi$, as can be easily checked using the completeness relation again. 

The vector space $\SNrm_T^\Ccal(\Sigma,M;\Bbf)$ clearly depends on the choice of a $G$-triangulation. Let $(\Tsf,g)\xrightarrow{F_f}(\Tsf^\prime,g^\prime) $ be a $G$-flip, where none of the edges of the two triangles is a boundary edge. We define the associated linear map 
\eq{
F_{\Tsf,\Tsf^\prime}:\SNrm^\Ccal_\Tsf(\Sigma,M;\Bbf) \rightarrow \SNrm^\Ccal_{\Tsf^\prime}(\Sigma,M;\Bbf) 
}
as the linear extension of the map defined on elements of the fat graph basis by
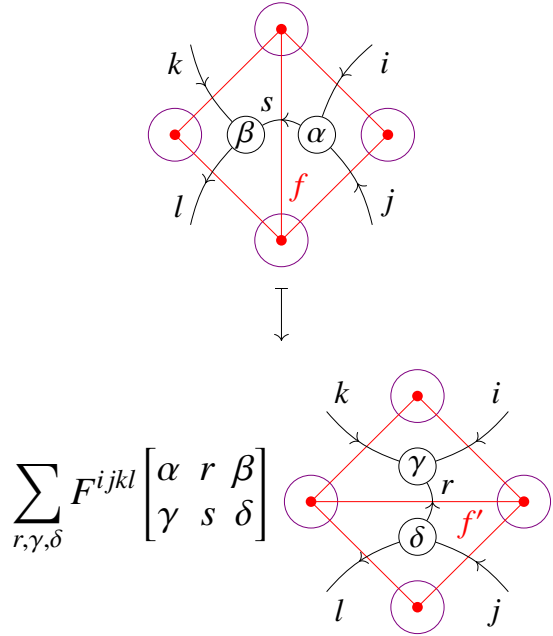
\begin{figure}[H]
    \centering
    \def\l{2}
    \makebox[\textwidth][c]{
    \resizebox{0.7 \width}{!}{
    \begin{tikzpicture}
        \node (LHS) at (0,0) {
            \begin{tikzpicture}
                
                \node[anchor=center, circle, fill=red, inner sep=0, minimum size=2mm] (A1) at (0,0) {};
                \node[anchor=center, circle, fill=red, inner sep=0, minimum size=2mm] (B1) at (-\l,\l) {};
                \node[anchor=center, circle, fill=red, inner sep=0, minimum size=2mm] (C1) at (\l,\l) {};
                \node[anchor=center, circle, fill=red, inner sep=0, minimum size=2mm] (D1) at (0,2*\l) {};
                
                \foreach \vertex in {A1, B1, C1, D1}{
                    \node[anchor=center,
                          circle,
                          violet,
                          draw,
                          inner sep=0,
                          minimum size=1cm]
                          at (\vertex) {};
                };
                
                \node[anchor=center,
                      fill=white,
                      circle, 
                      anchor=center, 
                      inner sep=0, 
                      minimum size=7mm, 
                      draw] 
                      (beta) at (barycentric cs:A1=0.5,B1=0.5,D1=0.5) {\Large $\beta$};
                
                \node[anchor=center,
                      fill=white,
                      circle, 
                      anchor=center, 
                      inner sep=0, 
                      minimum size=7mm, 
                      draw] 
                      (alpha) at (barycentric cs:A1=0.5,C1=0.5,D1=0.5) {\Large $\alpha$};
                
                \draw[red] (A1.center) 
                        -- (B1.center) node[anchor=center, midway] (AB1) {}
                        -- (D1.center) node[anchor=center, midway] (BD1) {}
                        -- cycle node[red, pos=0.75, anchor=west] {\Large $f$};
                \draw[red] (D1.center) 
                        -- (C1.center) node[anchor=center, midway] (CD1) {}
                        -- (A1.center) node[anchor=center, midway] (AC1) {};
                
                \node[anchor=center] at ($(AB1) + (-135:0.5*\l)$) (BL) {};
                \node[anchor=center] at ($(BD1) + (135:0.5*\l)$) (TL) {};
                \node[anchor=center] at ($(CD1) + (45:0.5*\l)$) (TR) {};
                \node[anchor=center] at ($(AC1) + (-45:0.5*\l)$) (BR) {};
                
                \draw[mid] (alpha) to[bend right] node[pos=0.5, anchor=south east] {\Large $s$} (beta);
                \draw[mid] (TL.center) node[anchor=north east] {\Large $k$} to[bend right=15] (beta);
                \draw[mid] (beta) to[bend right=15] (BL.center) node[anchor=south east] {\Large $l$};
                \draw[mid] (TR.center) node[anchor=north west] {\Large $i$} to[bend right=15] (alpha);
                \draw[mid] (BR.center) node[anchor=south west] {\Large $j$} to[bend right=15] (alpha);
            \end{tikzpicture}
        };
        \node[anchor=north] (res) at ($(LHS.south) + (0,-0.75*\l)$) {
            \begin{tikzpicture}
                \node[anchor=west] (sum) at ($(LHS.east) + (0.5*\l,0)$) {\LARGE $\displaystyle{\sum_{r,\gamma, \delta} F^{ijkl} \begin{bmatrix} \alpha & r & \beta \\ \gamma & s &\delta \end{bmatrix}}$};
                \node[anchor=west] (RHS) at (sum.east){
                    \begin{tikzpicture}
                        
                        \node[anchor=center, circle, fill=red, inner sep=0, minimum size=2mm] (A2) at (0,0) {};
                        \node[anchor=center, circle, fill=red, inner sep=0, minimum size=2mm] (B2) at (-\l,\l) {};
                        \node[anchor=center, circle, fill=red, inner sep=0, minimum size=2mm] (C2) at (\l,\l) {};
                        \node[anchor=center, circle, fill=red, inner sep=0, minimum size=2mm] (D2) at (0,2*\l) {};
                        
                        \foreach \vertex in {A2, B2, C2, D2}{
                            \node[anchor=center,
                                  circle,
                                  violet,
                                  draw,
                                  inner sep=0,
                                  minimum size=1cm]
                                  at (\vertex) {};
                        };
                        
                        \node[anchor=center,
                              fill=white,
                              circle, 
                              anchor=center, 
                              inner sep=0, 
                              minimum size=7mm, 
                              draw] 
                              (gamma) at (barycentric cs:C2=0.5,B2=0.5,D2=0.5) {\Large $\gamma$};
                        
                        \node[anchor=center,
                              fill=white,
                              circle, 
                              anchor=center, 
                              inner sep=0, 
                              minimum size=7mm, 
                              draw] 
                              (delta) at (barycentric cs:A2=0.5,C2=0.5,B2=0.5) {\Large $\delta$};
                        
                        \draw[red] (A2.center)
                                -- (B2.center) node[anchor=center, midway] (AB2) {}
                                -- (C2.center) node[red, pos=0.75, anchor=north] {\Large $f'$}
                                               node[anchor=center, midway] (BC2) {}
                                -- cycle node[anchor=center, midway] (AC2) {};
                        \draw[red] (B2.center) 
                                -- (D2.center) node[anchor=center, midway] (BD2) {}
                                -- (C2.center) node[anchor=center, midway] (CD2) {};
                        
                        \node[anchor=center] at ($(AB2) + (-135:0.5*\l)$) (DL) {};
                        \node[anchor=center] at ($(BD2) + (135:0.5*\l)$) (UL) {};
                        \node[anchor=center] at ($(CD2) + (45:0.5*\l)$) (UR) {};
                        \node[anchor=center] at ($(AC2) + (-45:0.5*\l)$) (DR) {};
                        
                        \draw[mid] (delta) to[bend right] node[pos=0.5, anchor=south west] {\Large $r$} (gamma);
                        \draw[mid] (UL.center) node[anchor=south west] {\Large $k$} to[bend right=15] (gamma);
                        \draw[mid] (delta) to[bend right=15] (DL.center) node[anchor=north west] {\Large $l$};
                        \draw[mid] (UR.center) node[anchor=south east] {\Large $i$} to[bend left=15] (gamma);
                        \draw[mid] (DR.center) node[anchor=north east] {\Large $j$} to[bend right=15] (delta);
                    \end{tikzpicture}
                };
            \end{tikzpicture}
        };
        \draw[{|[scale=1.5]}-{>[scale=1.5]}] ($(LHS.south) + (0,-0.25)$) -- ($(res.north) + (0,0.25)$);
    \end{tikzpicture}
    }
    }
\end{figure}

\noindent
where the map is the identity outside the open neighborhood of the edge $\fbf$ shown
in red in the picture. In case there is a boundary edge, we use first use the completeness relation to decompose the $\Ccal$-colored edges ending on
$\partial\Sigma$ into a sum of simple colored edges and then perform the same move as in the previous case, followed by the inverse of the completeness relation.

To a gauge transformation $\Tsf\xrightarrow{\lambda_g(v)}\Tsf^\prime$, we associate the map $G_g(v):\SNrm^\Ccal_\Tsf(\Sigma,M;\Bbf)\rightarrow \SNrm^\Ccal_{\Tsf^\prime}(\Sigma,M;\Bbf)$ adding an $I_g$-colored cloaking circle around the internal vertex $v$ of the $G$-triangulation $T$.

\begin{theo} There is a functor
\eq{
\SNrm^\Ccal:\Pi_1\lb \Pcal^G(\Sigma,\zeta)\rb &\rightarrow \Vectpzc_\Kbb\\
(\Tsf,g)&\mapsto \SNrm^\Ccal_\Tsf(\Sigma,M;\Bbf)\\
\left[\Tsf\xrightarrow{F_e}\Tsf^\prime\right] &\mapsto F_{\Tsf,\Tsf^\prime}\\
\left[\Tsf\xrightarrow{\lambda_g(v)}\Tsf^\prime\right]&\mapsto G_g(v)
}
\end{theo}
\begin{proof}
The linear maps $F_{\Tsf,\Tsf^\prime}$, $G_g(v)$, are indeed isomorphisms. To see the first statement, observe that $F_{\Tsf,\Tsf^\prime}$ is its own inverse. In case of the gauge map, this is a consequence of the completeness relation in $\Ccal$. Thus, we only have to check that relations \textbf{GP1}-\textbf{GP6} are mapped to an equivalence of linear isomorphisms. Relation \textbf{GP1} is exactly the statement that $F_{\Tsf,\Tsf^\prime}$ is its own inverse. The \textbf{GP2} equivalence follows, since we are manipulating disjoint parts of the string-net, which obviously commutes. Next, \textbf{GP3} gets mapped to the pentagon relation for the $6j$-symbols. Relation \textbf{GP4} is an easy computation using the graphical calculus in $\Ccal$. Since adding cloaking circles at different internal vertices commutes, relation \textbf{GP5} holds. The group law \textbf{GP6} follows by applying the completeness relation twice.

\end{proof} 

\begin{defn} Given a $G$-surface $(\Sigma,\zeta)$, a set of marked points $M$ and a $G$-graded spherical fusion category $\Ccal$, the \textit{based bare string-net space} is defined by
\eq{
\SNrm^\Ccal(\Sigma,M;\Bbf)\coloneqq \varprojlim_{\Pi_1\lb \Pcal^G(\Sigma,\zeta)\rb} \SNrm^\Ccal_\Tsf(\Sigma,M;\Bbf)
}
\end{defn}

\begin{rem} Instead of taking the limit over the diagram, we could have instead taken its colimit. The resulting vector spaces are isomorphic, since the diagram involves only linear isomorphisms. For the same reason with have a canonical isomorphism
\eq{
\SNrm^\Ccal(\Sigma,M;\Bbf)\cong \SNrm^\Ccal_\Tsf(\Sigma,M;\Bbf)
}
for any $G$-triangulation $\Tsf$.
\end{rem}

The bare string-net space based at $M$ only depends on the $G$-surface $(\Sigma,\zeta)$, the set of marked points $M$ and $\Ccal$, but not on the choice of a $G$-triangulation based at $\delta\cup M$. The set $\delta$ will be uniquely determined by the input datum of a $1$-morphism. Thus, the only arbitrary choice left is the set $M$ of marked points. However, there is a distinguished isomorphism between based bare string-net spaces for different choices of marked points. We show this by showing that for all based string-net spaces, there is a specific isomorphism to the based bare string-net space with $M=\emptyset$, or $M$ a single point in case of a disk. For a given $M$, we pick a $G$-triangulation and from Lemma \ref{fat graph basis} we get that $\SNrm^\Ccal(\Sigma,M;\Bbf)$ has a fat graph basis with additional cloaking circles around elements in $M$. To make arguments more palpable, we discuss the specific case of a genus $2$ surface $\Sigma$ with $4$ boundary components. We choose an arbitrary point $\star\in \Sigma\backslash M$ and a set of generators in $\pi_1(\Sigma\backslash M, \star)$

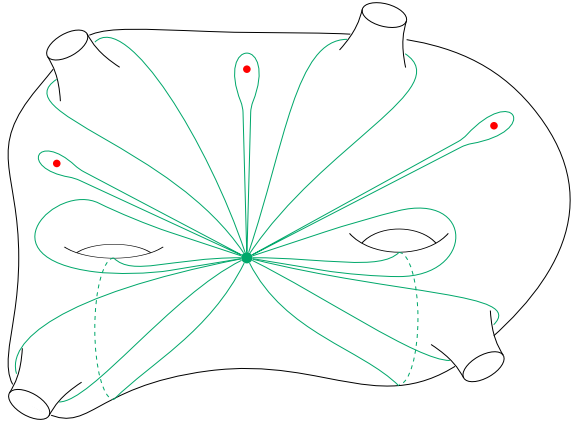
\begin{figure}[H]
    \centering
    \makebox[\textwidth][c]{
    \resizebox{0.5\width}{!}{
    \begin{tikzpicture}
        
        \begin{scope}[on background layer]
            \draw  plot[smooth cycle, thick, tension=1] coordinates {(-3.5,2.5) (1,4) (9.5,2) (8,-4.5) (1,-5) (-3.5,-6) (-4,-2)};
        \end{scope}
        
        \node[inner sep=-1mm] (LH) at (-1.5,-1.5) {
        \resizebox{0.75\width}{0.3\height}{
        \begin{tikzpicture}
        
            \FPset\rbig{2}
            \FPset\rsmol{1.5}
            
            \FPset\initang{-30}
            \FPset\finang{-150}
            \FPeval\rang{round(abs((\initang - \finang) / 2), 0)}
            
            \FPeval\dsmol{\rsmol * (1-cos(\rang * \FPpi / 180))}
            \FPeval\dbig{\rbig - root(2, (\rbig^2 - \rsmol^2 * sin(\rang * \FPpi / 180 )^2))}
            \FPeval\hsep{(\rsmol - \dsmol) + (\rbig - \dbig)}
            
            \node (c) at (0,0) {};
            \centerarc[thick](c)(\initang:\finang:\rbig cm) {big};
            \centerarc[thick, transform canvas={yshift=-\hsep cm}](c)(-\initang:-\finang:\rsmol cm);
        
        \end{tikzpicture}
        }
        };
    
    \node[inner sep=-1.5mm] (RH) at (6,-1) {
    \resizebox{0.75\width}{0.5\height}{
    \begin{tikzpicture}
        
            \FPset\rbig{2}
            \FPset\rsmol{1.5}
            
            \FPset\initang{-30}
            \FPset\finang{-150}
            \FPeval\rang{round(abs((\initang - \finang) / 2), 0)}
            
            \FPeval\dsmol{\rsmol * (1-cos(\rang * \FPpi / 180))}
            \FPeval\dbig{\rbig - root(2, (\rbig^2 - \rsmol^2 * sin(\rang * \FPpi / 180 )^2))}
            \FPeval\hsep{(\rsmol - \dsmol) + (\rbig - \dbig)}
            
            \node (c) at (0,0) {};
            \centerarc[thick](c)(\initang:\finang:\rbig cm) {big};
            \centerarc[thick, transform canvas={yshift=-\hsep cm}](c)(-\initang:-\finang:\rsmol cm);
        
        \end{tikzpicture}
        }
        };

    \node[fill=white, inner sep=-3mm, rotate=30] (TL) at (-2.5,3.25) {
    \resizebox{0.3\width}{0.3\height}{
    \begin{tikzpicture}[scale=1]
    
        \draw (0,0) arc (0:180:2 and 1);
        \draw (0,0) arc (0:-180:2 and 1);
        \draw (0,0) .. controls ++(0,-1) and ++(-1,2) .. (1,-4);
        \draw (-4,0) .. controls ++(0,-1) and ++(1,2) .. (-5,-4);
    
    \end{tikzpicture}
    }
    };
    
    \node[fill=white, inner sep=-2mm, rotate=-15] (TR) at (5.5,4) {
    \resizebox{0.3\width}{0.3\height}{
    \begin{tikzpicture}[scale=1]
    
        \draw (0,0) arc (0:180:2 and 1);
        \draw (0,0) arc (0:-180:2 and 1);
        \draw (0,0) .. controls ++(0,-1) and ++(-1,2) .. (1,-4);
        \draw (-4,0) .. controls ++(0,-1) and ++(1,2) .. (-5,-4);
    
    \end{tikzpicture}
    }
    };
    
    \node[fill=white, inner sep=-2.9mm, rotate=-150] (BR) at (8,-4.5) {
    \resizebox{0.3\width}{0.3\height}{
    \begin{tikzpicture}[scale=1]
    
        \draw (0,0) arc (0:180:2 and 1);
        \draw (0,0) arc (0:-180:2 and 1);
        \draw (0,0) .. controls ++(0,-1) and ++(-1,2) .. (1,-4);
        \draw (-4,0) .. controls ++(0,-1) and ++(1,2) .. (-5,-4);

    \end{tikzpicture}
    }
    };
    
    \node[fill=white, inner sep=-2.9mm, rotate=150] (BL) at (-3.5,-5.5) {
    \resizebox{0.3\width}{0.3\height}{
    \begin{tikzpicture}[scale=1]
    
        \draw (0,0) arc (0:180:2 and 1);
        \draw (0,0) arc (0:-180:2 and 1);
        \draw (0,0) .. controls ++(0,-1) and ++(-1,2) .. (1,-4);
        \draw (-4,0) .. controls ++(0,-1) and ++(1,2) .. (-5,-4);

    \end{tikzpicture}
    }
    };
    
    \node[circle, 
          fill=jade,
          inner sep=0, 
          minimum size=3mm]
          (*) at (2,-2) {};
    
    \draw[jade, rounded corners=10mm]
        (*) -- ++(-3,-0.5)
        .. controls ++(-4,0.5) and ++(-2,1) .. ($(*) + (-3,1)$)
        -- (*);
        
    \node (LHT) at (LH.south) {};
    \node (LHB) at ($(LHT) + (0,-3.75)$) {};
    \draw[jade]
        (*) .. controls ++(-2,0) and ++(0.5,-0.5) .. ($(LHT) + (0,0)$);
    \draw[jade, dashed] (LHT) arc (90:-90:-0.5 and 1.875);
    \draw[jade]
        (*).. controls ++(-1,-2) and ++(1,1) .. ($(LHB) + (0,0)$);
    
    \draw[jade, rounded corners=10mm]
        (*) -- ++(3,-0.5)
        .. controls ++(3,0) and ++(3,0.75)
        .. ($(*) + (3,1)$)
        -- (*);
        
    \node (RHT) at (RH.south) {};
    \node (RHB) at (6,-5.4) {};
    \draw[jade]
        (*) .. controls ++(2,0) and ++(-0.5,-0.5) .. ($(RHT) + (0,0)$);
    \draw[jade, dashed] (RHT) arc (90:-90:0.5 and 1.75);
    \draw[jade]
        (*).. controls ++(1,-2) and ++(-1,1) .. ($(RHB) + (0,0)$);
        
    \node[circle, fill=red, inner sep=0, minimum size=2mm] (M1) at (2,3) {};
    \draw[jade, rounded corners]
        (*) -- ++(88.5:4)
        .. controls ++(68.5:2) and ++(111.5:2) .. ($(*) + (91.5:4)$)
        -- (*);

    \node[circle, fill=red, inner sep=0, minimum size=2mm] (M0) at (-3,0.5) {};
    \draw[jade, rounded corners] 
        (*) -- ++(152:5)
        .. controls ++(137:1.5) and ++(170:1.5) .. ($(*) + (155:5)$)
        -- (*);
        
    \node[circle, fill=red, inner sep=0, minimum size=2mm] (M2) at (8.5,1.5) {};
    \draw[jade, rounded corners] 
        (*) -- ++(27.5:6.5)
        .. controls ++(10:2) and ++(47.5:2) .. ($(*) + (29:6.5)$)
        -- (*);

    \node (BRR) at ($(BR.west) + (-0.22,0.25)$) {};
    \node (BRL) at ($(BR.east) + (0.05,0.1)$) {};
    \draw[jade] (*) .. controls ++(4,-1) and ++ (1,1) .. (BRR);
    \draw[jade] (*) .. controls ++(2,-1) and ++ (-1,0) .. (BRL);
    
    \node (BLR) at ($(BL.west) + (-0.15,0)$) {};
    \node (BLL) at ($(BL.east) + (0.05,-0.05)$) {};
    \draw[jade] (*) .. controls ++(-2,-1) and ++(1,0) .. (BLR);
    \draw[jade] (*) .. controls ++(-5,-1) and ++(-0.3,1) .. (BLL);
    
    \node (TRR) at ($(TR.east) + (-0.2,-0.3)$) {};
    \node (TRL) at ($(TR.west) + (0,-0.4)$) {};
    \draw[jade] (*) .. controls ++(2,3) and ++(1.5,-1) .. (TRR);
    \draw[jade] (*) .. controls ++ (1,3) and ++ (-1.5,0) .. (TRL);
    
    \node (TLR) at ($(TL.east) + (-0.2,0)$) {};
    \node (TLL) at ($(TL.west) + (0.2,-0.1)$) {};
    \draw[jade] (*) .. controls ++(-1,3) and ++(1,1) .. (TLR);
    \draw[jade] (*) .. controls ++(-2,3) and ++(-1.5,-1) .. (TLL);
    
    \end{tikzpicture}
    }}
\end{figure}

For the dual fat graph $\Gamma$ of the ideal triangulation $\Tsf$, we pick a maximal tree $t$ and denote $\ell\coloneqq E(\Gamma)-E(t)$. Since the fat graph $\Gamma$ generates $\pi_1(\Sigma\backslash M,\star)$, we can in fact choose $t$ in such a way, that upon collapsing $t$ to $\star$, the homotopy classes of the remaining edges $\ell$ agree with the chosen generators of $\pi_1(\Sigma\backslash M,\star)$. This means, that for any of the generators $\gamma$, there exists an element $l_\gamma\in \ell$, such that. after collapsing $t$ to $\star$ it holds that $\left[l_\gamma\right]=\gamma$. The tree $t$ has a disk-shaped neighborhood. Using local relations on this disk, we can replace the elements of the fat graph basis with string-nets supported on the contracted graphs. These string-nets still constitute a basis for the based bare string-net space. We use the completeness relation, possibly followed by a gauge transformation isomorphism, to map these new basis elements to string-nets where the points of $M$ are solely surrounded by cloaking circles

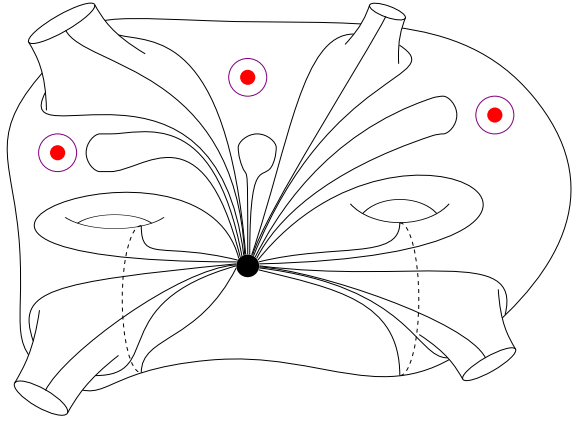
\begin{figure}[H]
    \centering
    \makebox[\textwidth][c]{
    \resizebox{0.5\width}{!}{
    \begin{tikzpicture}
        
        \begin{scope}[on background layer]
            \draw  plot[smooth cycle, thick, tension=1] coordinates {(-3.5,2.5) (1,4) (9.5,2) (8,-4.5) (1,-5) (-3,-5.5) (-4,-2)};
        \end{scope}
        
        \node[inner sep=-1mm] (LH) at (-1.5,-1) {
        \resizebox{0.75\width}{0.3\height}{
        \begin{tikzpicture}
        
            \FPset\rbig{2}
            \FPset\rsmol{1.5}
            
            \FPset\initang{-30}
            \FPset\finang{-150}
            \FPeval\rang{round(abs((\initang - \finang) / 2), 0)}
            
            \FPeval\dsmol{\rsmol * (1-cos(\rang * \FPpi / 180))}
            \FPeval\dbig{\rbig - root(2, (\rbig^2 - \rsmol^2 * sin(\rang * \FPpi / 180 )^2))}
            \FPeval\hsep{(\rsmol - \dsmol) + (\rbig - \dbig)}
            
            \node (c) at (0,0) {};
            \centerarc[thick](c)(\initang:\finang:\rbig cm) {big};
            \centerarc[thick, transform canvas={yshift=-\hsep cm}](c)(-\initang:-\finang:\rsmol cm);
        
        \end{tikzpicture}
        }
        };
    
        \node[inner sep=-1.5mm] (RH) at (6,-0.5) {
        \resizebox{0.75\width}{0.5\height}{
        \begin{tikzpicture}
        
            \FPset\rbig{2}
            \FPset\rsmol{1.5}
            
            \FPset\initang{-30}
            \FPset\finang{-150}
            \FPeval\rang{round(abs((\initang - \finang) / 2), 0)}
            
            \FPeval\dsmol{\rsmol * (1-cos(\rang * \FPpi / 180))}
            \FPeval\dbig{\rbig - root(2, (\rbig^2 - \rsmol^2 * sin(\rang * \FPpi / 180 )^2))}
            \FPeval\hsep{(\rsmol - \dsmol) + (\rbig - \dbig)}
            
            \node (c) at (0,0) {};
            \centerarc[thick](c)(\initang:\finang:\rbig cm) {big};
            \centerarc[thick, transform canvas={yshift=-\hsep cm}](c)(-\initang:-\finang:\rsmol cm);
        
        \end{tikzpicture}
        }
        };

    \node[inner sep=0mm, rotate=30] (TL) at (-2.5,3.25) {
    \resizebox{1\width}{1\height}{
    \begin{tikzpicture}[xscale=0.5,yscale=0.45]
    
        \fill[white] (0,0) arc (0:180:2 and 0.5)
        .. controls ++(0,-1) and ++(1,2) .. (-5,-4)
        -- (1,-4)
        .. controls ++(-1,2) and ++(0,-1) .. (0,0);
        
        \draw (0,0) arc (0:180:2 and 0.5)
        .. controls ++(0,-1) and ++(1,2) .. (-5,-4);
        \draw (1,-4) .. controls ++(-1,2) and ++(0,-1) .. (0,0);
        \draw (0,0) arc (0:-180:2 and 0.5);
    \end{tikzpicture}
    }
    };
    \node[anchor=center] (TL_out) at ($(TL.north) + (30-90:0.45)$) {};
    
    \node[inner sep=0mm, rotate=-20] (TR) at (5.5,3.75) {
    \resizebox{1\width}{1\height}{
    \begin{tikzpicture}[xscale=0.25,yscale=0.3]
    
        \fill[white] (0,0) arc (0:180:2 and 0.5)
        .. controls ++(0,-1) and ++(1,2) .. (-5,-4)
        -- (1,-4)
        .. controls ++(-1,2) and ++(0,-1) .. (0,0);
        
        \draw (0,0) arc (0:180:2 and 0.5)
        .. controls ++(0,-1) and ++(1,2) .. (-5,-4);
        \draw (1,-4) .. controls ++(-1,2) and ++(0,-1) .. (0,0);
        \draw (0,0) arc (0:-180:2 and 0.5);
    \end{tikzpicture}
    }
    };
    \node[anchor=center] (TR_out) at ($(TR.north) + (-20-90:0.3)$) {};
    
    \node[inner sep=0mm, rotate=-150] (BR) at (8,-4.5) {
    \resizebox{1\width}{1\height}{
    \begin{tikzpicture}[xscale=0.4,yscale=0.4]
    
        \fill[white] (0,0) arc (0:180:2 and 0.5)
        .. controls ++(0,-1) and ++(1,2) .. (-5,-4)
        -- (1,-4)
        .. controls ++(-1,2) and ++(0,-1) .. (0,0);
        
        \draw (0,0) arc (0:180:2 and 0.5)
        .. controls ++(0,-1) and ++(1,2) .. (-5,-4);
        \draw (1,-4) .. controls ++(-1,2) and ++(0,-1) .. (0,0);
        \draw (0,0) arc (0:-180:2 and 0.5);
    \end{tikzpicture}
    }
    };
    \node[anchor=center] (BR_out) at ($(BR.north) + (-150-90:0.4)$) {};
    
    \node[inner sep=0mm, rotate=150] (BL) at (-3,-5.25) {
    \resizebox{1\width}{1\height}{
    \begin{tikzpicture}[yscale=0.5,xscale=0.4]
        
        \fill[white] (0,0) arc (0:180:2 and 0.5)
        .. controls ++(0,-1) and ++(1,2) .. (-5,-4)
        -- (1,-4)
        .. controls ++(-1,2) and ++(0,-1) .. (0,0);
        
        \draw (0,0) arc (0:180:2 and 0.5)
        .. controls ++(0,-1) and ++(1,2) .. (-5,-4);
        \draw (1,-4) .. controls ++(-1,2) and ++(0,-1) .. (0,0);
        \draw (0,0) arc (0:-180:2 and 0.5);
        
    \end{tikzpicture}
    }
    };
    \node[anchor=center] (BL_out) at ($(BL.north) + (150-90:0.5)$) {};
    
    \node[circle, 
          fill,
          inner sep=0, 
          minimum size=6mm]
          (*) at (2,-2.5) {};

    \draw (*.140) .. controls ++(-1,4) and ++(-11.5,0) .. (*.160);
    
    \node (LHT) at ($(LH.south) + (0,0.075)$) {};
    \node[anchor=center] (LHB) at ($(LHT) + (0.7,-3.9)$) {};
    \draw
        (*.150) .. controls ++(-2,0.5) and ++(0,-1) .. ($(LHT) + (0.7,0)$);
    \draw[dashed] ($(LHT) + (0.7,0)$) arc (90:-90:-0.5 and 1.95);
    \draw
        (*.190).. controls ++(-1,-2) and ++(0,1) .. ($(LHB) + (0,0)$);
    
    \draw (*.10) .. controls ++(11,0.25) and ++(4,5) .. (*.30);
    
    \node (RHT) at (RH.south) {};
    \node[anchor=center] (RHB) at (6,-5.405) {};
    \draw
        (*.20) .. controls ++(2,0.5) and ++(0,-1) .. ($(RHT) + (0,0)$);
    \draw[dashed] (RHT) arc (90:-90:0.5 and 2.025);
    \draw
        (*.340).. controls ++(2,-0.5) and ++(0,2) .. (RHB.center);
        
    \node[circle, fill=red, inner sep=0, minimum size=4mm] (M1) at (2,2.5) {};
    \node[violet, circle, draw, anchor=center, inner sep=0, minimum size=1cm] at (M1.center) {};
    \draw[rounded corners] (*.80) .. controls ++(0,3) and ++(-0.25,-1) .. ($(M1.center) + (0.75,-2)$) arc (0:180:0.5 and 0.5) .. controls ++(0.25,-1) and ++(0,3) .. (*.90);

    \node[circle, fill=red, inner sep=0, minimum size=4mm] (M0) at (-3,0.5) {};
    \node[violet, circle, draw, anchor=center, inner sep=0, minimum size=1cm] at (M0.center) {};
    \draw[rounded corners] 
        (*.120) .. controls ++(-0.75,4) and ++(1,0) .. ($(M0.center) + (1.25,0.5)$) arc (90:270:0.5) .. controls ++(1,0) and ++(-1,4) .. (*.130);
   
    \node[circle, fill=red, inner sep=0, minimum size=4mm] (M2) at (8.5,1.5) {};
    \node[violet, circle, draw, anchor=center, inner sep=0, minimum size=1cm] at (M2.center) {};
    \draw[rounded corners] (*.40) .. controls ++(0.5,1) and ++(-3,-1) .. (7,1) arc (-90:90:0.5) .. controls ++(-3,-1) and ++(0.5,1) .. (*.50);
    
    \node (BRR) at ($(BR.west) + (-0.41,-0.1)$) {};
    \node (BRL) at ($(BR.east) + (0.55,0.22)$) {};
    \draw (*.0) .. controls ++(4,-0.5) and ++ (0.5,2) .. (BRR);
    \draw (*.350) .. controls ++(4,-0.5) and ++ (-1,0) .. (BRL);
    
    \draw (*.355) .. controls ++(1.5,-0.1) and ++(-0.5,1) .. (BR_out.center);
    
    \node (BLR) at ($(BL.west) + (-0.3,0.5)$) {};
    \node (BLL) at ($(BL.east) + (0.55,-0.05)$) {};
    \draw (*.180) .. controls ++(-3.5,-1) and ++(2,0) .. (BLR);
    \draw (*.170) .. controls ++(-5,-0.5) and ++(-0.5,1.5) .. (BLL);
    
    \draw (*.175) .. controls ++(-1.5,-0.1) and ++(0.5,1) .. (BL_out.center);
    
    \node (TRR) at ($(TR.east) + (-0.475,-0)$) {};
    \node (TRL) at ($(TR.west) + (0.25,-0.4)$) {};
    \draw (*.60) .. controls ++(2,5) and ++(1.5,-1) .. (TRR);
    \draw (*.70) .. controls ++ (1,3) and ++ (-1.5,0) .. (TRL);
    
    \draw (*.65) .. controls ++(0.25,1) and ++(-0.5,-1) .. (TR_out.center);
    
    \node (TLR) at ($(TL.east) + (-0.45,-0.4)$) {};
    \node (TLL) at ($(TL.west) + (0.65,-0.2)$) {};
    \draw (*.100) .. controls ++(-0.5,6) and ++(1,-0.25) .. (TLR);
    \draw(*.110) .. controls ++(-0.75,6) and ++(-1,-2) .. (TLL);
    
    \draw (*.105) .. controls ++(-0.25,5) and ++(0.5,-1) .. (TL_out.center);
    
    \end{tikzpicture}
    }}
\end{figure}

After further contracting the graph, we obtain 
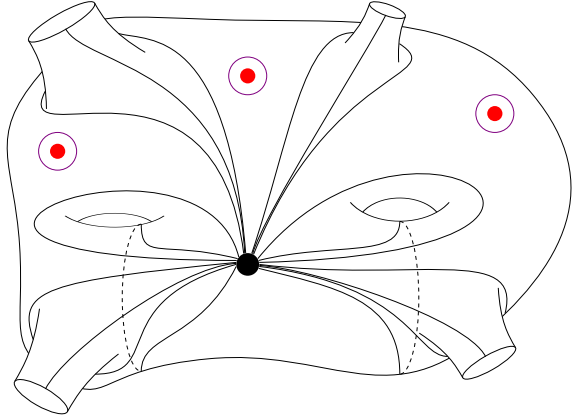
\begin{figure}[H]
    \centering
    \makebox[\textwidth][c]{
    \resizebox{0.5\width}{!}{
    \begin{tikzpicture}
        
        \begin{scope}[on background layer]
            \draw  plot[smooth cycle, thick, tension=1] coordinates {(-3.5,2.5) (1,4) (9.5,2) (8,-4.5) (1,-5) (-3,-5.5) (-4,-2)};
        \end{scope}
        
        \node[inner sep=-1mm] (LH) at (-1.5,-1) {
        \resizebox{0.75\width}{0.3\height}{
        \begin{tikzpicture}
        
            \FPset\rbig{2}
            \FPset\rsmol{1.5}
            
            \FPset\initang{-30}
            \FPset\finang{-150}
            \FPeval\rang{round(abs((\initang - \finang) / 2), 0)}
            
            \FPeval\dsmol{\rsmol * (1-cos(\rang * \FPpi / 180))}
            \FPeval\dbig{\rbig - root(2, (\rbig^2 - \rsmol^2 * sin(\rang * \FPpi / 180 )^2))}
            \FPeval\hsep{(\rsmol - \dsmol) + (\rbig - \dbig)}
            
            \node (c) at (0,0) {};
            \centerarc[thick](c)(\initang:\finang:\rbig cm) {big};
            \centerarc[thick, transform canvas={yshift=-\hsep cm}](c)(-\initang:-\finang:\rsmol cm);
        
        \end{tikzpicture}
        }
        };
    
        \node[inner sep=-1.5mm] (RH) at (6,-0.5) {
        \resizebox{0.75\width}{0.5\height}{
        \begin{tikzpicture}
        
            \FPset\rbig{2}
            \FPset\rsmol{1.5}
            
            \FPset\initang{-30}
            \FPset\finang{-150}
            \FPeval\rang{round(abs((\initang - \finang) / 2), 0)}
            
            \FPeval\dsmol{\rsmol * (1-cos(\rang * \FPpi / 180))}
            \FPeval\dbig{\rbig - root(2, (\rbig^2 - \rsmol^2 * sin(\rang * \FPpi / 180 )^2))}
            \FPeval\hsep{(\rsmol - \dsmol) + (\rbig - \dbig)}
            
            \node (c) at (0,0) {};
            \centerarc[thick](c)(\initang:\finang:\rbig cm) {big};
            \centerarc[thick, transform canvas={yshift=-\hsep cm}](c)(-\initang:-\finang:\rsmol cm);
        
        \end{tikzpicture}
        }
        };

    \node[inner sep=0mm, rotate=30] (TL) at (-2.5,3.25) {
    \resizebox{1\width}{1\height}{
    \begin{tikzpicture}[xscale=0.5,yscale=0.45]
    
        \fill[white] (0,0) arc (0:180:2 and 0.5)
        .. controls ++(0,-1) and ++(1,2) .. (-5,-4)
        -- (1,-4)
        .. controls ++(-1,2) and ++(0,-1) .. (0,0);
        
        \draw (0,0) arc (0:180:2 and 0.5)
        .. controls ++(0,-1) and ++(1,2) .. (-5,-4);
        \draw (1,-4) .. controls ++(-1,2) and ++(0,-1) .. (0,0);
        \draw (0,0) arc (0:-180:2 and 0.5);
    \end{tikzpicture}
    }
    };
    \node[anchor=center] (TL_out) at ($(TL.north) + (30-90:0.45)$) {};
    
    \node[inner sep=0mm, rotate=-20] (TR) at (5.5,3.75) {
    \resizebox{1\width}{1\height}{
    \begin{tikzpicture}[xscale=0.25,yscale=0.3]
    
        \fill[white] (0,0) arc (0:180:2 and 0.5)
        .. controls ++(0,-1) and ++(1,2) .. (-5,-4)
        -- (1,-4)
        .. controls ++(-1,2) and ++(0,-1) .. (0,0);
        
        \draw (0,0) arc (0:180:2 and 0.5)
        .. controls ++(0,-1) and ++(1,2) .. (-5,-4);
        \draw (1,-4) .. controls ++(-1,2) and ++(0,-1) .. (0,0);
        \draw (0,0) arc (0:-180:2 and 0.5);
    \end{tikzpicture}
    }
    };
    \node[anchor=center] (TR_out) at ($(TR.north) + (-20-90:0.3)$) {};
    
    \node[inner sep=0mm, rotate=-150] (BR) at (8,-4.5) {
    \resizebox{1\width}{1\height}{
    \begin{tikzpicture}[xscale=0.4,yscale=0.4]
    
        \fill[white] (0,0) arc (0:180:2 and 0.5)
        .. controls ++(0,-1) and ++(1,2) .. (-5,-4)
        -- (1,-4)
        .. controls ++(-1,2) and ++(0,-1) .. (0,0);
        
        \draw (0,0) arc (0:180:2 and 0.5)
        .. controls ++(0,-1) and ++(1,2) .. (-5,-4);
        \draw (1,-4) .. controls ++(-1,2) and ++(0,-1) .. (0,0);
        \draw (0,0) arc (0:-180:2 and 0.5);
    \end{tikzpicture}
    }
    };
    \node[anchor=center] (BR_out) at ($(BR.north) + (-150-90:0.4)$) {};
    
    \node[inner sep=0mm, rotate=150] (BL) at (-3,-5.25) {
    \resizebox{1\width}{1\height}{
    \begin{tikzpicture}[yscale=0.5,xscale=0.4]
        
        \fill[white] (0,0) arc (0:180:2 and 0.5)
        .. controls ++(0,-1) and ++(1,2) .. (-5,-4)
        -- (1,-4)
        .. controls ++(-1,2) and ++(0,-1) .. (0,0);
        
        \draw (0,0) arc (0:180:2 and 0.5)
        .. controls ++(0,-1) and ++(1,2) .. (-5,-4);
        \draw (1,-4) .. controls ++(-1,2) and ++(0,-1) .. (0,0);
        \draw (0,0) arc (0:-180:2 and 0.5);
        
    \end{tikzpicture}
    }
    };
    \node[anchor=center] (BL_out) at ($(BL.north) + (150-90:0.5)$) {};
    
    \node[circle, 
          fill,
          inner sep=0, 
          minimum size=6mm]
          (*) at (2,-2.5) {};

    \draw (*.140) .. controls ++(-1,4) and ++(-11.5,0) .. (*.160);
    
    \node (LHT) at ($(LH.south) + (0,0.075)$) {};
    \node[anchor=center] (LHB) at ($(LHT) + (0.7,-3.9)$) {};
    \draw
        (*.150) .. controls ++(-2,0.5) and ++(0,-1) .. ($(LHT) + (0.7,0)$);
    \draw[dashed] ($(LHT) + (0.7,0)$) arc (90:-90:-0.5 and 1.95);
    \draw
        (*.190).. controls ++(-1,-2) and ++(0,1) .. ($(LHB) + (0,0)$);
    
    \draw (*.10) .. controls ++(11,0.25) and ++(4,5) .. (*.30);
    
    \node (RHT) at (RH.south) {};
    \node[anchor=center] (RHB) at (6,-5.405) {};
    \draw
        (*.20) .. controls ++(2,0.5) and ++(0,-1) .. ($(RHT) + (0,0)$);
    \draw[dashed] (RHT) arc (90:-90:0.5 and 2.025);
    \draw
        (*.340).. controls ++(2,-0.5) and ++(0,2) .. (RHB.center);
        
    \node[circle, fill=red, inner sep=0, minimum size=4mm] (M1) at (2,2.5) {};
    \node[violet, circle, draw, anchor=center, inner sep=0, minimum size=1cm] at (M1.center) {};

    \node[circle, fill=red, inner sep=0, minimum size=4mm] (M0) at (-3,0.5) {};
    \node[violet, circle, draw, anchor=center, inner sep=0, minimum size=1cm] at (M0.center) {};
   
    \node[circle, fill=red, inner sep=0, minimum size=4mm] (M2) at (8.5,1.5) {};
    \node[violet, circle, draw, anchor=center, inner sep=0, minimum size=1cm] at (M2.center) {};
    
    \node (BRR) at ($(BR.west) + (-0.41,-0.1)$) {};
    \node (BRL) at ($(BR.east) + (0.55,0.22)$) {};
    \draw (*.0) .. controls ++(4,-0.5) and ++ (0.5,2) .. (BRR);
    \draw (*.350) .. controls ++(4,-0.5) and ++ (-1,0) .. (BRL);
    
    \draw (*.355) .. controls ++(1.5,-0.1) and ++(-0.5,1) .. (BR_out.center);
    
    \node (BLR) at ($(BL.west) + (-0.3,0.5)$) {};
    \node (BLL) at ($(BL.east) + (0.55,-0.05)$) {};
    \draw (*.180) .. controls ++(-3.5,-1) and ++(2,0) .. (BLR);
    \draw (*.170) .. controls ++(-5,-0.5) and ++(-0.5,1.5) .. (BLL);
    
    \draw (*.175) .. controls ++(-1.5,-0.1) and ++(0.5,1) .. (BL_out.center);
    
    \node (TRR) at ($(TR.east) + (-0.475,-0)$) {};
    \node (TRL) at ($(TR.west) + (0.25,-0.4)$) {};
    \draw (*.60) .. controls ++(2,5) and ++(1.5,-1) .. (TRR);
    \draw (*.70) .. controls ++ (1,3) and ++ (-1.5,0) .. (TRL);
    
    \draw (*.65) .. controls ++(0.25,1) and ++(-0.5,-1) .. (TR_out.center);
    
    \node (TLR) at ($(TL.east) + (-0.45,-0.4)$) {};
    \node (TLL) at ($(TL.west) + (0.65,-0.2)$) {};
    \draw (*.100) .. controls ++(-0.5,6) and ++(1,-0.25) .. (TLR);
    \draw (*.110) .. controls ++(-0.75,6) and ++(-1,-2) .. (TLL);
    
    \draw (*.105) .. controls ++(-0.25,5) and ++(0.5,-1) .. (TL_out.center);
    
    \end{tikzpicture}
    }}
\end{figure}

\noindent
where we can restrict to colorings by simple objects. Thus for any two choices of marked points $M_1$, $M_2$, by identifying the basis elements, we have a distinguished isomorphism between string-net spaces defined in terms of $G$-triangulations based at $\delta\cup M_1$ and $\delta\cup M_2$. The isomorphisms give a projective system  of vector spaces and we can take its projective limit, which we denote by $\SNrm^\Ccal(\Sigma;\Bbf)$.

\begin{defn}
The vector space $\SNrm^\Ccal(\Sigma;\Bbf)$ is called \textit{bare string-net space}.  
\end{defn}

To summarize our construction, we first needed a combinatorial replacement for a surface $\Sigma$ with map $\zeta:\Sigma\rightarrow BG$, in order to give a sensible definition for a string-net graph in $\Sigma$ to respect the $G$-structure. This point was settled by using $G$-triangulations. Once we came up with a definition of a pre-string-net space depending on a given $G$-triangulation, the whole rest of the section was devoted to build up a definition depending only on $\Sigma$ and $\zeta$, but not on the combinatorial model. Taking limits over projective systems only involving isomorphisms has the advantage that, whenever we want to work with the bare string-net space, we can just pick our favorite $G$-triangulation and compute everything relative to it. We will make use of this fact in section \ref{computations}, when actually computing string-net spaces. 

\begin{rem} Allowing an additional set $M$ of vertices in an ideal triangulation will simplify the description of the behavior of string-net spaces under gluing of $G$-surfaces described in section \ref{61}. The set of $G$-triangulations based on boundary points of surfaces is not preserved under gluing of surfaces since gluing gives a new internal vertex. Hence, we had to allow more general $G$-triangulations from the start.
\end{rem}

\subsection{Cylinder Category}

Since we want to construct the $(2,1)$-part of a functor
\eq{
Z:\GBordpzc_\star(3,2,1)\rightarrow \BiModpzc_\Kbb
}
a surface $\Sigma$ with $n$-boundary components should get mapped to a profunctor
\eq{
Z_\Sigma:\underbrace{\Cylpzc(\Ccal)^{\epsilon_1}\boxtimes\cdots\boxtimes \Cylpzc(\Ccal)^{\epsilon_n}}_{n}\rightarrow \Vectpzc_\Kbb
}
with $\Cylpzc(\Ccal)\coloneqq Z(S^1)$ or its dual, depending on the orientation of the boundary component. E.g. a pair of pants with two incoming and one outgoing boundary endows $\Cylpzc(S^1)$ with a monoidal structure. For a discussion of how a once extended $G$-equivariant HTQFT gives $\Cylpzc(S^1)$ the structure of a $G$-braided category see \cite{schweigert2020extended}.

Thus we need to get the value of the string-net functor on a circle with additional $G$-label $g$ (cf. Section 4). That is, we need to compute the cylinder category.
Since in the non-equivariant case of the string-net construction, the circle category is equivalent to the Drinfeld center of the input category, it is reasonable to expect that equivariant string-nets on a $g$-labeled circle compute the $g$-homogeneous component $\Zsf_G(\Ccal)_g$ of the $G$-center $\Zsf_G(\Ccal)$. 

\begin{rem}\label{remIntro}
In \cite{turaev20203} the cylinder category is chosen to be $\Zsf_G(\Ccal)$ by choosing a $\Zsf_G(\Ccal)$-coloring on surfaces with boundary. It is a nice feature of the string-net construction, that one doesn't need to make that choice, but one can actually compute the category. The only input needed is a $G$-graded spherical fusion category $\Ccal$, which is the supposed to be the value of a fully extended HTQFT on a point. 

At the same time, the following computation is a check that our construction is really sensible.
\end{rem}

Let $S^1_g$ be the $g$-labeled circle. Recall that $S^1_g$ has a distinguished base point $\star \in S^1_g$. We define a category $\widetilde{\Cylpzc}_g$ as follows. Objects are sets of $\Ccal$-labeled, signed marked points $(\mbf,\cbf)\coloneqq \lbr (m_1^{\epsilon_1},c_1^{\epsilon_1}),\cdots, (m_N^{\epsilon_N},c_N^{\epsilon_N})\rbr$. The points are linearly ordered by starting at $\star$ and going around the circle in counterclockwise direction. The object $c_i^{\epsilon_i}$ has underlying $G$-color $g_i^{\epsilon_i}$ and we require 
\eq{
g_1^{\epsilon_1}\cdots g_N^{\epsilon_N}=g\; . 
}
In the non-equivariant string-net construction, the morphism space of the cylinder category is defined to be the string-net space, with the appropriate boundary value, on the cylinder. Similarly, in our equivariant setting, $\hom_{\widetilde{\Cylpzc}_g}\lb (\mbf,\cbf),(\nbf,\dbf)\rb $ is given is by the string-net spaces, where we have to additionally specify a map to the classifying space $BG$. On the circle this map is given by $g$. It is appropriate for cylinder categories to require it to be given by the neutral element $e\in G$ for the ``radial'' direction, cf. figure \ref{radial direction}. Composition of morphisms is given by gluing cylinders and concatenating string-nets. 

\begin{prop} Let $\Cylpzc_g$ be the Karoubi-envelope of $\widetilde{\Cylpzc}_g$. There is an equivalence
\eq{
\Cylpzc_g\simeq \Zsf_G(\Ccal)_g\quad .
}
\end{prop}
\begin{proof}
An object in $\Cylpzc_g$ is a pair $((\mbf,\cbf),p)$ of an object in $\widetilde{\Cylpzc}_g$ and an idempotent $p$ in $\widetilde{\Cylpzc}_g$. We pick a $G$-triangulation of the cylinder without internal marked points. It is not hard to see, that $p$ can be represented by a graph shown in red in figure \ref{radial direction}.

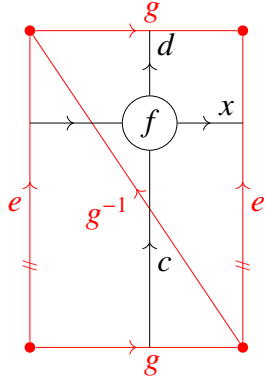
\begin{figure}[H]
    \makebox[\textwidth][c]{
    \resizebox{0.7\width}{!}{
    \begin{tikzpicture}
        \node[circle, fill=red, inner sep=0, minimum size=2mm] (a) at (2,3) {};
        \node[circle, fill=red, inner sep=0, minimum size=2mm] (b) at (2,-3) {};
        \node[circle, fill=red, inner sep=0, minimum size=2mm] (c) at (-2,-3) {};
        \node[circle, fill=red, inner sep=0, minimum size=2mm] (d) at (-2,3) {};
        
        \draw[red, mid={<}{0.5}] (a) -- (b) 
            node[near end, rotate=90] {\scriptsize //} 
            node[anchor=north west, midway] {\Large $e$};
        \draw[red, mid={<}{0.5}] (d) -- (c) 
            node[near end, rotate=90] {\scriptsize //} 
            node[anchor=north east, midway] {\Large $e$};
        \draw[red, mid={>}{0.5}] (c) -- (b) 
            node[anchor=north west, midway] {\Large $g$};
        \draw[red, mid={>}{0.5}] (d) -- (a) 
            node[anchor=south west, midway] {\Large $g$};
        \draw[red, mid={>}{0.5}] (b) -- (d) 
            node[anchor=north east, midway] {\Large $g^{-1}$};
        
        \node[circle, draw] (e) at (0.25,1.25) {\Large $f$};
        \draw[mid={>}{0.5}] (e) -- ++(1.75,0) 
            node[anchor=south west, midway] {\Large $x$};
        \draw[mid={<}{0.6}] (e) -- ++(-2.25,0);
        \draw[mid={>}{0.5}] (e) -- ++(0,1.75)
            node[anchor=south west, midway] {\Large $d$};
        \draw[mid={<}{0.5}] (e) -- ++(0,-4.25)
            node[anchor=north west, midway] {\Large $c$};
        
    \end{tikzpicture}
    }
    }
        \caption{In red we show an ideal triangulation of the cylinder with in- and outgoing boundary $S^1_g$. The vertical red lines are along the ``radial direction'' of the cylinder.}
    \label{radial direction}
\end{figure}

\noindent
with $c,d\in \Ccal_g$ the signed tensor products of objects $\cbf$, $\dbf$ and $x\in \Ccal_e$. We define a functor $F:\Cylpzc_g\rightarrow \Zsf_G(\Ccal)_g$, whose action on objects is given by
\eq{
F:((\mbf,\cbf),p)\longmapsto (I(c),F(p))\, 
} 
where $I$ is the induction functor (cf. \eqref{induction functor}) and $F(p)$ is the idempotent shown in the following figure:
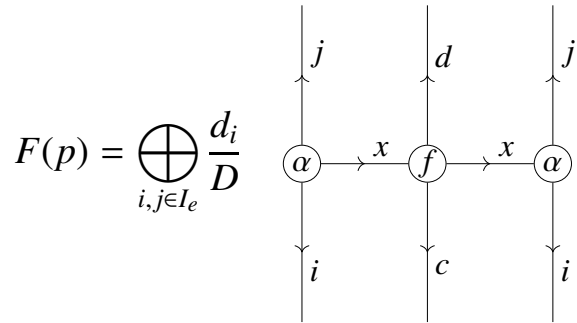
\begin{figure}[H]
    \makebox[\textwidth][c]{
    \resizebox{0.7\width}{!}{
    \begin{tikzpicture}
    
        \node (eq) at (0,0) {\LARGE $\displaystyle{F(p) = \bigoplus_{i,j \in I_e} \frac{d_i}{D}}$};
        \node[circle, 
              draw, 
              inner sep=0, 
              minimum size=7mm] 
              (a) at ($(eq.east) + (1,0)$) {\Large $\alpha$};
        \node[circle, 
              draw, 
              inner sep=0, 
              minimum size=7mm] 
              (b) at ($(a.east) + (2,0)$) {\Large $f$};
        \node[circle, 
              draw, 
              inner sep=0, 
              minimum size=7mm] 
              (c) at ($(b.east) + (2,0)$) {\Large $\alpha$};
        
        \draw[mid={>}{0.5}] (a) -- (b)
            node[midway, anchor=south west] {\Large $x$};
        \draw[mid={>}{0.5}] (b) -- (c)
            node[midway, anchor=south west] {\Large $x$};
            
        \draw[mid={>}{0.5}] (a) -- ++(0,3)
            node[midway, anchor=south west] {\Large $j$};
        \draw[mid={>}{0.5}] (b) -- ++(0,3)
            node[midway, anchor=south west] {\Large $d$};
        \draw[mid={>}{0.5}] (c) -- ++(0,3)
            node[midway, anchor=south west] {\Large $j$};
        \draw[mid={>}{0.5}] (a) -- ++(0,-3)
            node[midway, anchor=north west] {\Large $i$};
        \draw[mid={>}{0.5}] (b) -- ++(0,-3)
            node[midway, anchor=north west] {\Large $c$};
        \draw[mid={>}{0.5}] (c) -- ++(0,-3)
            node[midway, anchor=north west] {\Large $i$};    
    \end{tikzpicture}
    }
    }
\end{figure}
where $f$ is the morphism in the representation of $p$ given above. Its action on morphisms is defined in the obvious way. In the other direction, we define a functor $K:\Zsf_G(\Ccal)_g\longrightarrow \Cylpzc_g$ mapping $Z\in \Zsf_G(\Ccal)_g$ to the underlying object in $\Ccal_g$ and the idempotent  given by
\begin{figure}[H]
    \makebox[\textwidth][c]{
    \resizebox{0.7\width}{!}{
    \begin{tikzpicture}
    
        \node[circle, fill=red, inner sep=0, minimum size=2mm] (a) at (2,3) {};
        \node[circle, fill=red, inner sep=0, minimum size=2mm] (b) at (2,-3) {};
        \node[circle, fill=red, inner sep=0, minimum size=2mm] (c) at (-2,-3) {};
        \node[circle, fill=red, inner sep=0, minimum size=2mm] (d) at (-2,3) {};
        \draw[red, mid={<}{0.5}] (a) -- (b) 
            node[near end, rotate=90] {\scriptsize //} 
            node[anchor=north west, midway] {\Large $e$};
        \draw[red, mid={<}{0.5}] (d) -- (c) 
            node[near end, rotate=90] {\scriptsize //} 
            node[anchor=north east, midway] {\Large $e$};
        \draw[red, mid={>}{0.5}] (c) -- (b) 
            node[anchor=north west, midway] {\Large $g$};
        \draw[red, mid={>}{0.5}] (d) -- (a) 
            node[anchor=south west, midway] {\Large $g$};
        \draw[red, mid={>}{0.5}] (b) -- (d) 
            node[anchor=north east, midway] {\Large $g^{-1}$};
        
        \draw[violet] (-2,1.5) -- ++(4,0);
        \node[fill=white] at (0.5,1.5) {};
        \draw[mid={>}{0.8}] (0.5,-3) -- ++(0,6)
            node[near end, yshift=5mm, anchor=west] {\Large $Z$};

        \node[anchor=west] (eq) at (2.5,0) {\LARGE $\displaystyle{=\sum_{i \in I_e} \frac{d_i}{D^2}}$};
        \node (new_orig) at ($(eq.east) + (2.5,0)$) {};

        \node[circle, fill=red, inner sep=0, minimum size=2mm] (a2) at ($(new_orig) + (2,3)$) {};
        \node[circle, fill=red, inner sep=0, minimum size=2mm] (b2) at ($(new_orig) + (2,-3)$) {};
        \node[circle, fill=red, inner sep=0, minimum size=2mm] (c2) at ($(new_orig) + (-2,-3)$) {};
        \node[circle, fill=red, inner sep=0, minimum size=2mm] (d2) at ($(new_orig) + (-2,3)$) {};
        \draw[red, mid={<}{0.5}] (a2) -- (b2) 
            node[near end, rotate=90] {\scriptsize //} 
            node[anchor=north west, midway] {\Large $e$};
        \draw[red, mid={<}{0.5}] (d2) -- (c2) 
            node[near end, rotate=90] {\scriptsize //} 
            node[anchor=north east, midway] {\Large $e$};
        \draw[red, mid={>}{0.5}] (c2) -- (b2) 
            node[anchor=north west, midway] {\Large $g$};
        \draw[red, mid={>}{0.5}] (d2) -- (a2) 
            node[anchor=south west, midway] {\Large $g$};
        \draw[red, mid={>}{0.5}] (b2) -- (d2) 
            node[anchor=north east, midway] {\Large $g^{-1}$};
        
        \draw[mid={>}{0.5}] ($(new_orig) + (-2,1.5)$) -- ++(4,0)
            node[midway, anchor=south east] {\Large $i$};
        \node[fill=white] at ($(new_orig) + (0.5,1.5)$) {};
        \draw[mid={>}{0.8}] ($(new_orig) + (0.5,-3)$) -- ++(0,6)
            node[near end, yshift=5mm, anchor=west] {\Large $Z$};

    \end{tikzpicture}
    }}
    \caption{Part of the functor $K$ from $G$-center to the cylinder category.}
    \label{cylinder idempotent}
\end{figure}
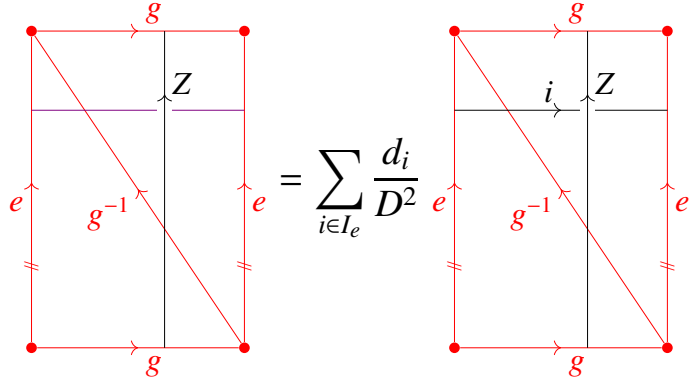
i.e. the crossing is the half-braiding in $\Zsf_G(\Ccal)_g$. The proof, that the functors $F$, $K$ give an equivalence of  categories is equivalent to the one given in \cite[Theorem~6.4]{kirillov2011string}.
\end{proof}

\section{$G$-equivariant String-Nets}

\subsection{Construction of $G$-equivariant String-Net Space}\label{61}

In this section, we demonstrate how to extract structure of the category $\Zsf_G(\Ccal)$ from $G$-equivariant string-nets. To this end, we have to be able to assign to each connected component of the boundary of $\Sigma$ an object $Z\in \Zsf_G(\Ccal)_g$. Given $Z$, consider the $G$-equivariant string-net given by the cylinder 

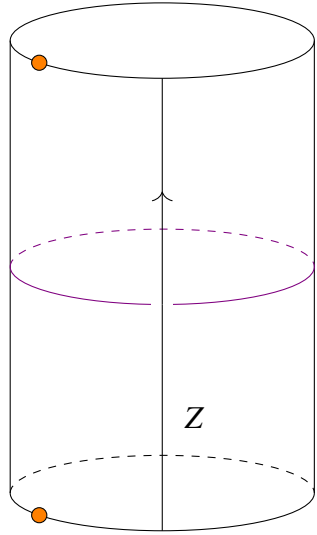
\begin{figure}[H]
    \centering
    \begin{tikzpicture}
        \def\rx{2}
        \def\ry{0.5}
        \def\rz{6}
        
        \draw (-\rx,0) -- ++(0,\rz);
        \draw (\rx,0) -- ++(0,\rz);
        
        \draw[dashed] (-\rx,0) arc (180:0:\rx cm and \ry cm);
        \draw (-\rx,0) arc (180:360:\rx cm and \ry cm)
            node[pos=0.2, circle, inner sep=0, minimum size=2mm,draw, fill=orange] {};
        
        \draw (-\rx, \rz) arc (180:-180:\rx cm and \ry cm) 
            node[pos=0.9, circle, inner sep=0, minimum size=2mm,draw, fill=orange] {};

        \draw[violet, dashed] (-\rx,0.5*\rz) arc (180:0:\rx cm and \ry cm);
        \draw[violet] (-\rx,0.5*\rz) arc (180:360:\rx cm and \ry cm)
            node[pos=0.5, fill=white] (b) {};
        
        \foreach \x in {b}
        {
            \draw (\x.center) -- ++(0, -0.5*\rz) node[pos=0.5,label=right:{$Z$}] {};
            \draw[mid] (\x.center) -- ++(0, 0.5*\rz);
        }
    
    
    \end{tikzpicture}
\end{figure}

where the purple line is a cloaking circle. This is exactly the idempotent of figure \ref{cylinder idempotent}. Gluing such a cylinder to a boundary component defines an idempotent on the corresponding bare string-net space. 

Given a boundary value $Z_i\in \Zsf_G(\Ccal)$ for each connected component of the boundary, we define the string-net space $\KSNrm^\Ccal(\Sigma;\Zbf)\coloneqq \KSNrm^\Ccal(\Sigma;Z_1,\dots,Z_n)$ for these boundary values as the common image of all these idempotents. 
 
The cylinder category together with the string-net space should be the $(2,1)$-part of functor 
\eq{
\KSNrm^\Ccal:\GBordpzc_\star(3,2,1)\rightarrow \BiModpzc_\Kbb \quad .
}
Thus, we have to discuss how string-net spaces behave under gluing of surfaces. Let $\Sigma$, $\Sigma^\prime$ be composable $1$-morphisms in $\GBordpzc_\star(3,2,1)$, then we need to show
\eq{\label{gluing iso}
\KSNrm^\Ccal(\Sigma\circ \Sigma^\prime;\Zbf,\Zbf^\prime)\simeq\int^{Z\in \Zsf_G(\Ccal)}\KSNrm^\Ccal(\Sigma;\Zbf,Z)\otimes \KSNrm^\Ccal(\Sigma^\prime;Z^\ast,\Zbf^\prime)\, , 
}
where $Z, \, Z^\ast\in \Zsf_G(\Ccal)$ are the boundary values of the boundary components at which $\Sigma$, $\Sigma^\prime$ are glued.  
Let $\Sigma$, $\Sigma^\prime$ be equipped with $G$-triangulations $\Tsf$, $\Tsf^\prime$. The glued surface $\Sigma\circ\Sigma^\prime$ inherits a $G$-triangulation, where the boundary vertices of the individual surfaces are mapped to  new internal vertices. For $\Sigma$ we pick a $G$-triangulation which looks around the gluing boundary like the $G$-triangulation shown in figure \ref{cylinder idempotent}.
 Using the chosen $G$-triangulation on an annular neighborhood around the gluing boundary, the proof of \eqref{gluing iso} is the same as in the non-equivariant case, c.f. \cite[section~5.2]{Goosen}. 

The alert reader may have noticed, that we haven't defined a string-net space on surfaces without boundary components. However, we can consider $(\Sigma^\prime, \zeta^\prime)$, a surface with a single boundary component $b$ and $\zeta|_{b}=e\in G$ the neutral element. Furthermore, $(\Sigma,\zeta)$ is the surface obtained from $(\Sigma^\prime, \zeta^\prime)$ by gluing a disk $D$ at $b$. Note that all closed $G$-surfaces can be obtained in this way and we simply define 
\eq{
\KSNrm^\Ccal(\Sigma)\coloneqq \int^{Z\in \Zsf_G(\Ccal)_e}\KSNrm^\Ccal(D;Z)\otimes \KSNrm^\Ccal(\Sigma;Z^\ast)\, .
}  
We expect that, as in the non-equivariant case, the maps
\eq{
\Cylpzc:(S^1,g)&\rightarrow \Cylpzc_g\\
\KSNrm^\Ccal: (\Sigma,\zeta)&\rightarrow \KSNrm^\Ccal(\Sigma)
}
comprise the $(2,1)$-part of a symmetric monoidal functor $\KSNrm^\Ccal:\GBordpzc_\star(3,2,1)\rightarrow \BiModpzc_\Kbb$.

\subsection{Computations of the $G$-String-Net Space}\label{computations}

In this section we compute the string-net space for different surfaces. We present the computation for a cylinder, a pair of pants and a genus two surface with three boundary components. The cylinder computation shows how to recover the $G$-crossing on $\Zsf_G(\Ccal)$ and the pair of pants gives the monoidal product. The higher genus computation connects our construction to the $G$-equivariant Turaev-Viro construction of \cite{turaev20203}. However, as we are using a slightly different input datum, we cannot get a direct isomorphism between vector spaces for $G$-surfaces. Nevertheless, we get an expression for the string-net space in terms of morphism space in the $G$-center, which is closely related to the result in \cite{turaev20203}. Finally we note that computing the string-net space for all other surfaces is analogous to the procedures presented here. 

\subsubsection{Cylinder}\label{cylinderSection}

We start by computing the string-net space on a cylinder. Since a cylinder exhibits a homotopy between the $G$-bundles on its boundary, the boundary circles have to carry conjugate $G$-labels. We pick an ideal triangulation and note that any string-net is a sum of string-nets of the following type

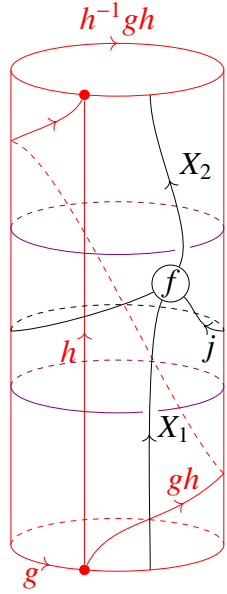
\begin{figure}[H]
\makebox[\textwidth][c]{
    \resizebox{0.7\width}{!}{
    \centering
    \begin{tikzpicture}
        \def\rx{2}
        \def\ry{0.5}
        \def\rz{9}
        
        \draw[red] (-\rx,0) -- ++(0,\rz);
        \draw[red] (\rx,0) -- ++(0,\rz);
        \draw[red, dashed] (-\rx,0) arc (180:0:\rx cm and \ry cm);
        \draw[red, mid={>}{0.2}] (-\rx,0) arc (180:360:\rx cm and \ry cm)
            node[pos=0.2, anchor=north] {\Large $g$}
            node[pos=0.6] (a) {}
            node[pos=0.4, circle, fill=red, inner sep=0, minimum size=2mm] (g1) {};
        \draw[red, mid={>}{0.25}] (-\rx, \rz) arc (180:-180:\rx cm and \ry cm)
            node[pos=0.25, anchor=south] {\Large $h^{-1}gh$}
            node[pos=0.7] (b) {}
            node[pos=0.8, circle, fill=red, inner sep=0, minimum size=2mm] (g2) {};
        \draw[red, mid] (g1) -- (g2) 
            node[midway, anchor=north east] {\Large $h$};
        \draw[red, mid={>}{0.7}] (g1) .. controls ++(0.5,1) and ++(-1,-1) .. (\rx,0.15*\rz)
            node[pos=0.75, anchor=south] {\Large $gh$};
        \draw[red, dashed] (\rx,0.15*\rz) .. controls ++(-1,1) and ++(1,-0.25) .. (-\rx,0.85*\rz);
        \draw[red, mid={>}{0.6}] (-\rx,0.85*\rz) .. controls ++(0.5,0.25) and ++(-0.25,-0.5) .. (g2);
    
        \draw[violet, dashed] (-\rx,1/3*\rz) arc (180:0:\rx cm and \ry cm);
        \draw[violet] (-\rx,1/3*\rz) arc (180:360:\rx cm and \ry cm)
            node[pos=0.6, fill=white] (c1) {};
        \draw[violet, dashed] (-\rx,2/3*\rz) arc (180:0:\rx cm and \ry cm);
        \draw[violet] (-\rx,2/3*\rz) arc (180:360:\rx cm and \ry cm)
            node[pos=0.71, fill=white] (c1) {};
        
        \node[circle, inner sep=0, minimum size=7mm, draw] (f) at (0.5*\rx, 0.55*\rz) {\Large $f$};
        \draw[mid] (a.center) .. controls ++(0,1) and ++(-0.5,-1) .. (f)
            node[midway, anchor=west] {\Large $X_1$};
        \draw[mid] (f) .. controls ++(0.5,1) and ++(0,-1) .. (b.center)
            node[midway, anchor=south west] {\Large $X_2$};
        \draw[mid={<}{0.75}] (f) .. controls ++(0.5,-0.5) and ++(-0.5,0) .. (\rx,0.45*\rz)
            node[pos=0.75, anchor=north] {\Large $j$};
        \draw[dashed] (-\rx,0.45*\rz) arc (180:0:\rx cm and \ry cm);
        \draw (-\rx,0.45*\rz) .. controls ++(0.5,0) and ++(-1,-0.5) .. (f);
        
    \end{tikzpicture}
    }
    }
    \caption{Basic string-net on a cylinder.}
    \label{basic string net cylinder}
\end{figure}

where $j\in I_h(\Ccal)$ for any $h\in G$. We denote the cylinder with lower/ upper boundary labeled by $g$/ $h^{-1}gh$ by $\prescript{}{g}{C}_{h^{-1}gh}$. From now on, we will not show the ideal triangulation in the arguments. Pictures get overloaded fast and besides establishing a useful presentation for string-nets, an ideal triangulation doesn't serve any further purpose.

\begin{prop}\label{cyl hom space} For $X_1\in \Zsf_G(\Ccal)_g$ and $X_2\in \Zsf_G(\Ccal)_{h^{-1}gh}$ we have 
\eq{
\KSNrm^\Ccal\lb \prescript{}{g}{C}_{h^{-1}gh};X_1,X_2 \rb\simeq \hom_{\Zsf(\Ccal)}(\phi_h(X_1),X_2)\,,
}
where $\phi_h$ is the $G$-crossing on $\Zsf_G(\Ccal)$ from Proposition \ref{new crossing}. In this way $G$-equivariant string-nets encode the $G$-action on $\Zsf_G(\Ccal)$.
\end{prop}

\begin{proof}
We can use the lower cloaking circle in figure \ref{basic string net cylinder} and the completeness relation to bring the $j$-labeled line and the other cloaking circle to the front

    \begin{equation}\label{cylinder basic form}
    \begin{minipage}[c]{0.2\textwidth}
    \makebox[\textwidth][c]{
    \resizebox{0.7\width}{!}{
    \centering

    }
    }
      \end{minipage}
    \end{equation}

This also changes the color of the lower cloaking circle. The lower three consecutive undercrossings in \eqref{cylinder basic form} constitute half-braidings $\gamma_{x_1,i\otimes n_0\otimes j^\ast}$, where $n_0\in I_e(\Ccal)$ is a simple object in the neutral component of $\Ccal$ coming from the cloaking circle in the middle.    
  
We define linear maps 
\begin{center}
%
    }
    }
\end{figure}
\end{center}

We claim that the two maps are inverse to each other. Applying $G\circ F$ to the string-net \eqref{cylinder basic form} yields

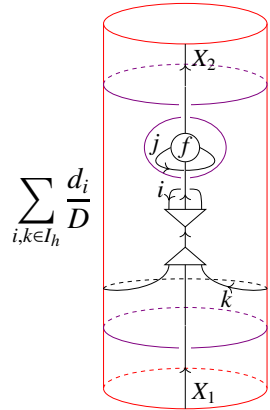
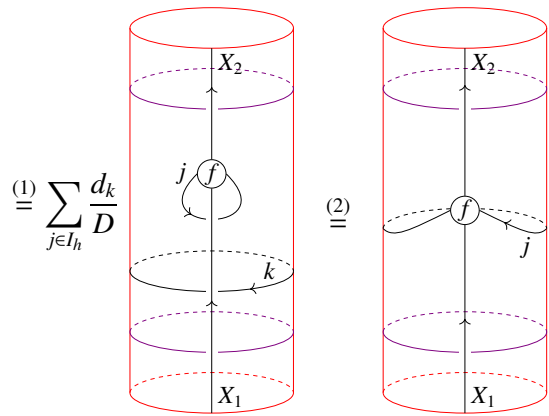
\begin{figure}[H]
    
    \makebox[\textwidth][c]{
    \resizebox{0.5\textwidth}{!}{
    %
}}
\end{figure}
In $(1)$ we use the definition of the idempotent $\pi(X_2)$ and then the completeness relation to get rid of the cloaking circle around $f$. Step $(2)$ is 
again using the completeness relation, this time with the $I_h$-colored cloaking circle.

For the other composition, let $f'\in \hom_{\Zsf(\Ccal)}(\phi_h(X_1),X_2)$, then
\begin{figure}[H]
 \makebox[\textwidth][c]{
    \resizebox{0.5\textwidth}{!}{
    \centering
    \def\a{1}
    %
        };
    \end{tikzpicture}
    }
    }
\end{figure}
Equality $(1)$ is the definition of the half-braiding $\gamma_{\phi_h(X_1),\bullet}$. $(2)$ holds, because $f'$ is already a morphism in the $G$-center.

\end{proof}

\subsubsection{Pair of Pants}

Similar to the cylinder, a pair of pants should equip the cylinder category with additional structure which in this case is the functor underlying a monoidal product. As a second consistency check, we prove that the equivariant string-net construction really induces the monoidal product in $\Zsf_G(\Ccal)$. We use the symbol $P_{g,h}^{gh}$ for the following surface with the shown boundary holonomies

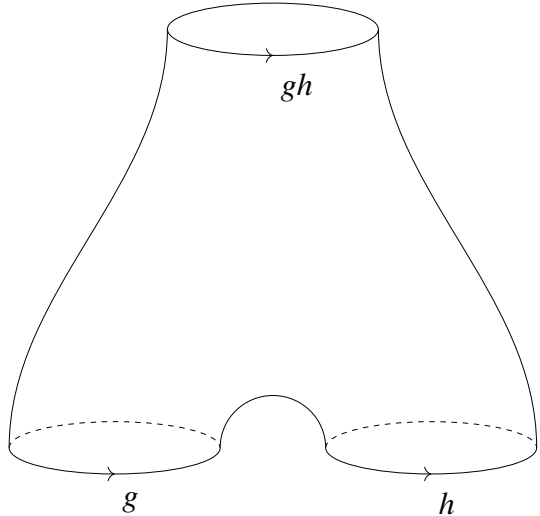
\begin{figure}[H]
\makebox[\textwidth][c]{
    \resizebox{0.5\textwidth}{!}{
    \centering
    \def\rx{2}
    \def\ry{0.5}
    

    }
    }
\end{figure}

\begin{prop}\label{monoidal prod} For $X_1\in \Zsf_G(\Ccal)_g$, $X_2\in \Zsf_G(\Ccal)_h $ and $X_3\in \Zsf_G(\Ccal)_{gh}$. Then it holds 
\eq{
\KSNrm^\Ccal(P_{g,h}^h; X_1, X_2,X_3)\simeq \hom_{\Zsf_G(\Ccal)}\lb X_1\otimes X_2,X_3\rb\, .
}

In this way $G$-equivariant string-nets encode the monoidal structure of $\Zsf_G(\Ccal)$.
\end{prop}

\begin{proof}
The proof is very similar as the previous one. We start with an ideal $G$-triangulation on a pair of pants

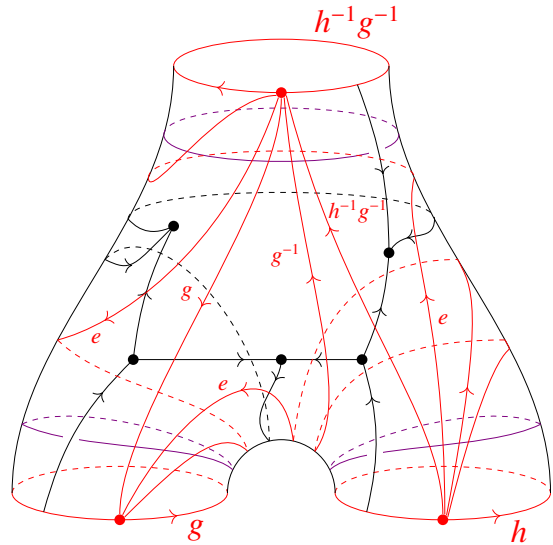
\begin{figure}[H]
\makebox[\textwidth][c]{
    \resizebox{0.5\textwidth}{!}{
    \centering
    %
    }
    }
\end{figure}

In black we also included the dual uni-trivalent fat graph. We pick a maximal subtree in the fat graph and collapse the corresponding string-net to get

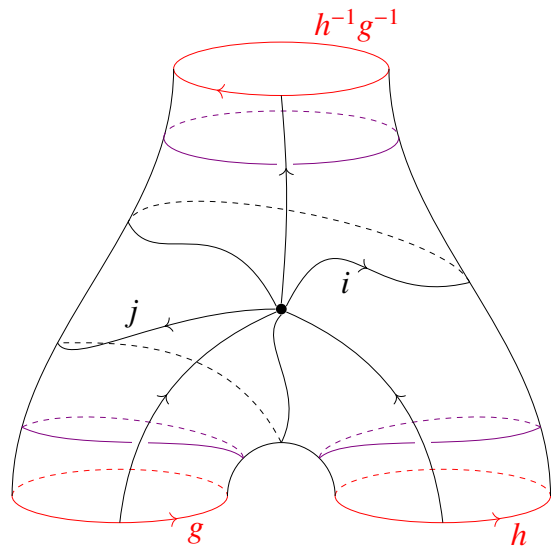
\begin{figure}[H]
\makebox[\textwidth][c]{
    \resizebox{0.5\textwidth}{!}{
    \centering
    %
    }
    }
\end{figure}

Note that due to the $G$-coloring in the ideal triangulation it holds $i,j\in I_e$. Using the completeness relation in $\Ccal$ we get a string-net

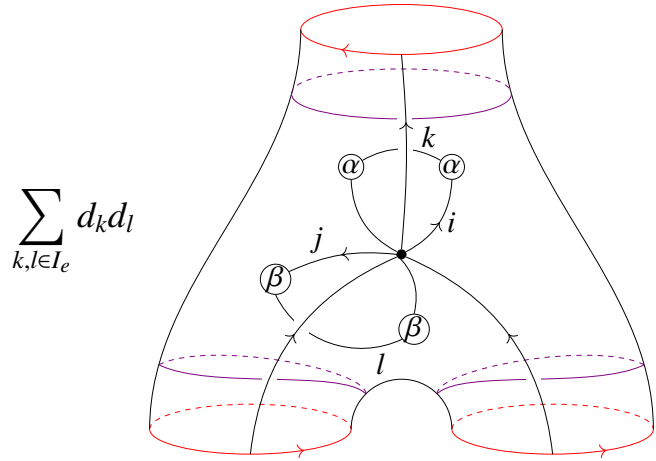
\begin{figure}[H]
    \makebox[\textwidth][c]{
    \resizebox{0.6\textwidth}{!}{
    %
}}
\end{figure}

which has an expansion in terms of string-nets

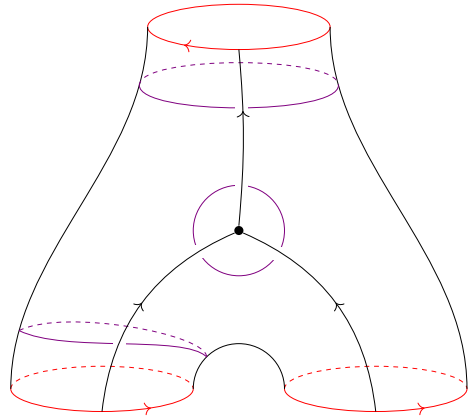
\begin{figure}[H]
    \makebox[\textwidth][c]{
    \resizebox{0.6\width}{!}{
    %
}}
\end{figure}

From that presentation of the string-net the isomorphism is obvious.
\end{proof}
\subsubsection{Higher Genus Surfaces}\label{higher genus}
As an example we consider $\Sigma$ a surface of genus $2$ with three boundary components. We pick a $G$-triangulation for $\Sigma$ with no internal vertices.

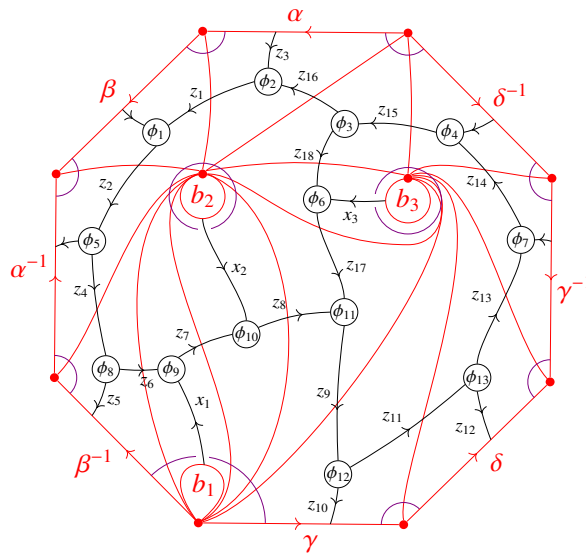
\begin{figure}[H]
    \makebox[\textwidth][c]{
    \resizebox{0.55\textwidth}{!}{
    %
    }
    }
\end{figure}

In the figure, the $G$-triangulation is shown in red. Its bounding octagon gives a polygonal decomposition of a genus $2$ surface, and the edges of the polygon are labeled with group elements $\alpha$, $\beta$, $\gamma$ and $\delta$. Boundary edges have $G$-label $b_1$, $b_2$ and $b_3$. In black we show a string-net on the surface, where we already collapsed all planar graphs inside faces of the $G$-triangulation to one-vertex graphs with labels $\phi_i$. The edges labeled by $\alpha,\beta,\gamma,\delta$ and $b_1,b_2,b_3$ generate the first homology group of $\Sigma$. By Lemma \ref{fat graph basis} we have a basis for $\KSNrm^\Ccal(\Sigma;\Bbf)$ spanned by equivalence classes of above graphs, with edges labeled by compatible simple objects.
 
We pick an arbitrary but fixed maximal tree $T$ in the above fat graph underlying the string-net. Such a tree has to lie inside a disk and we can collapse all its edges using the local relations. This yields a basis of $\KSNrm^\Ccal(\Sigma;\Bbf)$ in terms of (equivalence classes of) graphs of the form
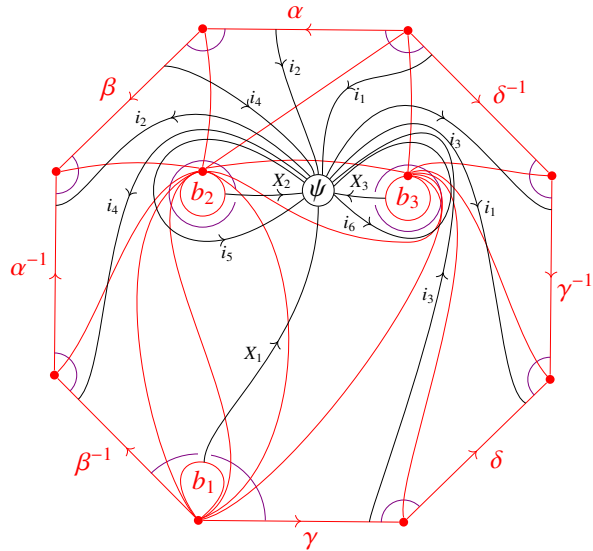
\begin{figure}[H]
    \makebox[\textwidth][c]{
    \resizebox{0.55\textwidth}{!}{
    \begin{tikzpicture}
        \def\d{12cm}
        \def\r{6cm}
        \node[circle, inner sep=0, minimum size=\d] (circ) at (0,0) {};
        \foreach \i in {0,45,90,135,180,225,270,315}{
            \node[circle, fill=red, inner sep=0, minimum size=2mm] (v\i) at (circ.\i+22.5) {};
        };
        \draw[mid, red] (v45) -- (v90)
            node[red, anchor=south, pos=0.55, yshift=1mm] {\Large $\alpha$}
            node[anchor=center, pos=0.65] (black60) {};
        \draw[mid, red] (v90) -- (v135)
            node[red, anchor=south east, pos=0.55] {\Large $\beta$}
            node[anchor=center, pos=0.25] (black120) {};
        \draw[mid, red] (v180) -- (v135)
            node[red, anchor=east, pos=0.55] {\Large $\alpha^{-1}$}
            node[anchor=center, pos=0.85] (black150) {};
        \draw[mid, red] (v225) -- (v180)
            node[red, anchor=north east, pos=0.55] {\Large $\beta^{-1}$}
            node[anchor=center, pos=0.85] (black210) {};
        \draw[mid, red] (v225) -- (v270)
            node[red, anchor=north, pos=0.55, yshift=-1mm] {\Large $\gamma$}
            node[anchor=center, pos=0.85] (black250) {};
        \draw[mid, red] (v270) -- (v315)
            node[red, anchor=north west, pos=0.55] {\Large $\delta$}
            node[anchor=center, pos=0.85] (black300) {};
        \draw[mid, red] (v0) -- (v315)
            node[red, anchor=west, pos=0.55] {\Large $\gamma^{-1}$}
            node[anchor=center, pos=0.15] (black330) {};
        \draw[mid, red] (v45) -- (v0)
            node[red, anchor=south west, pos=0.55] {\Large $\delta^{-1}$}
            node[anchor=center, pos=0.15] (black30) {};
        
        \draw let
            \p1 = (v45),
            \p2 = (v0),
            \p3 = (v90),
            \p4 = (v135)
        in
            node[circle, 
                 fill=red, 
                 inner sep=0, 
                 minimum size=2mm,
                 label={[text=red, font=\Large, yshift=-1cm]$b_3$}] 
                 (b3) at (\x1,\y2) {}
            node[circle, 
                 fill=red, 
                 inner sep=0, 
                 minimum size=2mm,
                 label={[text=red, font=\Large, yshift=-1cm]$b_2$}] (b2) at (\x3,\y4) {};
        
        \draw[red] (v225.center) .. controls ++(1.85,1.5) and ++(-1.5,2) .. (v225.center)
            node[anchor=center, pos=0.5, 
            label={[text=red, font=\Large, yshift=-1cm]$b_1$}] (bound1) {};
        \node[circle, draw, anchor=90, red, inner sep=0, minimum size=1cm] (bound2) at (b2.center) {};
        \node[circle, draw, anchor=90, red, inner sep=0, minimum size=1cm] (bound3) at (b3.center) {};

        \foreach \i in {0,45,90,135,180,270,315}{
            \centerarc[violet](v\i)(\i-90:\i-225:0.5cm);
        };
        \draw[violet] ($(v225.center) + (1.5,0)$) arc (0:135:1.5cm)
            node[anchor=center, fill=white, pos=0.625] {};
        \node[violet,
              draw,
              circle,
              inner sep=0,
              anchor=center,
              minimum size=1.5cm] (bc2) at (bound2.center) {};
        \node[violet,
              draw,
              circle,
              inner sep=0,
              anchor=center,
              minimum size=1.5cm] (bc3) at (bound3.center) {};
        \node[anchor=center, fill=white] at (bc2.357) {};
        \node[anchor=center, fill=white] at (bc3.180) {};
            
        \draw[red] (v90.300) to[bend left=10] (b2.90);
        \draw[red] (v45.210) to[bend left=0] (b2.45);
        \draw[red] (v135.0) to[bend left=10] (b2.135);
        \draw[red] (v180.45) .. controls ++(1,1) and ++(-1.5,-0.25) .. (b2.180);
        \draw[red] (v225.center) .. controls ++(-1,1) and ++(-2.5,-1) .. (b2.180);
        \draw[red] (v225.center) .. controls ++(2.45,1.25) and ++(-2.25,-1) .. (b2.180);
        \draw[red] (v225.30) .. controls ++(2.5,1) and ++(2.5,-1) .. (b2.0);
        \draw[red] (b2.45) to[bend left=10] (b3.135);
        \draw[red, rounded corners=0.1] (b2.0) .. controls ++(1,0) and ++(-2.5,0) .. ($(b3) + (0,-1.5)$)
            .. controls ++(1,0) and ++(1,0) .. (b3.center);
        \draw[red] (v45.270) to[bend left=5] (b3.90);
        \draw[red] (v0.center) .. controls ++(-1,0) and ++(0.5,0.5) .. (b3.60);
        \draw[red] (v315.center) .. controls ++(-1,1) and ++(1.5,0.5) .. (b3.45);
        \draw[red] (v270.90) .. controls ++(-0.25,1) and ++(2.5,0.25) .. (b3.30);
        \draw[red] (v225.15) ..controls ++(2,0.5) and ++(2.5,0) ..(b3.15);
    
        \node[fill=white,
              circle, 
              anchor=center,
              align=center,
              inner sep=0,
              text width=5mm,
              draw]
              (psi) at ($(circ.center) + (80:0.325*\r)$) {\Large $\psi$};
    
        \draw[mid] (bound1.center) .. controls ++(0,1) and ++(0,-3) .. (psi.270)
            node[anchor=east, pos=0.55] {$X_1$};
        \draw[mid={>}{0.75}] (bound2.0) to[bend right=5]
            node[anchor=south, pos=0.75] {$X_2$} (psi.190);
        \draw[mid={>}{0.75}] (bound3.180) to[bend right=-5]
            node[anchor=south, pos=0.5] {$X_3$} (psi.350);
        
        \draw[mid={>}{0.15}] (psi.345) .. controls ++(3.5,-3.5) and ++(3.5,3.5) .. (psi.15)
            node[anchor=north, pos=0.035] {$i_6$};
        \draw[mid={<}{0.25}] (psi.210) .. controls ++(-4.5,-3.5) and ++(-4.5,3.5) .. (psi.150)
            node[anchor=north, pos=0.15] {$i_5$};
            
        \draw[mid={>}{0.25}] (black60.center) .. controls ++(0,-1) and ++(0,0.5) .. (psi.90)
            node[anchor=west, pos=0.25] {$i_2$};
        \draw[mid={>}{0.5}] (black120.center) .. controls ++(1,0) and ++(0,0.5) .. (psi.105)
            node[anchor=south, pos=0.5] {$i_4$};
        \draw[mid={<}{0.5}] (black150.center) .. controls ++(1,0) and ++(-3.5,3.5) .. (psi.120)
            node[anchor=south, pos=0.5] {$i_2$};
        \draw[mid={<}{0.5}] (black210.center) .. controls ++(1,1) and ++(-5,4) .. (psi.135)
            node[anchor=east, pos=0.5] {$i_4$};
        \draw[mid={>}{0.5}] (black30.center) .. controls ++(-1,-1) and ++(0,3) .. (psi.75)
            node[anchor=north, pos=0.5] {$i_1$};
        \draw[mid={<}{0.5}] (black330.center) .. controls ++(-1,0) and ++(2,4) .. (psi.60)
            node[anchor=north, pos=0.5] {$i_3$};
        \draw[mid={<}{0.5}] (black300.center) .. controls ++(-1,0.5) and ++(4,4.5) .. (psi.45)
            node[anchor=west, pos=0.5] {$i_1$};
        \draw[mid={>}{0.5}] (black250.center) .. controls ++(0,0.5) and ++(5.6,4.75) .. (psi.30)
            node[anchor=north east, pos=0.475] {$i_3$};
    \end{tikzpicture}
    }
    }
\end{figure}
where the $i_n$'s are simple objects. We will not show the $G$-triangulation anymore from now on. We can isotope the cloaking circle around the first boundary component to run parallel to the fat graph. The resulting string-net, drawn on $\Sigma$, then looks as follows
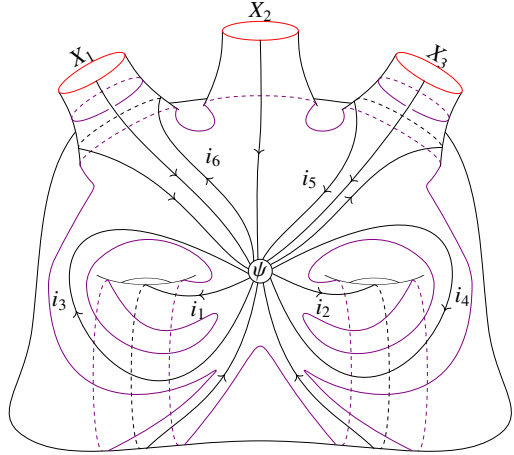
\begin{figure}[H]
    \centering
    \makebox[\textwidth][c]{
    \resizebox{0.5\width}{!}{
    \begin{tikzpicture}
    
        \begin{scope}[on background layer]
            \draw  plot[smooth cycle, thick, tension=0.75] coordinates {(-6,-3) (0,-3) (6,-3) (6,0) (5,5) (2,6) (-2,6) (-5,5) (-6,0)};
        \end{scope}
        
        
        \foreach \i/\n/\x/\y in 
            {
            1/LH/-3/1.5, 
            2/RH/3/1.5}
            {
            \node[inner sep=-0.75mm, anchor=center] (\n) at (\x,\y) {
                \resizebox{0.75\width}{0.25\height}{
                \begin{tikzpicture}
                    
                        \FPset\rbig{2}
                        \FPset\rsmol{1}
                        
                        \FPset\initang{-30}
                        \FPset\finang{-150}
                        \FPeval\rang{round(abs((\initang - \finang) / 2), 0)}
                        
                        \FPeval\dsmol{\rsmol * (1-cos(\rang * \FPpi / 180))}
                        \FPeval\dbig{\rbig - root(2, (\rbig^2 - \rsmol^2 * sin(\rang * \FPpi / 180 )^2))}
                        \FPeval\hsep{(\rsmol - \dsmol) + (\rbig - \dbig)}
                        
                        \node (c) at (0,0) {};
                        \centerarc[thick](c)(\initang:\finang:\rbig cm);
                        \centerarc[thick, transform canvas={yshift=-\hsep cm}](c)(-\initang:-\finang:\rsmol cm);
                    
                    \end{tikzpicture}
                    }
            };
            
            \node[anchor=center] (\n_m) at (\n.south) {};
            \node[anchor=center] (\n_r) at ($(\n.south) + (1,0.125)$) {};
            \node[anchor=center] (\n_l) at ($(\n.south) + (-1,0.125)$) {};
        };
        
        \foreach \i\n/\x/\y/\a/\sx/\sy in 
            {
            1/BL/-4/6/30/0.5/0.5, 
            2/BM/0/7/0/0.5/0.5, 
            3/BR/4/6/-30/0.5/0.5}
            {
            \node[anchor=center, inner sep=0, rotate=\a] (\n) at (\x,\y) {
                \begin{tikzpicture}[xscale=\sx,yscale=\sy]
                    
                    
                    \fill[white] (0,0) arc (0:180:2 and 0.5)
                    .. controls ++(0,-1) and ++(1,2) .. (-5,-4)
                    -- (1,-4)
                    .. controls ++(-1,2) and ++(0,-1) .. (0,0);
                    
                    \draw[red] (0,0) arc (0:180:2 and 0.5) coordinate (p);
                    \draw (p) .. controls ++(0,-1) and ++(1,2) .. (-5,-4)
                        node[pos=0.5, anchor=center] (upl) {}
                        node[pos=0.85, anchor=center] (\n_qupl) {}
                        node[pos=0.99, anchor=center] (\n_qupl_purp) {};
                    \draw (1,-4) .. controls ++(-1,2) and ++(0,-1) .. (0,0)
                        node[pos=0.5, anchor=center] (upr) {}
                        node[pos=0.15, anchor=center] (\n_qupr) {}
                        node[pos=0.01, anchor=center] (\n_qupr_purp) {};
                    \draw[red] (0,0) arc (0:-180:2 and 0.5);
                    
                    \ifthenelse{\i=2}
                    {
                        \draw[dashed, violet] (\n_qupl_purp.center) .. controls ++(0,4/3*0.5) and ++(0,4/3*0.5) .. (\n_qupr_purp.center);
                    }{
                        \draw[violet] (upl.center) .. controls ++(0,-4/3*0.5) and ++(0,-4/3*0.5) .. (upr.center)
                            node[anchor=center, inner sep=0, minimum size=2mm, fill=white, pos=0.5] {};
                        \draw[violet, dashed] (upl.center) .. controls ++(0,4/3*0.5) and ++(0,4/3*0.5) .. (upr.center);
                        \draw[dashed] (\n_qupl.center) .. controls ++(0,4/3*0.5) and ++(0,4/3*0.5) .. (\n_qupr.center);
                        \draw[violet, dashed] (\n_qupl_purp.center) .. controls ++(0,4/3*0.5) and ++(0,4/3*0.5) .. (\n_qupr_purp.center);
                    }
                    
                \end{tikzpicture}
            };
            \node[anchor=center] (\n_out) at ($(\n.north) + (\a-90:0.5)$) {};
            \node[anchor=center, rotate=\a] at ($(\n.north) + (\a+90:0.25)$) {\Large $X_{\i}$};
        };
        
        \node[anchor=center] (LH_m_bot) at ($(LH.south) + (0,-4.35)$) {};
        \node[anchor=center] (LH_r_bot) at ($(LH.south) + (1,-4.25)$) {};
        \node[anchor=center] (LH_l_bot) at ($(LH.south) + (-1,-4.4)$) {};
        \node[anchor=center] (RH_m_bot) at ($(RH.south) + (0,-4.35)$) {};
        \node[anchor=center] (RH_r_bot) at ($(RH.south) + (1,-4.4)$) {};
        \node[anchor=center] (RH_l_bot) at ($(RH.south) + (-1,-4.25)$) {};
        
        \node[anchor=center] (BL_purp_l) at ($(BL.south west) + (0,0.05)$){};
        \node[anchor=center] (BL_purp_r) at ($(BL.south east) + (-0.075,0.05)$){};
        \node[anchor=center] (BL_blac_l) at ($(BL.south west) + (-0.05,0.5)$){};
        \node[anchor=center] (BL_blac_r) at ($(BL.south east) + (-0.5,0.35)$){};

        \node[anchor=center] (BR_purp_l) at ($(BR.south west) + (0.05,0)$){};
        \node[anchor=center] (BR_purp_r) at ($(BR.south east) + (0,0.05)$){};
        \node[anchor=center] (BR_blac_l) at ($(BR.south west) + (0.35,0.25)$){};
        \node[anchor=center] (BR_blac_r) at ($(BR.south east) + (0.05,0.525)$){};
        
        \node[anchor=center] (BM_purp_l) at ($(BM.south west) + (0.05,0.05)$){};
        \node[anchor=center] (BM_purp_r) at ($(BM.south east) + (-0.05,0.1)$){};

    \node[fill=white,
          circle, 
          anchor=center, 
          inner sep=0, 
          minimum size=6mm, 
          draw] 
          (main) at (0,1.5) {\Large $\psi$};
    
        \foreach \n/\dx/\dy in 
        {
            LH/-0.5/0, 
            RH/0.5/0}
        {
            \draw[dashed, violet] (\n_l.center) .. controls ++(\dx,\dy) and ++(\dx,-\dy) .. (\n_l_bot);
            \draw[dashed] (\n_m.center) .. controls ++(\dx,\dy) and ++(\dx,-\dy) .. (\n_m_bot);
            \draw[dashed, violet] (\n_r.center) .. controls ++(\dx,\dy) and ++(\dx,-\dy) .. (\n_r_bot);
        };
        
        \draw[violet, rounded corners=10] 
            (LH_l.center) .. controls ++(1,-2) and ++(-1,-0.5) .. (-1,0.5)
                          .. controls ++(-1,-2) and ++(0,-2) .. ($(LH.west) + (-0.25,0)$)
                          .. controls ++(2,2) and ++(2,-0.5) .. (LH_r.center);
        
        \draw[violet, rounded corners=10]
            (LH_l_bot.center) .. controls ++(2,1) and ++(-0.5,-0.5) .. (-1,-1)
                              .. controls ++(-1.5,-1) and ++(0,-4) .. ($(LH.west) + (-1.25,0)$)
                              .. controls ++(1.5,3) and ++(0.5,-1) .. (BL_purp_l.center);

        \draw[violet, rounded corners=10]
            (BL_purp_r.center) .. controls (-2.5,5) and (-0.5,5) .. (BM_purp_l.center);
        
        \draw[violet, rounded corners=10]
            (BM_purp_r.center) .. controls (0.5,5) and (2.5,5) .. (BR_purp_l.center);

        \draw[violet, rounded corners=10] 
            (RH_r.center) .. controls ++(-1,-2) and ++(1,-0.5) .. (1,0.5)
                          .. controls ++(1,-2) and ++(0,-2) .. ($(RH.east) + (0.25,0)$)
                          .. controls ++(-2,2) and ++(-2,-0.5) .. (RH_l.center);
        
        \draw[violet, rounded corners=10]
            (RH_r_bot.center) .. controls ++(-2,1) and ++(0.5,-0.5) .. (1,-1)
                              .. controls ++(1.5,-1) and ++(0,-4) .. ($(RH.east) + (1.25,0)$)
                              .. controls ++(-1.5,3) and ++(-0.5,-1) .. (BR_purp_r.center);
                              
        \draw[violet, rounded corners=10]
            (LH_r_bot.center) .. controls ++(3,3.5) and ++(-3,3.5) .. (RH_l_bot.center);
        
        \foreach \ang/\ax/\ay/\bx/\by/\endpoint/\arr/\arrpos/\lab/\labanchor/\labpos in 
        {
        22.5/9/5/3/-8/main.326.25/>/0.5/i_4/west/0.45,
        45/2/1/-2/0/BR_blac_r.center/>/0.5//center/,
        67.5/1.5/1/-0.5/-1/BR_out.center/</.5//center/,
        90/0.5/1/0.5/-2/BR_blac_l.center/</0.5/i_5/south east/0.5,
        112.5/-0.1/4/0.1/-4/BM_out.center/</0.5//center/,
        135/0/2/-0.5/-1.5/BL_blac_r.center/>/0.5/i_6/south west/0.5,
        157.5/-1/2/1/-2/BL_out.center/</0.5//center/,
        180/-1/0.5/2.5/-0.5/BL_blac_l.center/</0.5//center/,
        202.5/-9/5/-3/-8/main.236.75/</0.5/i_3/east/0.45,
        225/-1/-0.5/1/-0.5/LH_m.center/>/0.5/i_1/north/0.5,
        270/-0.5/-3/1/1/LH_m_bot.center/</0.5//center/,
        295.5/0.5/-3/-1/1/RH_m_bot.center/</0.5//center/,
        0/1/-0.5/-1/-0.5/RH_m.center/>/0.5/i_2/north/0.5}
        {
            \draw[mid={\arr}{\arrpos}] (main.\ang-11.25) .. controls ++(\ax,\ay) and ++(\bx,\by) .. (\endpoint)
                node[anchor=\labanchor, pos=\labpos] {\Large $\lab$};
        }
        
    \end{tikzpicture}
    }}
\end{figure}

We can use the completeness relation several times, to bring the fat graph in the form 
\begin{figure}[H]
    \centering
    \makebox[\textwidth][c]{
    \resizebox{0.5\width}{!}{
    \begin{tikzpicture}
        \node[inner sep=0] (figure) at (0,0) {
            \begin{tikzpicture}
    
        \begin{scope}[on background layer]
            \draw  plot[smooth cycle, thick, tension=0.75] coordinates {(-6,-3) (0,-3) (6,-3) (6,0) (5,5) (2,6) (-2,6) (-5,5) (-6,0)};
        \end{scope}
        
        
        \foreach \i/\n/\x/\y in 
            {
            1/LH/-3/1.5, 
            2/RH/3/1.5}
            {
            \node[inner sep=-0.75mm, anchor=center] (\n) at (\x,\y) {
                \resizebox{0.75\width}{0.25\height}{
                \begin{tikzpicture}
                    
                        \FPset\rbig{2}
                        \FPset\rsmol{1}
                        
                        \FPset\initang{-30}
                        \FPset\finang{-150}
                        \FPeval\rang{round(abs((\initang - \finang) / 2), 0)}
                        
                        \FPeval\dsmol{\rsmol * (1-cos(\rang * \FPpi / 180))}
                        \FPeval\dbig{\rbig - root(2, (\rbig^2 - \rsmol^2 * sin(\rang * \FPpi / 180 )^2))}
                        \FPeval\hsep{(\rsmol - \dsmol) + (\rbig - \dbig)}
                        
                        \node (c) at (0,0) {};
                        \centerarc[thick](c)(\initang:\finang:\rbig cm);
                        \centerarc[thick, transform canvas={yshift=-\hsep cm}](c)(-\initang:-\finang:\rsmol cm);
                    
                    \end{tikzpicture}
                    }
            };
            
            \node[anchor=center] (\n_m) at (\n.south) {};
            \node[anchor=center] (\n_r) at ($(\n.south) + (1,0.125)$) {};
            \node[anchor=center] (\n_l) at ($(\n.south) + (-1,0.125)$) {};
        };
        
        \foreach \i\n/\x/\y/\a/\sx/\sy in 
            {
            1/BL/-4/6/30/0.5/0.5, 
            2/BM/0/7/0/0.5/0.5, 
            3/BR/4/6/-30/0.5/0.5}
            {
            \node[anchor=center, inner sep=0, rotate=\a] (\n) at (\x,\y) {
                \begin{tikzpicture}[xscale=\sx,yscale=\sy]
                    
                    
                    \fill[white] (0,0) arc (0:180:2 and 0.5)
                    .. controls ++(0,-1) and ++(1,2) .. (-5,-4)
                    -- (1,-4)
                    .. controls ++(-1,2) and ++(0,-1) .. (0,0);
                    
                    \draw[red] (0,0) arc (0:180:2 and 0.5) coordinate (p);
                    \draw (p) .. controls ++(0,-1) and ++(1,2) .. (-5,-4)
                        node[pos=0.5, anchor=center] (upl) {}
                        node[pos=0.85, anchor=center] (\n_qupl) {}
                        node[pos=0.99, anchor=center] (\n_qupl_purp) {};
                    \draw (1,-4) .. controls ++(-1,2) and ++(0,-1) .. (0,0)
                        node[pos=0.5, anchor=center] (upr) {}
                        node[pos=0.15, anchor=center] (\n_qupr) {}
                        node[pos=0.01, anchor=center] (\n_qupr_purp) {};
                    \draw[red] (0,0) arc (0:-180:2 and 0.5);
                    
                    \ifthenelse{\i=2}
                    {
                    }{
                        \draw[violet] (upl.center) .. controls ++(0,-4/3*0.5) and ++(0,-4/3*0.5) .. (upr.center)
                            node[anchor=center, inner sep=0, minimum size=2mm, fill=white, pos=0.5] {};
                        \draw[violet, dashed] (upl.center) .. controls ++(0,4/3*0.5) and ++(0,4/3*0.5) .. (upr.center);
                        \draw[dashed] (\n_qupl.center) .. controls ++(0,4/3*0.5) and ++(0,4/3*0.5) .. (\n_qupr.center);
                    }
                    
                \end{tikzpicture}
            };
            \node[anchor=center] (\n_out) at ($(\n.north) + (\a-90:0.5)$) {};
            \node[anchor=center, rotate=\a] at ($(\n.north) + (\a+90:0.25)$) {\Large $X_{\i}$};
        };
        
        \node[anchor=center] (LH_m_bot) at ($(LH.south) + (0,-4.35)$) {};
        \node[anchor=center] (LH_r_bot) at ($(LH.south) + (1,-4.25)$) {};
        \node[anchor=center] (LH_l_bot) at ($(LH.south) + (-1,-4.4)$) {};
        \node[anchor=center] (RH_m_bot) at ($(RH.south) + (0,-4.35)$) {};
        \node[anchor=center] (RH_r_bot) at ($(RH.south) + (1,-4.4)$) {};
        \node[anchor=center] (RH_l_bot) at ($(RH.south) + (-1,-4.25)$) {};
        
        \node[anchor=center] (BL_purp_l) at ($(BL.south west) + (0,0.05)$){};
        \node[anchor=center] (BL_purp_r) at ($(BL.south east) + (-0.075,0.05)$){};
        \node[anchor=center] (BL_blac_l) at ($(BL.south west) + (-0.05,0.5)$){};
        \node[anchor=center] (BL_blac_r) at ($(BL.south east) + (-0.5,0.35)$){};

        \node[anchor=center] (BR_purp_l) at ($(BR.south west) + (0.05,0)$){};
        \node[anchor=center] (BR_purp_r) at ($(BR.south east) + (0,0.05)$){};
        \node[anchor=center] (BR_blac_l) at ($(BR.south west) + (0.35,0.25)$){};
        \node[anchor=center] (BR_blac_r) at ($(BR.south east) + (0.05,0.525)$){};
        
        \node[anchor=center] (BM_purp_l) at ($(BM.south west) + (0.05,0.05)$){};
        \node[anchor=center] (BM_purp_r) at ($(BM.south east) + (-0.05,0.1)$){};

        \node[fill=white,
              circle, 
              anchor=center, 
              inner sep=0, 
              minimum size=6mm, 
              draw] 
              (main) at (0,1.5) {\Large $\psi$};
          
        \node[anchor=center, circle, draw, inner sep=0, minimum size=4mm]
            (a1) at ($(main.center) + (20:3)$) {$\alpha$};
        \node[anchor=center, circle, draw, inner sep=0, minimum size=4mm]
            (a2) at ($(main.center) + (-45:3)$) {$\alpha$};
        \draw[mid] (main.20) .. controls ++(2,0) and ++(-1,0) .. (a1.180)
            node[anchor=north, pos=0.75] {$i_4$};
        \draw[mid] (a1.0) .. controls ++(2,0) and ++(4,0) .. (a2.0)
            node[anchor=west, pos=0.5] {$m_5$};
        \draw[mid] (a2.180) .. controls ++(-1,0) and ++(0.5,0) .. (main.315);
        
        \node[anchor=center, circle, draw, inner sep=0, minimum size=4mm]
            (b1) at ($(main.center) + (-15:2)$) {$\beta$};
        \node[anchor=center, circle, draw, inner sep=0, minimum size=4mm]
            (b2) at ($(main.center) + (-75:2)$) {$\beta$};
        \draw[mid] (main.345) .. controls ++(1,0) and ++(-1,0) .. (b1.180)
            node[anchor=south, pos=0.75] {$i_2$};
        \draw[mid] (b1.340) to[bend right] (RH_m.center);
        \draw[mid={<}{0.5}] (main.285) .. controls ++(0,-0.5) and ++(0,0.5) .. (b2.90);
        \draw[mid] (RH_m_bot.center) .. controls ++(0,1) and ++(0,-1) .. (b2.270)
            node[anchor=north east, pos=0.5] {$m_4$};
        
        \node[anchor=center, circle, draw, inner sep=0, minimum size=4mm]
            (d1) at ($(main.center) + (35:4)$) {$\delta$};
        \node[anchor=center, circle, draw, inner sep=0, minimum size=4mm]
            (d2) at ($(main.center) + (55:4)$) {$\delta$};
        \draw[mid] (main.35) .. controls ++(1,1) and ++(-1,-1) .. (d1.225);
        \draw[mid={<}{0.5}] (main.55) .. controls ++(0,1) and ++(-1,-1) .. (d2.225)
            node[anchor=east, pos=0.5] {$i_5$};
        \draw[mid] (d1.60) .. controls ++(0.5,0.5) and ++(-0.75,-0.25) .. (BR_blac_r.center) 
            node[anchor=north west, pos=0.5] {$m_6$};
        \draw[mid={<}{0.5}] (d2.90) .. controls ++(0,0.5) and ++(0.25,-0.75) .. (BR_blac_l.center);
        
        \node[anchor=center, circle, draw, inner sep=0, minimum size=4mm]
            (c1) at ($(main.center) + (160:3)$) {$\gamma$};
        \node[anchor=center, circle, draw, inner sep=0, minimum size=4mm]
            (c2) at ($(main.center) + (225:3)$) {$\gamma$};
        \draw[mid={<}{0.5}] (main.160) .. controls ++(-2,0) and ++(1,0) .. (c1.0)
            node[anchor=north, pos=0.75] {$i_3$};
        \draw[mid={<}{0.5}] (c1.180) .. controls ++(-2,0) and ++(-4,0) .. (c2.180)
            node[anchor=east, pos=0.5] {$m_2$};
        \draw[mid={<}{0.5}] (c2.0) .. controls ++(1,0) and ++(-0.5,0) .. (main.225);
        
        \node[anchor=center, circle, draw, inner sep=0, minimum size=4mm]
            (e1) at ($(main.center) + (195:2)$) {$\epsilon$};
        \node[anchor=center, circle, draw, inner sep=0, minimum size=4mm]
            (e2) at ($(main.center) + (255:2)$) {$\epsilon$};
        \draw[mid] (main.195) .. controls ++(-1,0) and ++(1,0) .. (e1.0)
            node[anchor=south, pos=0.75] {$i_1$};
        \draw[mid] (e1.200) to[bend left] (LH_m.center);
        \draw[mid={<}{0.5}] (main.255) .. controls ++(0,-0.5) and ++(0,0.5) .. (e2.90);
        \draw[mid] (LH_m_bot.center) .. controls ++(0,1) and ++(0,-1) .. (e2.270)
            node[anchor=north west, pos=0.5] {$m_3$};
        
        \node[anchor=center, circle, draw, inner sep=0, minimum size=4mm]
            (f1) at ($(main.center) + (145:4)$) {$\zeta$};
        \node[anchor=center, circle, draw, inner sep=0, minimum size=4mm]
            (f2) at ($(main.center) + (125:4)$) {$\zeta$};
        \draw[mid] (main.145) .. controls ++(-1,1) and ++(1,-1) .. (f1.315);
        \draw[mid={<}{0.5}] (main.125) .. controls ++(0,1) and ++(1,-1) .. (f2.315)
            node[anchor=west, pos=0.5] {$i_6$};
        \draw[mid] (f1.150) .. controls ++(-0.5,0.5) and ++(0.75,-0.25) .. (BL_blac_l.center) 
            node[anchor=north east, pos=0.5] {$m_1$};
        \draw[mid={<}{0.5}] (f2.90) .. controls ++(0,0.5) and ++(-0.25,-0.75) .. (BL_blac_r.center);
        
        \draw[violet] (b2.0) .. controls ++(0.5,0) and ++(0,-1) .. (a2.270);
        \draw[violet] (a2.90) .. controls ++(0.1,0.75) and ++(-0.1,-0.75) .. (b1.270);
        \draw[violet] (b1.135) .. controls ++(-0.75,0.75) and ++(0,-0.75) .. (a1.270);
        \draw[violet] (a1.90) .. controls ++(0,0.5) and ++(0,-0.5) .. (d1.270);
        \draw[violet] (e2.180) .. controls ++(-0.5,0) and ++(0,-1) .. (c2.270);
        \draw[violet] (c2.90) .. controls ++(-0.1,0.75) and ++(0.1,-0.75) .. (e1.270);
        \draw[violet] (e1.45) .. controls ++(0.75,0.75) and ++(0,-0.75) .. (c1.270);
        \draw[violet] (c1.90) .. controls ++(0,0.5) and ++(0,-0.5) .. (f1.270);
        \draw[violet] (f2.0) .. controls ++(2,-0.75) and ++(-2,-0.75) .. (d2.180)
            node[anchor=center, fill=white, midway, minimum size=2mm] {}
            node[anchor=north, yshift=-1mm, violet, pos=0.25] {$i$};
        
        \draw[mid={<}{0.5}] (main.45) .. controls ++(1,1.5) and ++(-1,-2) .. (BR_out.center);
        \draw[mid={<}{0.5}] (main.135) .. controls ++(-1,1.5) and ++(1,-2) .. (BL_out.center);
        \draw[mid={<}{0.5}] (main.90) .. controls ++(0,2) and ++(0,-2) .. (BM_out.center);
        
        \foreach \n/\dx/\dy in 
        {
            LH/-0.5/0, 
            RH/0.5/0}
        {
            \draw[dashed] (\n_m.center) .. controls ++(\dx,\dy) and ++(\dx,-\dy) .. (\n_m_bot);
        };

        \foreach \ang/\ax/\ay/\bx/\by/\endpoint/\arr/\arrpos/\lab/\labanchor/\labpos in 
        {
        }
        {
        }
        
    \end{tikzpicture}
        };
        \node[anchor=south west] (sum) at (figure.north west){\LARGE $\displaystyle{\sum_{\substack{i \in I_{e}, \\ m_1,\dots,m_6}} \frac{d_i}{D}d_{m_1}d_{m_2}d_{m_3}d_{m_4}d_{m_5}d_{m_6}}$};
    \end{tikzpicture}
    }}
\end{figure}

Note that $\zeta$ fixes the $G$-color of a $G$-triangulation with no internal marked points uniquely. By collapsing a maximal tree, we get a fixed $G$-color of the edges in the one vertex graph. Since the cyclicity condition on vertices is preserved by the local moves, this coloring satisfies
\eq{
h b_3h^{-1}b_1gb_2g^{-1}\gamma\delta\gamma^{-1}\delta^{-1}\alpha\beta\alpha^{-1}\beta^{-1}=e\quad. 
}
where $h$ is the $G$-color of $i_8$ and $g$ the one of $i_5$.
By contracting the maximal tree $T$ of the fat graph, the non-contracted edges constitute generators for $\pi_1(\Sigma)$ based at the single vertex. Choosing another maximal tree, we get a different set of free generators. However, since the $G$-color of the fat graph was uniquely determined, the $G$-color for the generators is also uniquely determined in both cases. It solely depends on the map $\zeta$.

For the tree chosen here, we define an object
\eq{
c_{X_3,X_1,X_2,\mathbf{m}}\coloneqq \phi_h(X_3) \otimes X_1\otimes \phi_g(X_2)\otimes m_5\otimes m_4\otimes m_5^\ast\otimes m_4^\ast \otimes m_3\otimes m_2\otimes m_3^\ast\otimes m_2^\ast\; ,
}
where $m_5\in I_\beta$, $m_4\in I_\alpha$, $m_3\in I_\delta$, $m_2\in I_\gamma$ and $X_i\in \Zsf_G(\Ccal)_{b_i}$ are the boundary values. The object
\eq{\label{c defi}
c\coloneqq \bigoplus_{\mathbf{m}}c_{{X_3,X_1,X_2},\mathbf{m}}
}
has a natural half-braiding

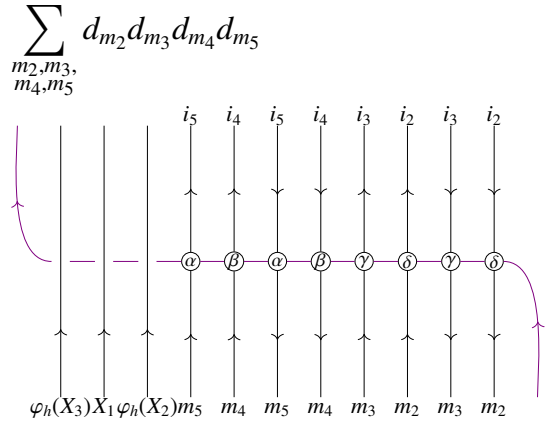
\begin{figure}[H]
    \centering
    \makebox[\textwidth][c]{
    \resizebox{0.6\width}{!}{
    \begin{tikzpicture}
        \def\x{0.95}
        \def\h{3}
        \node[inner sep=0] (lines) at (0,0) {
            \begin{tikzpicture}
                
                \node[anchor=center, circle, inner sep=0, minimum size=4mm] 
                    (n0) at (0,0) {};
                \node[anchor=center, circle, inner sep=0, minimum size=4mm] 
                    (n1) at (\x,0) {};
                \node[anchor=center, circle, inner sep=0, minimum size=4mm] 
                    (n2) at (2*\x,0) {};
                \node[anchor=center, circle, inner sep=0, minimum size=4mm, draw] 
                    (n3) at (3*\x,0) {$\alpha$};
                \node[anchor=center, circle, inner sep=0, minimum size=4mm, draw] 
                    (n4) at (4*\x,0) {$\beta$};
                \node[anchor=center, circle, inner sep=0, minimum size=4mm, draw] 
                    (n5) at (5*\x,0) {$\alpha$};
                \node[anchor=center, circle, inner sep=0, minimum size=4mm, draw] 
                    (n6) at (6*\x,0) {$\beta$};
                \node[anchor=center, circle, inner sep=0, minimum size=4mm, draw] 
                    (n7) at (7*\x,0) {$\gamma$};
                \node[anchor=center, circle, inner sep=0, minimum size=4mm, draw] 
                    (n8) at (8*\x,0) {$\delta$};
                \node[anchor=center, circle, inner sep=0, minimum size=4mm, draw] 
                    (n9) at (9*\x,0) {$\gamma$};
                \node[anchor=center, circle, inner sep=0, minimum size=4mm, draw] 
                    (n10) at (10*\x,0) {$\delta$};
                    
                \foreach \i/\d/\num in {
                    3/>/5,
                    4/>/4,
                    5/</5,
                    6/</4,
                    7/>/3,
                    8/>/2,
                    9/</3,
                    10/</2}
                {
                    \draw[mid={\d}{0.5}] (\i*\x,-\h) node[anchor=north] {\large $\phantom{(}m_{\num}\phantom{)}$} -- (n\i.south);
                    \draw[mid={\d}{0.5}] (n\i.north) -- (\i*\x,\h) node[anchor=south] {\large $\phantom{(}i_{\num}\phantom{)}$};
                };
                
                \draw[violet, mid] (11*\x,-\h) .. controls ++(0,0.5) and ++(1,0) .. (n10.east);
                \draw[violet, mid] (n0.west) .. controls ++(-1,0) and ++(0,-0.5) .. (-\x,\h);
                \foreach \a/\b in {
                    n0.west/n3.west,
                    n3.east/n4.west,
                    n4.east/n5.west,
                    n5.east/n6.west,
                    n6.east/n7.west,
                    n7.east/n8.west,
                    n8.east/n9.west,
                    n9.east/n10.west}
                {
                    \draw[violet] (\a) -- (\b);
                };

                \node[circle, inner sep=0, anchor=center, minimum size=4mm, fill=white] at (n0.center) {};
                \node[circle, inner sep=0, anchor=center, minimum size=4mm, fill=white] at (n1.center) {};
                \node[circle, inner sep=0, anchor=center, minimum size=4mm, fill=white] at (n2.center) {};

                \foreach \i/\n in {
                    0/\varphi_h(X_3), 
                    1/\phantom{(}X_1\phantom{)}, 
                    2/\varphi_h(X_2)}
                {
                    \draw[mid={>}{0.25}] (\i*\x,-\h) node[anchor=north] {\large $\n$} -- (\i*\x,\h);
                };

            \end{tikzpicture}
           
        };
        \node[anchor=south west] (sum) at (lines.north west) {\LARGE $\displaystyle{\sum_{\substack{m_2,m_3,\\ m_4,m_5}}d_{m_2}d_{m_3}d_{m_4}d_{m_5}}$};
    \end{tikzpicture}
     }
            }
\end{figure}
and thus can be seen as an object $(c,\sigma)\in \Zsf_G(\Ccal)$ (\cite[section~10.4]{turaev20203}).

The computation is now similar to the one for the cylinder. We define  a linear map
\begin{align*}
\hom_{\Zsf_G(\Ccal)}(\onebb,c)\longrightarrow \KSNrm^\Ccal(\Sigma;X_1,X_2,X_3)
\end{align*}
which acts on $\Upsilon\in \hom_{\Zsf_G(\Ccal)}(\onebb,c)$ by 
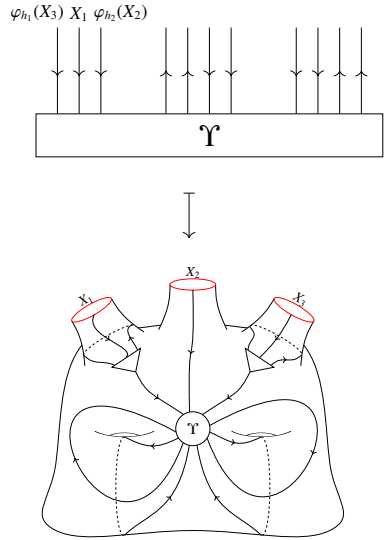
\begin{figure}[H]
    \centering
    \makebox[\textwidth][c]{
    \resizebox{0.6\width}{!}{
    \def\x{8}
    \def\y{1}
    \def\d{2}
    \def\h{1}
    \FPeval\s{1 / 16}
    \FPeval\t{2 * \s * \x + 0.5}
    \begin{tikzpicture}
        \node (strings) at (0,0){
            \begin{tikzpicture}
                \node (morphism) at (0,0) {
                    \centering
                    \makebox[\textwidth][c]{
                    \resizebox{\StringScale\width}{!}{

                    \begin{tikzpicture}
                        
                        \node[draw, anchor=center, inner sep=0, minimum height=\y cm, minimum width=\x cm] 
                            (psi) at (0,0) {\LARGE $\Upsilon$};
                        
                        \foreach \i/\dir in 
                        {
                        1/<,
                        2/<,
                        3/<,
                        6/>,
                        7/>,
                        8/<,
                        9/<,
                        12/<,
                        13/<,
                        14/>,
                        15/>}
                        {
                            \FPeval\loc{\i * \s}
                            \ifthenelse{\i=1}
                            {
                                \draw[mid={\dir}{0.5}] ($(psi.north west)!\loc!(psi.north east)$) -- ++(0,2*\h) node[anchor=south east, xshift=3mm] {$\varphi_{h_1}(X_3)$};
                            }
                            {
                                \ifthenelse{\i=2}
                                {
                                    \draw[mid={\dir}{0.5}] ($(psi.north west)!\loc!(psi.north east)$) -- ++(0,2*\h) node[anchor=south] {$X_1$};
                                }
                                {
                                    \ifthenelse{\i=3}
                                    {
                                        \draw[mid={\dir}{0.5}] ($(psi.north west)!\loc!(psi.north east)$) -- ++(0,2*\h) node[anchor=south west, xshift=-3mm] {$\varphi_{h_2}(X_2)$};
                                    }
                                    {
                                        \draw[mid={\dir}{0.5}] ($(psi.north west)!\loc!(psi.north east)$) -- ++(0,2*\h);
                                    };
                                };
                            };
                        };

                    \end{tikzpicture}}}
                };
            \end{tikzpicture}
            
        };
        \node[anchor=north] (pic) at ($(strings.south) + (0,-\d)$){
            
            \begin{tikzpicture}
                \node[inner sep=0] (figure) at (0,0) {
                    \centering
                    \makebox[\textwidth][c]{
                    \resizebox{\SurfScale\width}{!}{

                    \begin{tikzpicture}
            
                        \begin{scope}[on background layer]
                            \draw  plot[smooth cycle, thick, tension=0.75] coordinates {(-6,-3) (0,-3) (6,-3) (6,0) (5,5) (2,6) (-2,6) (-5,5) (-6,0)};
                        \end{scope}
                        
                        
                        \foreach \i/\n/\x/\y in 
                            {
                            1/LH/-3/1.5, 
                            2/RH/3/1.5}
                            {
                            \node[inner sep=-0.75mm, anchor=center] (\n) at (\x,\y) {
                                \resizebox{0.75\width}{0.25\height}{
                                \begin{tikzpicture}[]
                                    
                                        \FPset\rbig{2}
                                        \FPset\rsmol{1}
                                        
                                        \FPset\initang{-30}
                                        \FPset\finang{-150}
                                        \FPeval\rang{round(abs((\initang - \finang) / 2), 0)}
                                        
                                        \FPeval\dsmol{\rsmol * (1-cos(\rang * \FPpi / 180))}
                                        \FPeval\dbig{\rbig - root(2, (\rbig^2 - \rsmol^2 * sin(\rang * \FPpi / 180 )^2))}
                                        \FPeval\hsep{(\rsmol - \dsmol) + (\rbig - \dbig)}
                                        
                                        \node (c) at (0,0) {};
                                        \centerarc[thick](c)(\initang:\finang:\rbig cm);
                                        \centerarc[thick, transform canvas={yshift=-\hsep cm}](c)(-\initang:-\finang:\rsmol cm);
                                    
                                    \end{tikzpicture}
                                    }
                            };
                            
                            \node[anchor=center] (\n_m) at (\n.south) {};
                            \node[anchor=center] (\n_r) at ($(\n.south) + (1,0.125)$) {};
                            \node[anchor=center] (\n_l) at ($(\n.south) + (-1,0.125)$) {};
                        };
                        
                        \foreach \i\n/\x/\y/\a/\sx/\sy in 
                            {
                            1/BL/-4/6/30/0.5/0.5, 
                            2/BM/0/7/0/0.5/0.5, 
                            3/BR/4/6/-30/0.5/0.5}
                            {
                            \node[anchor=center, inner sep=0, rotate=\a] (\n) at (\x,\y) {
                                \begin{tikzpicture}[xscale=\sx,yscale=\sy]
                                    
                                    
                                    \fill[white] (0,0) arc (0:180:2 and 0.5)
                                    .. controls ++(0,-1) and ++(1,2) .. (-5,-4)
                                    -- (1,-4)
                                    .. controls ++(-1,2) and ++(0,-1) .. (0,0);
                                    
                                    \draw[red] (0,0) arc (0:180:2 and 0.5) coordinate (p);
                                    \draw (p) .. controls ++(0,-1) and ++(1,2) .. (-5,-4)
                                        node[pos=0.5, anchor=center] (upl) {}
                                        node[pos=0.85, anchor=center] (\n_qupl) {}
                                        node[pos=0.99, anchor=center] (\n_qupl_purp) {};
                                    \draw (1,-4) .. controls ++(-1,2) and ++(0,-1) .. (0,0)
                                        node[pos=0.5, anchor=center] (upr) {}
                                        node[pos=0.15, anchor=center] (\n_qupr) {}
                                        node[pos=0.01, anchor=center] (\n_qupr_purp) {};
                                    \draw[red] (0,0) arc (0:-180:2 and 0.5);
                                    
                                    \ifthenelse{\i=2}
                                    {
                                    }{
                                        \draw[dashed] (\n_qupl.center) .. controls ++(0,4/3*0.5) and ++(0,4/3*0.5) .. (\n_qupr.center);
                                    }
                                \end{tikzpicture}
                            };
                            \node[anchor=center] (\n_out) at ($(\n.north) + (\a-90:0.5)$) {};
                            \node[anchor=center, rotate=\a] at ($(\n.north) + (\a+90:0.25)$) {\Large $X_{\i}$};
                        };
                        
                        \node[anchor=center] (LH_m_bot) at ($(LH.south) + (0,-4.35)$) {};
                        \node[anchor=center] (RH_m_bot) at ($(RH.south) + (0,-4.35)$) {};
                        
                        \node[anchor=center] (BL_blac_l) at ($(BL.south west) + (-0.05,0.5)$){};
                        \node[anchor=center] (BL_blac_r) at ($(BL.south east) + (-0.5,0.35)$){};
                
                        \node[anchor=center] (BR_blac_l) at ($(BR.south west) + (0.35,0.25)$){};
                        \node[anchor=center] (BR_blac_r) at ($(BR.south east) + (0.05,0.525)$){};
                        
                        \foreach \n/\dx/\dy in 
                        {
                            LH/-0.5/0, 
                            RH/0.5/0}
                        {
                            \draw[dashed] (\n_m.center) .. controls ++(\dx,\dy) and ++(\dx,-\dy) .. (\n_m_bot);
                        };

                        \node[fill=white,
                              circle, 
                              anchor=center, 
                              inner sep=0, 
                              minimum size=15mm, 
                              draw] 
                              (main) at (0,1.5) {\LARGE $\Upsilon$};
                          
                        \foreach \i/\ang/\ax/\ay/\bx/\by/\endpoint/\arr in 
                        {
                        1/22.5/9/5/3/-8/main.326.25/>,
                        5/112.5/-0.1/4/0.1/-4/BM_out.center/<,
                        9/202.5/-9/5/-3/-8/main.236.75/<,
                        10/225/-1/-0.5/1/-0.5/LH_m.center/>,
                        11/270/-0.5/-3/1/1/LH_m_bot.center/<,
                        12/295.5/0.5/-3/-1/1/RH_m_bot.center/<,
                        13/0/1/-0.5/-1/-0.5/RH_m.center/>
                        }
                        {
                            \draw[mid={\arr}{0.5}] (main.\ang-11.25) .. controls ++(\ax,\ay) and ++(\bx,\by) .. (\endpoint);
                        }
                        
                        \node[inner sep=0, anchor=north, rotate=146.25] (RTri) at ($(main.center) + (2.5,2.5)$) {
                            \begin{tikzpicture}
                                \draw[fill=white] (-0.5*\t,0) -- (0,0.5*\t) -- (0.5*\t,0) -- cycle;
                            \end{tikzpicture}
                        };
                        \draw[mid={<}{0.5}] (main.56.25) .. controls ++(1.5,1) and ++(-0.5,-1) .. (RTri.north);
                        \draw[mid={<}{0.5}] (RTri.south) .. controls ++(0.5,0.5) and ++(-0.5,-1) .. (BR_out.center);
                        \draw[mid] ($(RTri.south west)!0.1!(RTri.south east)$) .. controls ++(0.5,0.5) and ++(-0.5,-0.5) .. (BR_blac_r);
                        \draw ($(RTri.south west)!0.9!(RTri.south east)$) .. controls ++(0,0.5) and ++(0.5,-0.5) .. (BR_blac_l);
                    
                        \node[inner sep=0, anchor=north, rotate=213.75] (LTri) at ($(main.center) + (-2.5,2.5)$) {
                            \begin{tikzpicture}
                                \draw[fill=white] (-0.5*\t,0) -- (0,0.5*\t) -- (0.5*\t,0) -- cycle;
                            \end{tikzpicture}
                        };
                        \draw[mid={<}{0.5}] (main.123.75) .. controls ++(-1.5,1) and ++(0.5,-1) .. (LTri.north);
                        \draw[mid={<}{0.5}] (LTri.south) .. controls ++(0.5,0.5) and ++(-0.5,-1) .. (BL_out.center);
                        \draw[mid] ($(LTri.south west)!0.1!(LTri.south east)$) .. controls ++(-0.25,0.5) and ++(-0.5,-0.25) .. (BL_blac_r);
                        \draw ($(LTri.south west)!0.9!(LTri.south east)$) .. controls ++(-0.5,0.25) and ++(0.5,-0.5) .. (BL_blac_l);

                    \end{tikzpicture}}}
                };
            \end{tikzpicture}
        };
        
        \draw[{|[scale=1.5]}-{>[scale=1.5]}] ($(strings.south) + (0,-0.5)$) -- ($(pic.north) + (0,0.5)$);
    \end{tikzpicture}}}
\end{figure}

and a linear map in the other direction 
\begin{align*}
\KSNrm^\Ccal(\Sigma;X_1,X_2,X_3)\longrightarrow\hom_{\Zsf_G(\Ccal)}(\onebb,c) 
\end{align*}
\begin{figure}[H]
    \centering
    \makebox[\textwidth][c]{
    \resizebox{0.6\width}{!}{
    \def\x{10}
    \def\y{1}
    \def\d{2}
    \def\h{1}
    \FPeval\s{1 / 16}
    \FPeval\t{2 * \s * \x + 0.5}
    \begin{tikzpicture}
        \node (pic) at (0,0){
            
            \begin{tikzpicture}
                \node[inner sep=0] (figure) at (0,0) {
                    \centering
                    \makebox[\textwidth][c]{
                    \resizebox{\SurfScale\width}{!}{

                    \begin{tikzpicture}
            
                        \begin{scope}[on background layer]
                            \draw  plot[smooth cycle, thick, tension=0.75] coordinates {(-6,-3) (0,-3) (6,-3) (6,0) (5,5) (2,6) (-2,6) (-5,5) (-6,0)};
                        \end{scope}
                        
                        
                        \foreach \i/\n/\x/\y in 
                            {
                            1/LH/-3/1.5, 
                            2/RH/3/1.5}
                            {
                            \node[inner sep=-0.75mm, anchor=center] (\n) at (\x,\y) {
                                \resizebox{0.75\width}{0.25\height}{
                                \begin{tikzpicture}[]
                                    
                                        \FPset\rbig{2}
                                        \FPset\rsmol{1}
                                        
                                        \FPset\initang{-30}
                                        \FPset\finang{-150}
                                        \FPeval\rang{round(abs((\initang - \finang) / 2), 0)}
                                        
                                        \FPeval\dsmol{\rsmol * (1-cos(\rang * \FPpi / 180))}
                                        \FPeval\dbig{\rbig - root(2, (\rbig^2 - \rsmol^2 * sin(\rang * \FPpi / 180 )^2))}
                                        \FPeval\hsep{(\rsmol - \dsmol) + (\rbig - \dbig)}
                                        
                                        \node (c) at (0,0) {};
                                        \centerarc[thick](c)(\initang:\finang:\rbig cm);
                                        \centerarc[thick, transform canvas={yshift=-\hsep cm}](c)(-\initang:-\finang:\rsmol cm);
                                    
                                    \end{tikzpicture}
                                    }
                            };
                            
                            \node[anchor=center] (\n_m) at (\n.south) {};
                            \node[anchor=center] (\n_r) at ($(\n.south) + (1,0.125)$) {};
                            \node[anchor=center] (\n_l) at ($(\n.south) + (-1,0.125)$) {};
                        };
                        
                        \foreach \i\n/\x/\y/\a/\sx/\sy in 
                            {
                            1/BL/-4/6/30/0.5/0.5, 
                            2/BM/0/7/0/0.5/0.5, 
                            3/BR/4/6/-30/0.5/0.5}
                            {
                            \node[anchor=center, inner sep=0, rotate=\a] (\n) at (\x,\y) {
                                \begin{tikzpicture}[xscale=\sx,yscale=\sy]
                                    
                                    
                                    \fill[white] (0,0) arc (0:180:2 and 0.5)
                                    .. controls ++(0,-1) and ++(1,2) .. (-5,-4)
                                    -- (1,-4)
                                    .. controls ++(-1,2) and ++(0,-1) .. (0,0);
                                    
                                    \draw[red] (0,0) arc (0:180:2 and 0.5) coordinate (p);
                                    \draw (p) .. controls ++(0,-1) and ++(1,2) .. (-5,-4)
                                        node[pos=0.5, anchor=center] (upl) {}
                                        node[pos=0.85, anchor=center] (\n_qupl) {}
                                        node[pos=0.99, anchor=center] (\n_qupl_purp) {};
                                    \draw (1,-4) .. controls ++(-1,2) and ++(0,-1) .. (0,0)
                                        node[pos=0.5, anchor=center] (upr) {}
                                        node[pos=0.15, anchor=center] (\n_qupr) {}
                                        node[pos=0.01, anchor=center] (\n_qupr_purp) {};
                                    \draw[red] (0,0) arc (0:-180:2 and 0.5);
                                    
                                    \ifthenelse{\i=2}
                                    {
                                    }{
                                        \draw[violet] (upl.center) .. controls ++(0,-4/3*0.5) and ++(0,-4/3*0.5) .. (upr.center)
                                            node[anchor=center, inner sep=0, minimum size=2mm, fill=white, pos=0.5] {};
                                        \draw[violet, dashed] (upl.center) .. controls ++(0,4/3*0.5) and ++(0,4/3*0.5) .. (upr.center);
                                        \draw[dashed] (\n_qupl.center) .. controls ++(0,4/3*0.5) and ++(0,4/3*0.5) .. (\n_qupr.center);
                                    }
                                    
                                \end{tikzpicture}
                            };
                            \node[anchor=center] (\n_out) at ($(\n.north) + (\a-90:0.5)$) {};
                            \node[anchor=center, rotate=\a] at ($(\n.north) + (\a+90:0.25)$) {\Large $X_{\i}$};
                        };
                        
                        \node[anchor=center] (LH_m_bot) at ($(LH.south) + (0,-4.35)$) {};
                        \node[anchor=center] (LH_r_bot) at ($(LH.south) + (1,-4.25)$) {};
                        \node[anchor=center] (LH_l_bot) at ($(LH.south) + (-1,-4.4)$) {};
                        \node[anchor=center] (RH_m_bot) at ($(RH.south) + (0,-4.35)$) {};
                        \node[anchor=center] (RH_r_bot) at ($(RH.south) + (1,-4.4)$) {};
                        \node[anchor=center] (RH_l_bot) at ($(RH.south) + (-1,-4.25)$) {};
                        
                        \node[anchor=center] (BL_purp_l) at ($(BL.south west) + (0,0.05)$){};
                        \node[anchor=center] (BL_purp_r) at ($(BL.south east) + (-0.075,0.05)$){};
                        \node[anchor=center] (BL_blac_l) at ($(BL.south west) + (-0.05,0.5)$){};
                        \node[anchor=center] (BL_blac_r) at ($(BL.south east) + (-0.5,0.35)$){};
                
                        \node[anchor=center] (BR_purp_l) at ($(BR.south west) + (0.05,0)$){};
                        \node[anchor=center] (BR_purp_r) at ($(BR.south east) + (0,0.05)$){};
                        \node[anchor=center] (BR_blac_l) at ($(BR.south west) + (0.35,0.25)$){};
                        \node[anchor=center] (BR_blac_r) at ($(BR.south east) + (0.05,0.525)$){};
                        
                        \node[anchor=center] (BM_purp_l) at ($(BM.south west) + (0.05,0.05)$){};
                        \node[anchor=center] (BM_purp_r) at ($(BM.south east) + (-0.05,0.1)$){};

                        \node[fill=white,
                              circle, 
                              anchor=center, 
                              inner sep=0, 
                              minimum size=6mm, 
                              draw] 
                              (main) at (0,1.5) {\Large $\psi$};
                          
                        \node[anchor=center, circle, draw, inner sep=0, minimum size=4mm]
                            (a1) at ($(main.center) + (20:3)$) {$\alpha$};
                        \node[anchor=center, circle, draw, inner sep=0, minimum size=4mm]
                            (a2) at ($(main.center) + (-45:3)$) {$\alpha$};
                        \draw[mid] (main.20) .. controls ++(2,0) and ++(-1,0) .. (a1.180)
                            node[anchor=north, pos=0.75] {$i_4$};
                        \draw[mid] (a1.0) .. controls ++(2,0) and ++(4,0) .. (a2.0)
                            node[anchor=west, pos=0.5] {$m_5$};
                        \draw[mid] (a2.180) .. controls ++(-1,0) and ++(0.5,0) .. (main.315);
                        
                        \node[anchor=center, circle, draw, inner sep=0, minimum size=4mm]
                            (b1) at ($(main.center) + (-15:2)$) {$\beta$};
                        \node[anchor=center, circle, draw, inner sep=0, minimum size=4mm]
                            (b2) at ($(main.center) + (-75:2)$) {$\beta$};
                        \draw[mid] (main.345) .. controls ++(1,0) and ++(-1,0) .. (b1.180)
                            node[anchor=south, pos=0.75] {$i_2$};
                        \draw[mid] (b1.340) to[bend right] (RH_m.center);
                        \draw[mid={<}{0.5}] (main.285) .. controls ++(0,-0.5) and ++(0,0.5) .. (b2.90);
                        \draw[mid] (RH_m_bot.center) .. controls ++(0,1) and ++(0,-1) .. (b2.270)
                            node[anchor=north east, pos=0.5] {$m_4$};
                        
                        \node[anchor=center, circle, draw, inner sep=0, minimum size=4mm]
                            (d1) at ($(main.center) + (35:4)$) {$\delta$};
                        \node[anchor=center, circle, draw, inner sep=0, minimum size=4mm]
                            (d2) at ($(main.center) + (55:4)$) {$\delta$};
                        \draw[mid] (main.35) .. controls ++(1,1) and ++(-1,-1) .. (d1.225);
                        \draw[mid={<}{0.5}] (main.55) .. controls ++(0,1) and ++(-1,-1) .. (d2.225)
                            node[anchor=east, pos=0.5] {$i_5$};
                        \draw[mid] (d1.60) .. controls ++(0.5,0.5) and ++(-0.75,-0.25) .. (BR_blac_r.center) 
                            node[anchor=north west, pos=0.5] {$m_6$};
                        \draw[mid={<}{0.5}] (d2.90) .. controls ++(0,0.5) and ++(0.25,-0.75) .. (BR_blac_l.center);
                        
                        \node[anchor=center, circle, draw, inner sep=0, minimum size=4mm]
                            (c1) at ($(main.center) + (160:3)$) {$\gamma$};
                        \node[anchor=center, circle, draw, inner sep=0, minimum size=4mm]
                            (c2) at ($(main.center) + (225:3)$) {$\gamma$};
                        \draw[mid={<}{0.5}] (main.160) .. controls ++(-2,0) and ++(1,0) .. (c1.0)
                            node[anchor=north, pos=0.75] {$i_3$};
                        \draw[mid={<}{0.5}] (c1.180) .. controls ++(-2,0) and ++(-4,0) .. (c2.180)
                            node[anchor=east, pos=0.5] {$m_2$};
                        \draw[mid={<}{0.5}] (c2.0) .. controls ++(1,0) and ++(-0.5,0) .. (main.225);
                        
                        \node[anchor=center, circle, draw, inner sep=0, minimum size=4mm]
                            (e1) at ($(main.center) + (195:2)$) {$\epsilon$};
                        \node[anchor=center, circle, draw, inner sep=0, minimum size=4mm]
                            (e2) at ($(main.center) + (255:2)$) {$\epsilon$};
                        \draw[mid] (main.195) .. controls ++(-1,0) and ++(1,0) .. (e1.0)
                            node[anchor=south, pos=0.75] {$i_1$};
                        \draw[mid] (e1.200) to[bend left] (LH_m.center);
                        \draw[mid={<}{0.5}] (main.255) .. controls ++(0,-0.5) and ++(0,0.5) .. (e2.90);
                        \draw[mid] (LH_m_bot.center) .. controls ++(0,1) and ++(0,-1) .. (e2.270)
                            node[anchor=north west, pos=0.5] {$m_3$};
                        
                        \node[anchor=center, circle, draw, inner sep=0, minimum size=4mm]
                            (f1) at ($(main.center) + (145:4)$) {$\zeta$};
                        \node[anchor=center, circle, draw, inner sep=0, minimum size=4mm]
                            (f2) at ($(main.center) + (125:4)$) {$\zeta$};
                        \draw[mid] (main.145) .. controls ++(-1,1) and ++(1,-1) .. (f1.315);
                        \draw[mid={<}{0.5}] (main.125) .. controls ++(0,1) and ++(1,-1) .. (f2.315)
                            node[anchor=west, pos=0.5] {$i_6$};
                        \draw[mid] (f1.150) .. controls ++(-0.5,0.5) and ++(0.75,-0.25) .. (BL_blac_l.center) 
                            node[anchor=north east, pos=0.5] {$m_1$};
                        \draw[mid={<}{0.5}] (f2.90) .. controls ++(0,0.5) and ++(-0.25,-0.75) .. (BL_blac_r.center);
                        
                        \draw[violet] (b2.0) .. controls ++(0.5,0) and ++(0,-1) .. (a2.270);
                        \draw[violet] (a2.90) .. controls ++(0.1,0.75) and ++(-0.1,-0.75) .. (b1.270);
                        \draw[violet] (b1.135) .. controls ++(-0.75,0.75) and ++(0,-0.75) .. (a1.270);
                        \draw[violet] (a1.90) .. controls ++(0,0.5) and ++(0,-0.5) .. (d1.270);
                        \draw[violet] (e2.180) .. controls ++(-0.5,0) and ++(0,-1) .. (c2.270);
                        \draw[violet] (c2.90) .. controls ++(-0.1,0.75) and ++(0.1,-0.75) .. (e1.270);
                        \draw[violet] (e1.45) .. controls ++(0.75,0.75) and ++(0,-0.75) .. (c1.270);
                        \draw[violet] (c1.90) .. controls ++(0,0.5) and ++(0,-0.5) .. (f1.270);
                        \draw[violet] (f2.0) .. controls ++(2,-0.75) and ++(-2,-0.75) .. (d2.180)
                            node[anchor=center, fill=white, midway, minimum size=2mm] {}
                            node[anchor=north, yshift=-1mm, violet, pos=0.25] {$i$};
                        
                        \draw[mid={<}{0.5}] (main.45) .. controls ++(1,1.5) and ++(-1,-2) .. (BR_out.center);
                        \draw[mid={<}{0.5}] (main.135) .. controls ++(-1,1.5) and ++(1,-2) .. (BL_out.center);
                        \draw[mid={<}{0.5}] (main.90) .. controls ++(0,2) and ++(0,-2) .. (BM_out.center);
                        
                        \foreach \n/\dx/\dy in 
                        {
                            LH/-0.5/0, 
                            RH/0.5/0}
                        {
                            \draw[dashed] (\n_m.center) .. controls ++(\dx,\dy) and ++(\dx,-\dy) .. (\n_m_bot);
                        };
                
                        \foreach \ang/\ax/\ay/\bx/\by/\endpoint/\arr/\arrpos/\lab/\labanchor/\labpos in 
                        {
                        }
                        {
                        }
                        
                    \end{tikzpicture}}}
                };
                \node[anchor=south west] (sum) at (figure.north west){\Large $\displaystyle{\sum_{\substack{i \in I_{e}, \\ m_1,\dots,m_6}} \frac{d_i}{D}d_{m_1}d_{m_2}d_{m_3}d_{m_4}d_{m_5}d_{m_6}}$};
            \end{tikzpicture}
        };
        \node[anchor=north] (strings) at ($(pic.south) + (0,-\d)$){
            \begin{tikzpicture}
                \node (morphism) at (0,0) {
                    \centering
                    \makebox[\textwidth][c]{
                    \resizebox{\StringScale\width}{!}{

                    \begin{tikzpicture}
                        
                        \node[draw, anchor=center, inner sep=0, minimum height=\y cm, minimum width=\x cm] 
                            (psi) at (0,0) {\LARGE $\psi$};
                        
                        \foreach \i/\dir/\lab/\c/\mor in 
                        {
                        1/>/i_6/1/\zeta,
                        2/</X_3/0/,
                        3/</i_6/1/\zeta,
                        4/<//0/,
                        5/</i_5/1/\delta,
                        6/</X_2/0/,
                        7/>/i_5/1/\delta,
                        8/>/i_4/1/\alpha,
                        9/>/i_2/1/\beta,
                        10/</i_4/1/\alpha,
                        11/</i_2/1/\beta,
                        12/</i_1/1/\epsilon,
                        13/>/i_3/1/\gamma,
                        14/>/i_1/1/\epsilon,
                        15/</i_3/1/\gamma}
                        {
                            \FPeval\loc{\i * \s}
                            \draw[mid={\dir}{0.5}] ($(psi.north west)!\loc!(psi.north east)$) node[anchor=south west] {$\lab$} -- ++(0,\h)
                                node[anchor=center] (temp_node\i) {};
                            \ifthenelse{\c=1}
                            {
                                \node[anchor=south, draw, circle, inner sep=0, minimum size=4mm] (mid_node\i) at (temp_node\i.center) {$\mor$};
                                \draw (mid_node\i.north) -- ($(temp_node\i.center) + (0,\h)$) node[anchor=center] (temp_node\i) {};
                            }
                            {
                                \node[anchor=south, circle, inner sep=0, minimum size=4mm] (mid_node\i) at (temp_node\i.center) {};
                                \draw (temp_node\i.center) -- ++(0,\h) node[anchor=center] (temp_node\i) {};
                            }
                        };
                        
                        \foreach \i/\lab/\dir in 
                        {1/m_1/,
                        2/\varphi_{h_1}(X_3)/,
                        3/m_1/,
                        4/X_1/,
                        5/m_6/,
                        6/\varphi_{h_2}(X_2)/,
                        7/m_6/,
                        8/m_5/>,
                        9/m_4/>,
                        10/m_5/<,
                        11/m_4/>,
                        12/m_3/<,
                        13/m_1/>,
                        14/m_3/>,
                        15/m_1/<}
                        {
                            
                            \ifthenelse{\i=1 \OR \i=3 \OR \i=5 \OR \i=7}
                            {
                                \node[anchor=north west] at (temp_node\i.center) {$\lab$};
                            }
                            {
                                \ifthenelse{\i=2 \OR \i=6}
                                {
                                    \node[anchor=south, inner sep=0] (tri\i) at (temp_node\i) {
                                        \begin{tikzpicture}
                                            \draw (-0.5*\t,0) -- (0,0.5*\t) -- (0.5*\t,0) -- cycle;
                                        \end{tikzpicture}
                                    };
                                    \draw[mid={<}{0.5}] (tri\i.north) -- ($(temp_node\i.center) + (0,2*\h)$)
                                        node[anchor=south] {$\lab$};;
                                }
                                {
                                    \draw[mid={\dir}{0.5}] (temp_node\i.center) -- ++(0,2*\h) 
                                        node[anchor=south] {$\lab$};
                                }
                            }
                        };
                        \draw[violet] (mid_node1.west) arc (90:270:1.5*\y) coordinate (p);
                        \draw[violet] (mid_node15.east) arc (90:-90:1.5*\y) coordinate (q);
                        \draw[violet, mid] (p) -- (q) node[anchor=north, midway, violet] {$i$};
                        
                        \foreach \i in {3,4,7,8,9,10,11,12,13,14}{
                            \FPeval\j{round((\i + 1),0)}
                            \draw[violet] (mid_node\i.east) -- (mid_node\j.west);
                        };
                    \end{tikzpicture}}}
                };
                \node[anchor=south west] (sum_r) at (morphism.north west) {\Large $\displaystyle{\sum_{\substack{i \in I_{e}, \\ m_1,\dots,m_6}} \frac{d_i}{D}d_{m_1}d_{m_2}d_{m_3}d_{m_4}d_{m_5}d_{m_6}}$};
                
            \end{tikzpicture}
            
        };
        \draw[{|[scale=1.5]}-{>[scale=1.5]}] ($(pic.south) + (0,-0.5)$) -- ($(strings.north) + (0,0.5)$);
    \end{tikzpicture}}}
\end{figure}

Using the same arguments as before it is now easy to show that the two maps are inverse to each other. 

The above computation can be summarized in the following proposition.
\begin{prop} For $\Sigma$ a genus two surface with three boundary components and $c\in \Zsf_G(\Ccal)$ as in \eqref{c defi}, it holds
\eq{\label{genus 2 vector space}
\KSNrm^\Ccal(\Sigma;X_1,X_2,X_3)\simeq \hom_{\Zsf_G(\Ccal)}(\onebb,c)\quad .
}
\end{prop}

\begin{rem} The string-net space $\KSNrm^\Ccal(\Sigma; X_1,X_2,X_3)$ has the same structure as the vector spaces obtained in a different construction in \cite{turaev20203}. The different geometric input data in the two constructions make a detailed comparison of the vector spaces uninstructive.
\end{rem} 

\newpage 

\appendix

\section{Proof of Theorem \ref{G scc}}
The Theorem will follow from \cite[Proposition~6.2]{bakalov2000lego}, which we quickly state for the reader's convenience.

\begin{theo}\label{BakalovKirillov} Given two $2$-dimensional CW-complexes $F$, $B$ and a map $\pi:F^{[1]}\rightarrow B^{[1]}$ of their $1$-skeletons, 
	such that.:
\begin{enumerate}[label=\roman*)]
\item For any vertex $v$ and edge $e$ in $B$, there exists a vertex $v^\prime$ and an edge $e^\prime$ with $\pi(v^\prime)=v$ and $\pi(e^\prime)=e$. Furthermore, for $v_1\xrightarrow{e}v_2$ an edge in $B$, and any $v_1^\prime\in \pi^{-1}(v_1)$, there exists an edge $v_1^\prime\xrightarrow{e^\prime}v_2^\prime$ such that $\pi(e^\prime)=e$.
\item $B$ is connected and simply connected.
\item There exist a vertex $v\in B$, such that $\pi^{-1}(v)$ is connected and simply connected.
\item Given an edge $v_1\xrightarrow{e}v_2$ in $B$, an edge $v_1^\prime\xrightarrow{f_1}v_1^{\prime\prime}$ in $\pi^{-1}(v_1)$ and two lifts $v_1^\prime\xrightarrow{e^\prime}v_2^\prime$, $v_1^{\prime\prime}\xrightarrow{e^{\prime\prime}}v_2^{\prime\prime}$ of $e$, there exists an edge $v_2^\prime\xrightarrow{f_2}v_2^{\prime\prime}$ in $\pi^{-1}(v_2)$ such that the square
\begin{center}
\begin{tikzcd}
v_1^\prime\ar[r,"e^\prime"]\ar[d,"f_1"'] & v_2^\prime\ar[d,"f_2"]\\
v_1^{\prime\prime}\ar[r,"e^{\prime\prime}"] & v_2^{\prime\prime}
\end{tikzcd}
\end{center}
is contractible in $F$.
\item The boundary $\p C$ of any $2$-cell in $B$ has a contractible lift in $B$.
\end{enumerate}
Then the complex $F$ is connected and simply connected.
\end{theo}

\begin{proof}[Proof of Theorem \ref{G scc}] The proof consists of checking points i)-v) of Theorem \ref{BakalovKirillov} for the forgetful map $\pi:\Pcal^G(\Sigma,\zeta,M)^{[1]}\rightarrow\Pcal(\Sigma,M)^{[1]}$, which forgets the $G$-labels of the edges, maps $G$-flips to flips of the underlying ideal triangulation and maps gauge transformations to the identity. Note that the fiber of $\pi$ over any ideal triangulation $\Tsf$ is the set of all possible $G$-labels of $\Delta$ induced by $\zeta:\Sigma\rightarrow BG$.
\begin{enumerate}[label=\roman*)]
\item The forgetful map is clearly surjective on vertices and edges. In addition, given a flip $\Delta\xrightarrow{F_e}\Delta^\prime$ and any $G$-triangulation $(\Delta,g)\in \pi^{-1}(\Delta)$, there is a $G$-flip $(\Delta,g)\xrightarrow{F^G_\ebf}(\Delta^\prime,g^\prime)$ covering $e$.
\item The Ptolemy-complex is connected and simply connected by theorem \ref{Ptolemy c sc}.
\item Let $\Delta$ be any ideal triangulation of $\Sigma$. The fiber $\pi^{-1}(\Delta)$ is connected, as the gauge group $\Gcal_\Delta$ relates any two $G$-labelings induced by $\zeta$. Due to relation \textbf{GP5} any sequence of gauge transformations in $\pi^{-1}(\Delta)$ is homotopic to a sequence of the form $\lambda_1(v)\cdots\lambda_{n_v}(v)\lambda_1(w)\cdots\lambda_{n_w}(w)\lambda_1(u)\cdots $ running trough all vertices of $\Delta$ in arbitrary order. Furthermore, any two paths having the same start- and endpoint in $\pi^{-1}(\Delta)$ and consisting entirely of different gauge transformations at the same vertex are homotopic. This follows from the observation, that due to \textbf{GP6} the subcomplex spanned by gauge transformations at a single vertex of a $G$-triangulation is the $2$-skeleton of a $|G|-1$-dimensional simplex and therefore is simply connected. It follows that $\pi^{-1}(\Delta)$ is simply connected.
\item This directly follows from the mixed relation \textbf{GP4}.
\item Similar to the previous point, this holds, since the boundaries of the $2$-cells \textbf{GP1}, \textbf{GP2} and \textbf{GP3} project to the boundaries of the $2$-cells \textbf{P1}, \textbf{P2} and \textbf{P3}.
\end{enumerate} 

\end{proof}

\bibliographystyle{alpha}
\bibliography{references}

\begin{thebibliography}{BDSPV15}

\bibitem[Bal10a]{balsam}
Benjamin Balsam.
\newblock {Turaev-Viro invariants as an extended TQFT II}.
\newblock {\em arXiv:1010.1222}, 2010.

\bibitem[Bal10b]{balsam2010turaevviro}
Benjamin Balsam.
\newblock {Turaev-Viro invariants as an extended TQFT III}.
\newblock {\em arXiv:1012.0560}, 2010.

\bibitem[Bar22]{bartlett2022three}
Bruce Bartlett.
\newblock Three-dimensional tqfts via string-nets and two-dimensional surgery.
\newblock {\em arXiv preprint arXiv:2206.13262}, 2022.

\bibitem[BDSPV15]{bartlett2015modular}
Bruce Bartlett, Christopher~L Douglas, Christopher~J Schommer-Pries, and Jamie
  Vicary.
\newblock Modular categories as representations of the 3-dimensional bordism
  2-category.
\newblock {\em arXiv preprint arXiv:1509.06811}, 2015.

\bibitem[{Ben}10]{KirillovBalsam}
{Benjamin Balsam and Alexander Kirillov Jr.}
\newblock {Turaev-Viro invariants as an extended TQFT}.
\newblock {\em arXiv:1004.1533}, 2010.

\bibitem[BK00]{bakalov2000lego}
Bojko Bakalov and Alexander Kirillov.
\newblock {On the lego-Teichm{\"u}ller game}.
\newblock {\em Transformation groups}, 5(3):207--244, 2000.

\bibitem[Bor94]{borceux_1994}
Francis Borceux.
\newblock {\em Handbook of Categorical Algebra}, volume~1 of {\em Encyclopedia
  of Mathematics and its Applications}.
\newblock Cambridge University Press, 1994.

\bibitem[EGNO10]{tensorCats}
Pavel Etingof, Shlomo Gelaki, Dimitri Nikshych, and Victor Ostrik.
\newblock {\em {Tensor Categories}}, volume 205 of {\em {Mathematical Surveys
  and Monographs}}.
\newblock {American Mathematical Society}, 2010.

\bibitem[FFRS06]{Fjelstad:2005ua}
Jens Fjelstad, J{\"u}rgen Fuchs, Ingo Runkel, and Christoph Schweigert.
\newblock {TFT construction of RCFT correlators. V. Proof of modular invariance
  and factorisation}.
\newblock {\em Theor. Appl. Categor.}, 16:342--433, 2006.

\bibitem[FGK88]{felder1988spectra}
Giovanni Felder, Krzysztof Gawedzki, and Antti Kupiainen.
\newblock {Spectra of Wess-Zumino-Witten models with arbitrary simple groups}.
\newblock {\em Communications in mathematical physics}, 117(1):127--158, 1988.

\bibitem[FRS02]{Fuchs:2002cm}
J{\"u}rgen Fuchs, Ingo Runkel, and Christoph Schweigert.
\newblock {TFT construction of RCFT correlators I. Partition functions}.
\newblock {\em Nucl. Phys. B}, 646:353--497, 2002.

\bibitem[FRS04a]{Fuchs:2003id}
J{\"u}rgen Fuchs, Ingo Runkel, and Christoph Schweigert.
\newblock {TFT construction of RCFT correlators. II. Unoriented world sheets}.
\newblock {\em Nucl. Phys. B}, 678:511--637, 2004.

\bibitem[FRS04b]{Fuchs:2004dz}
J{\"u}rgen Fuchs, Ingo Runkel, and Christoph Schweigert.
\newblock {TFT construction of RCFT correlators. III. Simple currents}.
\newblock {\em Nucl. Phys. B}, 694:277--353, 2004.

\bibitem[FRS05]{Fuchs:2004xi}
J{\"u}rgen Fuchs, Ingo Runkel, and Christoph Schweigert.
\newblock {TFT construction of RCFT correlators IV: Structure constants and
  correlation functions}.
\newblock {\em Nucl. Phys. B}, 715:539--638, 2005.

\bibitem[FSY21]{fuchs2021string}
J{\"u}rgen Fuchs, Christoph Schweigert, and Yang Yang.
\newblock {String-net construction of RCFT correlators}.
\newblock {\em arXiv preprint arXiv:2112.12708}, 2021.

\bibitem[GNN09]{gelaki2009centers}
Shlomo Gelaki, Deepak Naidu, and Dmitri Nikshych.
\newblock Centers of graded fusion categories.
\newblock {\em Algebra \& Number Theory}, 3(8):959--990, 2009.

\bibitem[{Goo}18]{Goosen}
{Goosen, Gerrit}.
\newblock {Oriented 123-TQFTs via String-Nets and State-Sums}.
\newblock {\em {PhD Thesis, Stellenbosch University}}, 2018.

\bibitem[GR04]{gawkedzki2004basic}
Krzysztof Gawedzki and Nuno Reis.
\newblock Basic gerbe over non-simply connected compact groups.
\newblock {\em Journal of Geometry and Physics}, 50(1-4):28--55, 2004.

\bibitem[GSW11]{gawkedzki2011bundle}
Krzysztof Gawedzki, Rafal~R Suszek, and Konrad Waldorf.
\newblock Bundle gerbes for orientifold sigma models.
\newblock {\em Advances in Theoretical and Mathematical Physics},
  15(3):621--687, 2011.

\bibitem[HBFL16]{heinrich2016symmetry}
Chris Heinrich, Fiona Burnell, Lukasz Fidkowski, and Michael Levin.
\newblock Symmetry-enriched string nets: Exactly solvable models for set
  phases.
\newblock {\em Physical Review B}, 94(23):235136, 2016.

\bibitem[KJ11]{kirillov2011string}
Alexander Kirillov~Jr.
\newblock {String-net model of Turaev-Viro invariants}.
\newblock {\em arXiv preprint arXiv:1106.6033}, 2011.

\bibitem[KJB10]{BalsamKirillov}
Alexander Kirillov~Jr and Benjamin Balsam.
\newblock {Turaev-Viro invariants as an extended TQFT}.
\newblock {\em arXiv preprint arXiv:1004.1533}, 2010.

\bibitem[LW05]{Levin:2004mi}
Michael~A. Levin and Xiao-Gang Wen.
\newblock {String net condensation: A Physical mechanism for topological
  phases}.
\newblock {\em Phys. Rev. B}, 71:045110, 2005.

\bibitem[NS07]{ng2007higher}
Siu-Hung Ng and Peter Schauenburg.
\newblock Higher frobenius-schur indicators for pivotal categories.
\newblock {\em Hopf algebras and generalizations}, 441:63--90, 2007.

\bibitem[Pen12]{PennerBook}
Robert~C. Penner.
\newblock {\em {Decorated Teichm\"uller Theory}}.
\newblock The QGM Master Class Series. European Mathematical Society Publishing
  House, 2012.

\bibitem[SW20]{schweigert2020extended}
Christoph Schweigert and Lukas Woike.
\newblock Extended homotopy quantum field theories and their orbifoldization.
\newblock {\em Journal of Pure and Applied Algebra}, 224(4):106213, 2020.

\bibitem[SY21]{Schweigert:2019zwt}
Christoph Schweigert and Yang Yang.
\newblock {CFT Correlators for Cardy Bulk Fields via String-Net Models}.
\newblock {\em Symmetry, Integrability and Geometry: Methods and Applications},
  Apr 2021.

\bibitem[Tra22]{traube2022cardy}
Matthias Traube.
\newblock Cardy algebras, sewing constraints and string-nets.
\newblock {\em Communications in Mathematical Physics}, 390:1--45, 2022.

\bibitem[Tur10]{TuraevHQFTbook}
Vladimir Turaev.
\newblock {\em {Homotopy Quantum Field Theory}}, volume~10 of {\em EMS Tracts
  in Mathematics}.
\newblock European Mathematical Society Publishing House, 2010.

\bibitem[TV12]{turaev20123}
Vladimir Turaev and Alexis Virelizier.
\newblock {On 3-dimensional homotopy quantum field theory, I}.
\newblock {\em International Journal of Mathematics}, 23(09):1250094, 2012.

\bibitem[TV13]{turaev2013graded}
Vladimir Turaev and Alexis Virelizier.
\newblock {On the graded center of graded categories}.
\newblock {\em Journal of Pure and Applied Algebra}, 217(10):1895--1941, 2013.

\bibitem[TV20]{turaev20203}
Vladimir Turaev and Alexis Virelizier.
\newblock {On 3-dimensional homotopy quantum field theory III: comparison of
  two approaches}.
\newblock {\em International Journal of Mathematics}, 31(10):2050076, 2020.

\bibitem[Wal06]{Walker}
Kevin Walker.
\newblock {TQFTs}.
\newblock {\em available at https://canyon23.net/math/}, 2006.

\end{thebibliography}
\end{document}